\documentclass[a4paper, titlepage, oneside, 11pt]{amsart}
\usepackage{amsmath,amssymb,amsfonts,amstext,amscd,cite,geometry,graphicx,latexsym,mathptmx,xcolor,pstricks}

\usepackage{ascmac}
\usepackage{amsmath}
\usepackage[utf8]{inputenc}
\usepackage[all]{xy}
\usepackage{color}

\newtheorem{lem}{Lemma}[section]
\newtheorem{prop}[lem]{Proposition}
\newtheorem{thm}[lem]{Theorem}

\newtheorem{rem}[lem]{Remark}

\renewcommand{\a}{\alpha}
\renewcommand{\b}{\beta}

\newcommand{\ga}{\gamma}

\newcommand{\la}{\lambda}

\newcommand{\E}{\mathbf{E}}

\def\R{{\mathbb{R}}}
\def\N{{\mathbb{N}}}
\def\Z{{\mathbb{Z}}}
\def\T{{\mathbb{T}}}

\def\M{{\mathcal{M}}}

\def\A{{\mathcal{A}}}
\def\E{{\mathcal{E}}}

\newcommand{\vertiii}[1]{{\left\vert\kern-0.25ex\left\vert\kern-0.25ex\left\vert #1
    \right\vert\kern-0.25ex\right\vert\kern-0.25ex\right\vert}}

\numberwithin{equation}{section}

\allowdisplaybreaks

\begin{document}

\title[Thermal conductivity in a magnetic field]{Thermal conductivity for coupled charged harmonic oscillators with noise in a magnetic field }

\author[K.~Saito]{Keiji Saito}
\address{Department of Physics, Faculty of Science and Technology, Keio University, 3-14-1, Hiyoshi, Kohoku-ku, Yokohama, Kanagawa, 223-8522, Japan}
\email{saitoh@rk.phys.keio.ac.jp}

\author[M.~Sasada]{Makiko Sasada}
\address{Graduate School of Mathematical Sciences, University of Tokyo, 3-8-1, Komaba, Meguro-ku, Tokyo, 153--8914, Japan}
\email{sasada@ms.u-tokyo.ac.jp}

\begin{abstract}
We introduce a $d$-dimensional system of charged harmonic oscillators in a magnetic field perturbed by a stochastic dynamics which conserves energy but not momentum. We study the thermal conductivity via the Green--Kubo formula, focusing on the asymptotic behavior of the Green--Kubo integral up to time $t$ (i.e., the integral of the correlation function of the total energy current). We employ the microcanonical measure to calculate the Green--Kubo formula in general dimension $d$ for uniformly charged oscillators. 
We also develop a method to calculate the Green--Kubo formula with the canonical measure for uniformly and alternately charged oscillators in dimension $1$. We prove that the
thermal conductivity diverges in dimension $1$ and $2$ while it remains finite in dimension $3$. 
The Green--Kubo integral calculated with the microcanonical ensemble diverges 
as $t^{1/4}$ for uniformly charged oscillators in dimension $1$,
while it is known to diverge as $t^{1/2}$ without magnetic field. This is the first rigorous example of the new exponent 
$1/4$ in the asymptotic behavior for the Green--Kubo integral.
 We also demonstrate that our result provides the first rigorous example of a diverging thermal conductivity with vanishing sound speed. 
In addition, employing the canonical measure in the Green--Kubo formula, we prove 
 that the Green--Kubo integral for uniformly and alternately charged oscillators respectively diverges as $t^{1/4}$ and $t^{1/2}$.
This means that the exponent depends not only on a non-zero magnetic field but also on the charge structure of oscillators.
\end{abstract}


%\subjclass[2010]{37B15 (primary), 60G50, 60J10, 60J65, 82B99 (secondary)}

\date{\today}

\maketitle

\setcounter{tocdepth}{3}
\tableofcontents
\newpage

\section{Introduction}

\subsection{Background and summary}
In low dimensions, the thermal conductivity generally diverges 
in the thermodynamic limit. 
This phenomenon is called anomalous heat transport or superdiffusion of energy, a research area of which has now become an interdisciplinary field from physics to mathematics \cite{L,BBJKO,Sp,LLP03,dhar08}. 

Diverging thermal conductivity and 
%{\color{red}macroscopic} 
superdiffusion of energy have been studied for prototypical interacting systems
with several conservation laws \cite{SS,BS,Popkov,BGJ,BG}. Among them, the system of oscillators is one of the most studied models.
Using these models, 
the mechanism of diverging thermal conductivity has been studied mainly through the Green--Kubo formula. Considering the $1$-dimensional periodic chain of $N$ oscillators, we define the local energy $\E_x $ of the oscillator labeled by $x\in \T_N$, where $\T_N=\Z/N\Z$ is the discrete torus. From the conservation of total energy, the local energy current is defined through the continuity equation with respect to energy as $\partial_t \E_x (t)= J_{x-1,x} (t) - J_{x,x+1} (t) $, where $J_{x,x+1}$ is the energy current between the sites $x$ and $x+1$. 
Then, up to a conventional normalization, the Green--Kubo formula for thermal conductivity $\kappa$ is written as 
\begin{align*}
\kappa &= \lim_{t\to\infty} \kappa_{GK} (t) \,,\\
\kappa_{GK} (t) &= \lim_{N\to\infty}\sum_{x\in \T_N} \int_{0}^{t} ds \langle J_{x,x+1} (s)J_{0,1} (0) \rangle  \,,   
\end{align*}
where the thermal conductivity $\kappa$ is obtained by taking the infinite time limit in the Green--Kubo integral $\kappa_{GK}(t)$.
The symbol $\langle...\rangle$ implies the average over the microcanonical or canonical ensemble.
Note that the Green--Kubo integral $\kappa_{GK}(t)$ contains the correlation function of the instantaneous total energy current.
In general low-dimensional systems, the correlation function shows a power-law decay in the long-time regime, and this leads to the divergence of the Green--Kubo integral.

So far,
there have been many discussions on the divergence of the Green--Kubo formula in terms of the asymptotic behavior of the Green--Kubo integral and the current correlation function with numerical and analytical approaches \cite{L}. Recently, many theoretical advances have been achieved. The key-ingredient of the theories is to focus on the conserved quantities.
In any nonlinear chain of oscillators connected with spring forces depending only on the distance between particles, the momentum and energy are conserved. In addition, the so-called stretch is also conserved. Renormalization type arguments for the hydrodynamic equations of the conserved quantities are performed and they predict a $t^{1/3}$ asymptotic behavior of the Green--Kubo integral in $1$-dimensional systems \cite{NR}. More recently, nonlinear fluctuating hydrodynamics argues for a remarkable connection between the dynamics of the sound modes in heat conduction and the Kardar--Parisi--Zhang equation \cite{Sp}. The asymptotic divergent behavior of the Green--Kubo integral is classified into $t^{1/3}$ or $t^{1/2}$ depending on the symmetry of the potential function and the pressure \cite{Sp}. The nonlinear fluctuating hydrodynamics also classifies the possible universality class of the dynamical exponent in general diffusive dynamics with several conservation laws, and a Fibonacci sequence of exponents, including the golden mean, is found \cite{Popkov}.

Another important approach to understand the diverging thermal conductivity relies on an exactly soluble models, among of them the momentum exchange model \cite{BBO}. This model shows a hybrid dynamics consisting of a deterministic part and stochastic perturbation.
In the momentum exchange model, the momenta of the nearest neighbor oscillators are stochastically exchanged so that the conserved quantities defined for deterministic dynamics remain conserved. 
A rigorous analysis shows the $s^{-1/2}$ power-law decay in the correlation of the total current if the chain is harmonic \cite{BBO}, and this leads to a $t^{1/2}$ divergence in the Green--Kubo integral.
For this harmonic momentum exchange model, it is also established that under a proper space--time scaling limit, the macroscopic energy diffusion is described by a fractional diffusion equation \cite{JKO}. Recently variants of the momentum exchange model have been proposed to simplify the analysis \cite{BS,BGJ}.

In this paper, we propose another variant of the harmonic momentum exchange model, where the oscillators are charged and a magnetic field is applied. We will show a new exponent in the divergence of the Green--Kubo integral. The magnetic field induces cyclotron motion of the particle and hence the standard momentum conservation is broken.
This model was partially analyzed in the previous paper \cite{TSS}.
Herein we generalize the previous work and develop several novel mathematical techniques
We consider a system where the charged harmonic oscillators are arranged on the $d$-dimensional lattice with periodic boundary conditions and each oscillator can move in the $d^*$-dimensional space. We assume $d^* \ge 2$ so to be able to apply a magnetic field, but $d\, (\ge 1)$ is arbitrary. 
The goal of our study is to derive the asymptotic behavior of the Green--Kubo integral for this model. To this end, we calculate the current correlation function with microcanonical and canonical measures in different lattice dimension $d$ and charge structures of oscillators.

In Section \ref{microsec}, we study the \
thermal conductivity via the \textit{microcanonical} Green--Kubo formula in general dimension $d$ for uniformly charged oscillators, namely the case where the charges of all oscillators are equal. Here, the microcanonical Green--Kubo formula implies that we employ the microcanonical ensemble for the correlation function of energy current in the Green--Kubo integral. We show that the thermal conductivity diverges in lattice dimensions $1$ and $2$, while it is finite in $d \ge 3$. 
%KS
%{\color{red}In contrast to the harmonic momentum exchange model, our system breaks the conservation of total momentum. For some variants of the momentum exchange model with noise breaking the conservation of total momentum, the divergence of the thermal conductivity has been also shown \cite{BG,BGJ,BS}. However, our noise conserves the total momentum, and the deterministic part breaks the conservation instead. We prove the divergence of the thermal conductivity for this type of dynamics for the first time.}
In dimension $1$, the microcanonical Green--Kubo integral diverges as $t^{1/4}$.
This is a remarkable difference from the case with zero magnetic field having $t^{1/2}$ behavior as reported in \cite{BBO}.
This is the first rigorous example of the exponent 
  $1/4$ in the divergent behavior of the Green--Kubo integral for the thermal conductivity.

Further, in our model, the sound velocity, namely the first derivative of the dispersion relation at wave number $0$, vanishes. In \cite{KO}, the relation between the vanishing sound speed and the divergence of the thermal conductivity is discussed. In particular, the authors conjectured that the non-vanishing speed of sound is a necessary condition for the superdiffusion of thermal energy. However, our result provides a counterexample to this conjecture. 

Our strategy for the proof is primarily the same as in \cite{BBO}, where an explicit solution of a resolvent equation is crucial. With this solution, we calculated the Laplace transform of the correlation function of the energy current explicitly. In the study of its inverse Laplace transform, numerous computations with careful asymptotic analysis are required. Here, we remark that the correlation function has a strong oscillation and it does not have a monotone decay as $s^{-3/4}$.

In Section \ref{canosec}, we develop a technique to calculate the Green--Kubo integral with the canonical measure in lattice dimension $1$.
To introduce the canonical measure, the coordinate of the system must be changed from the pair of position and velocity to the pair of deformation (the difference in position between two neighboring oscillators) and velocity. We use this new coordinate
%this method 
for uniformly and alternately charged oscillators in dimension $1$. Here, the alternate charge 
means that the sign of charges is alternating, while the magnitude is the same for all oscillators.
 In addition, we apply this change of coordinates
%this method 
to a system of uncharged oscillators (which is equivalent to the original harmonic momentum exchange model), because the canonical measures are not properly defined in \cite{BBO}. For the canonical Green--Kubo integral, we again find that the uniformly charged system asymptotically shows $t^{1/4}$. However, uncharged and alternately charged systems show $t^{1/2}$. This reveals that the exponent depends not only on a non-zero magnetic field, which breaks momentum conservation, but also on the charge structure of the oscillators. Heuristically the average effect from the magnetic field in the alternate charge model can be regarded as $0$. However, we do not have any rigorous argument for the reason why uncharged and alternately charged systems show the same exponent thus far.

Concerning the proof, unlike the microcanonical case, we cannot solve the resolvent equation directly because of the change of coordinates. Several nice symmetries exist for position coordinates, which are essential to solve the resolvent equation explicitly. 
Hence, we introduce a new technique by using a solution in the original position coordinates. 
On the other hand, as the canonical measures are product measures in the new coordinates of deformations,
% (for $d=1$), 
the expectation of observables are easy to compute and the equivalence of ensembles is not invoked. 
%On the other hand, 
%KS
%Using the technique we developed in this paper, we can take advantages of both.

Once the divergence of the thermal conductivity is shown, it is natural to inquire about the exact nature of the superdiffusion, or more precisely about the macroscopic evolution equation of energy. We are currently working on this problem using the kinetic approach and expect a corresponding fractional diffusion equation.

\smallskip

\subsection{Model and main result}

We now introduce our model. We consider a $d$-dimensional chain of oscillators with periodic boundary conditions of length $N$ moving in the $d^*$-dimensional space. The oscillators are labeled by $\mathbf{x} \in \Z^d_N:=\Z^d/ N\Z^d$. We emphasize that $d$ is not necessarily equal to $d^*$. In particular, as we aim to study the system in a magnetic field, we always assume $d^* \ge 2$, but $d \ge 1$ is arbitrary. The velocity of the oscillator $\mathbf{x}$ is denoted by $\mathbf{v}_{\mathbf{x}} \in \R^{d^*}$ and the displacement from its equilibrium position is denoted by $\mathbf{q}_{\mathbf{x}} \in \R^{d^*}$. 

We denote by $(\mathbf{e}_1,\mathbf{e}_2,\dots,\mathbf{e}_n)$ the canonical basis of $\R^n$ and the coordinates of a vector $\mathbf{u} \in \R^n$ in this basis by $(u^1,\dots,u^n)$. Its Euclidean norm $|\mathbf{u}|$ is defined by $|\mathbf{u}|=\sqrt{(u^1)^2+\cdots+(u^n)^2}$ and the scalar product of $\mathbf{u}$ and $\mathbf{v}$ is denoted by $\mathbf{u} \cdot \mathbf{v}$. When it is clear from the context, we will not explicitly mention whether a vector is in $\R^d$ or $\R^{d^*}$. Instead, for better readability, we use the alphabets $a,b$ for indexes in $\{1,2,\dots,d\}$ and $j,k,\ell,m$ for indexes in $\{1,2,\dots,d^*\}$. For example, $\mathbf{x}=(x^{a})_{a}$ for $\mathbf{x} \in \Z^d_N$ and $ \mathbf{v}_{\mathbf{x}} =(v_{\mathbf{x}}^j)_{j}$ for $\mathbf{v}_{\mathbf{x}} \in \R^{d^*}$. 

For two functions $f(t)$ and $g(t)$ defined on $[0,\infty)$, we denote by $f(t) \sim g(t) ( t \to \infty)$ if there exists a constant $C>0$ such that for large enough $t$, $\frac{1}{C}g(t) \le f(t) \le Cg(t)$.

For $F:\Z^d_N \to \R$ or $F:\Z^d \to \R$, we introduce the discrete gradient of $F$ in the direction $\mathbf{e}_{a}$ defined by 
\[(\nabla_{\mathbf{e}_{a}} F)(\mathbf{x})=F(\mathbf{x}+\mathbf{e}_{a})-F(\mathbf{x})\]
and the discrete Laplacian of $F$ defined by 
\[(\Delta F) (\mathbf{x})=\sum_{|\mathbf{y}-\mathbf{x}|=1}(F(\mathbf{y})-F(\mathbf{x}))=\sum_{a=1}^d\{ F(\mathbf{x}+\mathbf{e}_{a})+F(\mathbf{x}-\mathbf{e}_{a})-2F(\mathbf{x})\}.\]

The dynamics of the chain of oscillators in a magnetic field is given by the following differential equations
\begin{align}\label{eq:dynamics} 
\begin{cases}
\frac{d}{dt} \mathbf{q}_{\mathbf{x}}  & =\mathbf{v}_{\mathbf{x}} \\
\frac{d}{dt} \mathbf{v}_{\mathbf{x}} & =[\Delta \mathbf{q}]_{\mathbf{x}} -B\sigma  \mathbf{v}_{\mathbf{x}}
\end{cases}
\end{align}
for $\mathbf{x} \in \Z_N^d$ where $\sigma=(\sigma_{j,k})$ is $d^* \times d^*$ matrix with entries $\sigma_{1,2}=-1, \sigma_{2,1}=1$, $\sigma_{j,k}=0$ for $(j,k) \neq (1,2),(2,1)$ and $B \in \R$ is the signed strength of the magnetic field. In this model, the oscillators are uniformly charged. For simplicity, we only consider the magnetic field orthogonal to the $(1,2)$ plain herein, but we can replace the matrix $B\sigma$ to an arbitrary real skew-symmetric matrix (see Remark \ref{remskew}).

The dynamics (\ref{eq:dynamics}) is also written in the following way:
\begin{align}\label{eq:dynamics2} 
\begin{cases}
\frac{d}{dt} q_{\mathbf{x}}^j & =v_{\mathbf{x}}^j \\
\frac{d}{dt} v_{\mathbf{x}}^j & =[\Delta q^j]_{\mathbf{x}} +\delta_{j,1}Bv_{\mathbf{x}}^2-\delta_{j,2}B v_{\mathbf{x}}^1
\end{cases}
\end{align}
for $j=1,\dots, d^*$ and $\mathbf{x} \in \Z_N^d$. 

\begin{rem}\label{rem:hamiltonian}
If $B=0$, the dynamics (\ref{eq:dynamics}) is the Hamiltonian system with the following Hamiltonian
\begin{displaymath}
H_N^0=\sum_{\mathbf{x} \in \Z^d_N} \left( \frac{|\mathbf{v}_{\mathbf{x}}|^2}{2}+\sum_{|\mathbf{y}-\mathbf{x}|=1}\frac{|\mathbf{q}_{\mathbf{y}}-\mathbf{q}_{\mathbf{x}}|^2}{4} \right).
\end{displaymath}
For $B \neq 0$, the dynamics is also a Hamiltonian system with a canonical momentum 
\[
\mathbf{p}_{\mathbf{x}}:=\mathbf{v}_{\mathbf{x}}+\frac{1}{2}B\sigma\mathbf{q}_{\mathbf{x}}
\]
and the Hamiltonian 
\begin{displaymath}
H_N^B=\sum_{\mathbf{x} \in \Z^d_N} \left( \frac{|\mathbf{p}_{\mathbf{x}}-\frac{1}{2}B\sigma\mathbf{q}_{\mathbf{x}}
|^2}{2}+\sum_{|\mathbf{y}-\mathbf{x}|=1}\frac{|\mathbf{q}_{\mathbf{y}}-\mathbf{q}_{\mathbf{x}}|^2}{4}\right).
\end{displaymath}
However, we do not use the canonical momentum in this paper. 
%{\color{purple} Note that the total canonical momentum is not conserved by the dynamics (\ref{eq:dynamics}).}
\end{rem}

It is obvious that the dynamics (\ref{eq:dynamics}) conserves the total energy $\sum_{\mathbf{x} \in \Z^d_N}\E_{\mathbf{x}}$, where we define the energy of the oscillator $\mathbf{x}$ by
\[
\E_{\mathbf{x}}=\frac{|\mathbf{v}_{\mathbf{x}}|^2}{2}+\sum_{|\mathbf{y}-\mathbf{x}|=1}\frac{|\mathbf{q}_{\mathbf{y}}-\mathbf{q}_{\mathbf{x}}|^2}{4}.
\]
Whereas the total velocity (or canonical momentum) $\sum_{\mathbf{x} \in \Z^d_N}\mathbf{v}_{\mathbf{x}}$ is also conserved if $B=0$, this is not the case if $B \neq 0$. Instead, the sum of the pseudomomentum $\tilde{\mathbf{p}}_{\mathbf{x}}$ 
\[
\tilde{\mathbf{p}}_{\mathbf{x}}:=\mathbf{v}_{\mathbf{x}}+B\sigma\mathbf{q}_{\mathbf{x}}
\]
is conserved, which is generally the case for Hamiltonian systems in a magnetic field \cite{JHY}. In particular, if the dynamics starts from the initial value satisfying $\sum_{\mathbf{x} \in \Z^d_N}\mathbf{q}_{\mathbf{x}}=\sum_{\mathbf{x} \in \Z^d_N}\mathbf{v}_{\mathbf{x}}=\mathbf{0}$, then for any time $t$, 
\begin{align*}
\frac{d}{dt} \sum_{\mathbf{x} \in \Z^d_N}\mathbf{q}_{\mathbf{x}} = \sum_{\mathbf{x} \in \Z^d_N}\mathbf{v}_{\mathbf{x}}= -B\sigma\sum_{\mathbf{x} \in \Z^d_N} \mathbf{q}_{\mathbf{x}},
\end{align*}
since $\sum_{\mathbf{x} \in \Z^d_N}\tilde{\mathbf{p}}_{\mathbf{x}}=\mathbf{0}$. Hence, $\sum_{\mathbf{x} \in \Z^d_N}\mathbf{q}_{\mathbf{x}}=\sum_{\mathbf{x} \in \Z^d_N}\mathbf{v}_{\mathbf{x}}=0$ for any time $t$.
This observation shows that the following microcanonical state space
\begin{displaymath}
\Omega_{N,E}:=
\{(\mathbf{q}_{\mathbf{x}},\mathbf{v}_{\mathbf{x}})_{ \mathbf{x} \in \Z^d_N}; \sum_{\mathbf{x}}\mathbf{q}_{\mathbf{x}}=\sum_{\mathbf{x}}\mathbf{v}_{\mathbf{x}}=0, \sum_{\mathbf{x}}\E_{\mathbf{x}}=N^d E \}
\end{displaymath}
is conserved by the dynamics for each fixed averaged energy $E>0$. Note that the microcanonical state space does not depend on $B$. We present some discussions on the microcanonical state space for general initial configurations in Subsection \ref{measec1}.

Now, we add a stochastic perturbation to this deterministic model. We introduce the velocity exchange noise, i.e., each pair of nearest neighbor oscillators exchanges the $j$-th component of their velocities at random exponential times with intensity $\gamma >0$ for each $j \in \{1,2,\dots,d^*\}$. $\gamma$ represents the strength of the noise. This stochastic perturbation is given by the operator $\ga S$ acting on functions $f : \R^{2d^*N} \to \R$ as
\begin{align}\label{Sop}
Sf = & \sum_{\mathbf{x} \in \Z^d_N}\sum_{a=1}^d\sum_{j=1}^{d^*}(f(\mathbf{q},\mathbf{v}^{j,\mathbf{x},\mathbf{x}+\mathbf{e}_a})-f(\mathbf{q},\mathbf{v})). 
%S^{alt} f = & \sum_{x \in \Z^d_N}\sum_{j=1}^{2}(f(q,v^{j,x,x+2e_1})-f(q,v))+\\
%& \sum_{x \in \Z^d_N}\sum_{k=2}^d\sum_{j=1}^{2}(f(q,v^{j,x,x+e_k})-f(q,v))+ \sum_{x \in \Z^d_N}\sum_{k=1}^d\sum_{j=3}^{d^*}(f(q,v^{j,x,x+e_k})-f(q,v)). \\
\end{align}
Here, $\mathbf{v}^{j,\mathbf{x},\mathbf{y}}$ is obtained from $\mathbf{v}$ by exchanging the variables $v^j_{\mathbf{x}}$ and $v^j_{\mathbf{y}}$. 

\begin{rem}
This specific choice of noise is not important. Our proof is also applicable for the continuous noise used in \cite{BBO}, and yields the same divergence exponent of the thermal conductivity. 
\end{rem}

The whole dynamics is generated by the operator $L:=A+BG+\gamma S$ with the operators $A$ and $G$ acting on functions $f \in C^1(\R^{2d^*N^d})$ as 
\begin{align}\label{AGop}
A  & = \sum_{\mathbf{x} \in \Z^d_N} \{\mathbf{v}_{\mathbf{x}} \cdot \partial_{\mathbf{q}_{\mathbf{x}}}+ [\Delta \mathbf{q}]_{\mathbf{x}} \cdot \partial_{ \mathbf{v}_{\mathbf{x}}} \} =\sum_{\mathbf{x} \in \Z^d_N}\sum_{j=1}^{d^*}\big( v_{\mathbf{x}}^j  \partial_{q_{\mathbf{x}}^j} + [\Delta q^j]_{\mathbf{x}} \partial_{v_{\mathbf{x}}^j}\big), \\
G  & = \sum_{\mathbf{x} \in \Z^d_N} \mathbf{v}_{\mathbf{x}} \cdot \sigma \partial_{\mathbf{v}_{\mathbf{x}}}  = \sum_{\mathbf{x} \in \Z^d_N} \big( v_{\mathbf{x}}^2  \partial_{v_{\mathbf{x}}^1} -  v_{\mathbf{x}}^1  \partial_{v_{\mathbf{x}}^2} \big).  \nonumber 
%G^{alt} = & \sum_{x \in \Z^d_N} (-1)^{x_1}\big( v_{\mathbf{x}}^2  \partial_{v_{\mathbf{x}}^1} -  v_{\mathbf{x}}^1  \partial_{v_{\mathbf{x}}^2} \big) \\
\end{align}
Because $\sum_{\mathbf{x} \in \Z^d_N}\mathbf{q}_{\mathbf{x}}$, $\sum_{\mathbf{x} \in \Z^d_N}\mathbf{v}_{\mathbf{x}}$ and $\sum_{\mathbf{x} \in \Z^d_N}\E_{\mathbf{x}}$ are conserved by the operator $S$, the microcanonical state space $\Omega_{N,E}$ is also conserved by the whole dynamics. Moreover, the uniform probability measure on the space $\Omega_{N,E}$ is stationary for the dynamics, and $A$ and $G$ are antisymmetric and $S$ is symmetric with respect to this measure. We denote this microcanonical measure by $\mu_{N,E}$ and the expectation with respect to it by $E_{N,E}[ \cdot ]$.

We aim to obtain the asymptotic behavior of the microcanonical Green--Kubo integral defined by taking the infinite size limit in the correlation of the integrated total current $\kappa_{N,E}^{a,b}(t)$;
\begin{align}\label{microthermal}
\kappa_{N,E}^{a,b}(t)=\frac{1}{2 N^dE^2 t} \mathbb{E}_{N,E} [ \big( \sum_{\mathbf{x} \in \Z^d_N} J_{\mathbf{x},\mathbf{x}+\mathbf{e}_a}([0,t]) \big)\big( \sum_{\mathbf{x} \in \Z^d_N} J_{\mathbf{x},\mathbf{x}+\mathbf{e}_b}([0,t]) \big) ]
\end{align}
for $a,b \in \{1,2,\dots,d\}$ where $\mathbb{E}_{N,E}$ is the expectation for the dynamics starting from the microcanonical measure $\mu_{N,E}$ and $J_{\mathbf{x},\mathbf{x}+\mathbf{e}_a}([0,t])$ is the total energy current from $\mathbf{x}$ to $\mathbf{x}+\mathbf{e}_a$ up to time $t$. By symmetry, it is easy to see that $\kappa_{N,E}^{a,b}(t)=\delta_{a,b}\kappa_{N,E}^{1,1}(t)$ (see Lemma \ref{symm}).

Our main result for the microcanonical Green--Kubo integral is the following: 
\begin{thm}\label{mr1}
For any $d \ge 1$, $d^* \ge 2$ and $t \ge 0$, $\kappa^{1,1}_{E}(t):=\lim_{N \to \infty}\kappa_{N,E}^{1,1}(t)$ exists. Moreover,
if $B \neq 0$,%As $T \to \infty$, 
\begin{align}
\begin{cases}\label{microasym}
\kappa^{1,1}_{E}(t) \sim t^{\frac{1}{4}} \  (t \to \infty) \quad  & \ \text{if} \ d=1 \ \text{and} \ d^*=2, \\ 
\kappa^{1,1}_{E}(t) \sim t^{\frac{1}{2}} \  (t \to \infty) \quad  & \  \text{if} \ d=1 \ \text{and} \ d^*\ge 3, \\ 
\kappa^{1,1}_{E}(t) \sim  \log t \  (t \to \infty) \quad  & \ \text{if} \ d=2 \ \text{and} \ d^*\ge 2,\\
\displaystyle \limsup_{t \to \infty} \kappa^{1,1}_{E}(t)  < \infty \quad  & \  \text{if} \ d \ge 3 \ \text{and} \ d^*\ge 2.
\end{cases}
\end{align}
\end{thm}

\begin{rem}
Our proof also applies to the case $B=0$ and for that case the asymptotic behavior of $\kappa^{1,1}_{E}(t)$ is
\begin{align}
\begin{cases}\label{microasym}
\kappa^{1,1}_{E}(t) \sim t^{\frac{1}{2}} \  (t \to \infty) \quad  & \ \text{if} \ d=1,  \\ 
\kappa^{1,1}_{E}(t) \sim  \log t \  (t \to \infty) \quad &  \ \text{if} \ d=2, \\
\displaystyle \limsup_{t \to \infty} \kappa^{1,1}_{E}(t)  < \infty \quad &  \ \text{if} \ d\ge3
\end{cases}
\end{align}
for any $d^* \ge 1$. Note that $d^*=1$ is also allowed at this remark because $B=0$. This asymptotic behavior has already been shown in Theorem 1 and 2 of \cite{BBO} under the condition $d=d^*$. 
\end{rem}

\begin{rem}
The thermal conductivity via the microcanocial Green--Kubo integral in the direction $\mathbf{e}_1$ is defined as the limit (when it exists)
\begin{align*}
\kappa^{1,1}_{E}=\lim_{t \to \infty} \lim_{N \to \infty} \frac{1}{2 E^2 t N^d} \mathbb{E}_{N,E} [(\sum_{\mathbf{x} \in \Z^d_N} J_{\mathbf{x},\mathbf{x}+\mathbf{e}_1}([0,t]))^2].
\end{align*} 
Theorem \ref{mr1} shows that the Green-Kubo integral
\begin{align*}
\kappa^{1,1}_{E}(t) & =  \lim_{N \to \infty} \frac{1}{2 E^2 t N^d} \mathbb{E}_{N,E} [(\sum_{\mathbf{x} \in \Z^d_N} J_{\mathbf{x},\mathbf{x}+\mathbf{e}_1}([0,t]))^2]
\end{align*}
exists for any $t$, but $\kappa^{1,1}_{E}=\lim_{t \to \infty}\kappa^{1,1}_{E}(t)=\infty$ for $d=1,2$.
\end{rem}

\begin{rem}
In this model, the dispersion relations are $\tilde{\omega}^{\pm}(\theta)=\sqrt{\omega_{\theta}^2+(\frac{B}{2})^2}\pm \frac{B}{2}$ for $d^*=2$, and $\tilde{\omega}^{\pm}(\theta)$ and $\omega_{\theta}$ for $d^*\ge 3$ where $\omega_{\theta}=\sqrt{4\sum_{a=1}^d\sin^2(\pi \theta^a)}$ is the dispersion relation of the dynamics with zero magnetic field, $B=0$. In particular, $\partial_{\theta^1}\tilde{\omega}^{\pm}(\theta)=\frac{4\pi\sin(\pi \theta^1)\cos(\pi \theta^1)}{\sqrt{\omega_{\theta}^2+(\frac{B}{2})^2}}$ and the sound velocity of these modes vanish: $\lim_{\theta \to 0}|\partial_{\theta^1}^{\pm}\tilde{\omega}(\theta)|=0$. 
\end{rem}

\smallskip

Next, we consider the canonical measure, which is defined in \cite{BBO} by 
\[
\mu_{N,\beta}(dqdv)=\frac{1}{Z}\exp(-\b \sum_{\mathbf{x} \in \Z^d_N}\E_{\mathbf{x}})\Pi_{{\mathbf{x}} \in \Z^d_N}\Pi_{j=1}^{d^*}dq^j_{\mathbf{x}}dv^j_{\mathbf{x}}.
\]
However, the partition function $Z$ in the expression above is infinite since the total energy $\sum_{\mathbf{x} \in \Z^d_N}\E_{\mathbf{x}}$ is invariant under the translation $(\mathbf{q}_{\mathbf{x}})_{\mathbf{x} \in \Z^d_N} \to (\mathbf{q}_{\mathbf{x}}+\mathbf{c})_{\mathbf{x} \in \Z^d_N}$ for any $\mathbf{c} \in \R^{d^*}$. Hence, to handle the canonical measures, we redefine the dynamics with different coordinates and state spaces, which is physically identical to (\ref{eq:dynamics}) in the bulk, but the boundary conditions are different. Since state spaces and canonical measures in the new coordinates become complicated for $d \ge 2$, we study only the case $d=1$. Moreover, for notational simplicity, we restrict ourselves to the case $d^*=2$. In this setting, we study three versions of the dynamics, namely $(0)$ zero magnetic field, $B=0$, (i) uniform charges, and (ii) alternate charges. 

We now introduce the dynamics in new coordinates. Suppose $d=1$ and $d^*=2$. The oscillators are labeled by $x \in \Z_N:=\Z/ N\Z$. We use $x$ rather than $\mathbf{x}$ to emphasize that $x$ is not a vector. The velocity of the oscillator $x$ is denoted by $\mathbf{v}_x \in \R^{2}$ as before, and instead of $\mathbf{q}_{x}$, we introduce $\mathbf{r}_{x} \in \R^2$ which represents the difference in the displacement between the oscillators $x$ and $x+1$. We can interpret $\mathbf{r}_{x}=\mathbf{q}_{x+1}-\mathbf{q}_x$ but $\mathbf{q}_x$ is no longer periodic. In other words, we do not assume $\sum_{x \in \Z_N}\mathbf{r}_{x}=\mathbf{0}$. 

The equations of motion for the three cases are
\begin{align*} (0) 
\begin{cases}
\frac{d}{dt} r_x^j & =v_{x+1}^j -v_x^j,\\
\frac{d}{dt} v_x^j & =r_{x}^j -r_{x-1}^j,
\end{cases}
\end{align*}
\begin{align*} (i) 
\begin{cases}
\frac{d}{dt} r_x^j & =v_{x+1}^j -v_x^j,\\
\frac{d}{dt} v_x^j & =r_{x}^j -r_{x-1}^j +\delta_{j,1}Bv_x^2-\delta_{j,2}B v_x^1,
\end{cases}
\end{align*}
and 
\begin{align*} (ii) 
\begin{cases}
\frac{d}{dt} r_x^j & =v_{x+1}^j -v_x^j,\\
\frac{d}{dt} v_x^j & =r_{x}^j -r_{x-1}^j + (-1)^x (\delta_{j,1}Bv_x^2-\delta_{j,2}B v_x^1)
\end{cases}
\end{align*}
for $j=1,2$ and $x \in \Z_N$. For dynamics (ii), we assume that $N$ is even. Note that dynamics (i) is obtained from the dynamics (\ref{eq:dynamics2}) by changing the coordinates formally as $\mathbf{r}_{x}=\mathbf{q}_{x+1}-\mathbf{q}_x$. 

As before obvious by the dynamics (0), (i) and (ii) conserve the total energy $\sum_{x \in \Z_N}\E_x$, where we define the energy of the oscillator $x$ by
\[
\E_x=\frac{|\mathbf{v}_x|^2}{2}+\frac{|\mathbf{r}_x|^2}{4}+\frac{|\mathbf{r}_{x-1}|^2}{4}.
\]

We now consider the stochastic perturbation of these dynamics. The operator of the perturbation is given by the operator $\ga S_r$ acting on functions $f : \R^{4N} \to \R$ as
\begin{align}\label{Sopc}
S_rf = & \sum_{x \in \Z_N}\sum_{j=1}^{2}(f(\mathbf{r},\mathbf{v}^{j,x,x+1})-f(\mathbf{r},\mathbf{v})) 
%S^{alt} f = & \sum_{x \in \Z^d_N}\sum_{j=1}^{2}(f(q,v^{j,x,x+2e_1})-f(q,v))+\\
%& \sum_{x \in \Z^d_N}\sum_{k=2}^d\sum_{j=1}^{2}(f(q,v^{j,x,x+e_k})-f(q,v))+ \sum_{x \in \Z^d_N}\sum_{k=1}^d\sum_{j=3}^{d^*}(f(q,v^{j,x,x+e_k})-f(q,v)). \\
\end{align}
whose physical interpretation is the same as before. The generator of the whole dynamics is $L^{\#}_r:=A_r+BG^{(\#)}_r+\gamma S_r$ with the operators $A_r$ and $G^{(\#)}_r$ for $\#=0,i,ii$ acting on functions $f \in C^1(\R^{4N})$ as 
\begin{align}\label{AGopc}
A_r = & \sum_{x \in \Z_N}\sum_{j=1}^{2}\big( (v_{x+1}^j -v_x^j) \partial_{r_x^j} + (r_{x}^j -r_{x-1}^j) \partial_{v_x^j}\big), \nonumber \\
G^{(0)}_r  = & 0,\nonumber \\
G^{(i)}_r = & \sum_{x \in \Z_N} \big( v_x^2  \partial_{v_x^1} -  v_x^1  \partial_{v_x^2} \big), \\
G^{(ii)}_r = & \sum_{x \in \Z_N} (-1)^x\big( v_x^2  \partial_{v_x^1} -  v_x^1  \partial_{v_x^2} \big). \nonumber
\end{align}

For each $\beta>0$, the canonical measure $\mu_{N,\beta}$ on $\R^{4N}$ is defined by
\begin{align*}
\mu_{N,\b}(dr dp) & =\frac{1}{Z_{\b}^{4N}}\exp(-\beta \sum_{x} \E_x) \Pi_{x \in \Z_N} \Pi_{j=1}^2 dr_x^j dv_x^j \\
& =\frac{1}{Z_{\b}^{4N}}\exp(-\frac{\beta}{2} \sum_{x}\sum_{j=1}^2((v_x^j)^2+(r_x^j)^2) ) \Pi_{x \in \Z_N}\Pi_{j=1}^2 dr_x^j dv_x^j
\end{align*}
where $Z_{\b}=\sqrt{\frac{2 \pi}{\b}}$. Dynamics (0), (i) and (ii) are all stationary under $\mu_{N,\beta}$ since $A_r$ and $G^{(\#)}_r$ for $\#=i,ii$ are antisymmetric and $S_r$ is symmetric with respect to this measure.

We consider the canonical Green--Kubo integral defined by taking the infinite size limit in the correlation of the integrated total current $\kappa^{(\#)}_{N,\b}(t)$;
\begin{align*}
\kappa^{(\#)}_{N,\b}(t)=\frac{\b^2}{8Nt}\mathbb{E}_{N,\b}^{(\#)} [ \big(\sum_{\mathbf{x} \in \Z_N}  J_{x,x+1}([0,t]) \big)^2]
\end{align*}
where $\mathbb{E}_{N,\b}^{(\#)}$ is the expectation for the dynamics generated by $L^{(\#)}$ starting from the canonical measure $\mu_{N,\b}$ and $J_{x,x+1}([0,t])$ is the total energy current from $x$ to $x+1$ up to time $t$. 

Our main result for the canonical Green--Kubo integral is the following:
\begin{thm}\label{mr2}
For any $t \ge 0$ and $\#=0,i, ii$, $\kappa^{(\#)}_{\b}(t):=\lim_{N\to \infty}\kappa^{(\#)}_{N,\b}(t)$ exists. Moreover,
%As $T \to \infty$, 
\begin{align*}
\begin{cases}
\kappa^{(0)}_{\b}(t) \sim t^{\frac{1}{2}} \  (t \to \infty), \\ 
\kappa^{(i)}_{\b}(t) \sim t^{\frac{1}{4}} \  (t \to \infty), \\ 
\kappa^{(ii)}_{\b}(t) \sim t^{\frac{1}{2}} \  (t \to \infty) \quad \text{if} \quad \ga \le 1.
\end{cases}
\end{align*}
\end{thm}

\begin{rem}
We conjecture that the assumption $\ga \le 1$ is not essential and the asymptotic behavior of $\kappa^{(ii)}_{\b}(t)$ does not depend on the parameter $\gamma$. This assumption is needed only in Subsection \ref{inverselaplace2}. We provide some discussions on it in Remark \ref{ga>1case}.
\end{rem}

\smallskip

The remainder of the article is organized as follows; In Section \ref{microsec}, we study the model with uniform charges under microcanonical measures in general dimensions, and establish Theorem \ref{mr1}. Section \ref{canosec} concerns the one-dimensional chain of uncharged oscillators and oscillators with uniform and alternate charges under canonical measures in the two-dimensional space, including the proof of Theorem \ref{mr2}.

\section{Uniform charge model under microcanonical measures}\label{microsec}

This section provides a detailed study of the uniform charge model in general dimensions in the coordinate $(\mathbf{q},\mathbf{v})$. We start by recalling the description of the model in Subsection \ref{modelsec1}. Subsection \ref{measec1} is for the analysis of conserved quantities and microcanonical states. In Subsection \ref{currentsec1}, we study the thermal conductivity in terms of the current-current correlation and give a proof of Theorem \ref{mr1} by assuming one key proposition. The proof of this proposition is divided into several steps and given in the last three subsections.

\subsection{Model}\label{modelsec1}

We consider the Markov process $(\mathbf{q}_{\mathbf{x}}(t), \mathbf{v}_{\mathbf{x}}(t))_{\mathbf{x} \in \Z^d_N}$ on $\R^{2d^*N^d}$ generated by the operator $L=A+BG+\gamma S$ where $A,G$ and $S$ are defined at (\ref{Sop}) and (\ref{AGop}). Note that $B \in \R$ and $\gamma >0$ are parameters that regulate the strength (and the direction) of the magnetic field and the stochastic perturbation respectively. 

\begin{rem}\label{remskew}
We can replace the matrix $B\sigma$ by any real skew-symmetric matrix $M$ with size $d^*$ at (\ref{eq:dynamics}) and consider the dynamics associated to this matrix, which conserves the total energy. Note that any real skew-symmetric matrix $M$ is given in the form $M=U\Sigma U^T$ where $U$ is orthogonal and $\Sigma$ is block diagonal with a form
\[
\Sigma=\left(
\begin{array}{ccccccc}
B_1 J & 0 & \cdots & \cdots &\cdots & \cdots &0 \\
0 & B_2 J &  0 & \cdots & \cdots & \cdots & 0 \\
\vdots & 0 & \cdots & 0 & \cdots & \cdots & 0  \\
\vdots & \vdots & 0 & B_m J  & \cdots & \cdots & 0 \\
\vdots & \vdots & \vdots & 0  & 0 & \cdots & 0 \\
\vdots  & \vdots  & \vdots  & \vdots  &  \vdots & \cdots & 0 \\
0  & 0 & 0 &0  &  0& \cdots & 0

\end{array}
\right), \qquad J= \left(
\begin{array}{cc}
0 & 1 \\
-1 & 0 
\end{array}
\right)
\]
for some $m \le \frac{d^*}{2}$ and real numbers $B_i \neq 0, i=1,2,\dots,m$. In the above expression of $\Sigma$, each $0$ correspond to $2 \times 2$ size zero matrix, but only when $d^*$ is odd, each $0$ in the last row is a scalar. Obviously, $\Sigma=B\sigma$ if $m=1$ and $B_1=B$. By changing the coordinates $(\mathbf{q}_{\mathbf{x}}, \mathbf{v}_{\mathbf{x}})$ to $(\tilde{\mathbf{q}}_{\mathbf{x}}=U\mathbf{q}_{\mathbf{x}}, \tilde{\mathbf{v}}_{\mathbf{x}}=U\mathbf{v}_{\mathbf{x}})$, the deterministic dynamics associated to the matrix $M$ is rewritten as 
\begin{align*}
\begin{cases}
\frac{d}{dt} \tilde{\mathbf{q}}_{\mathbf{x}}  & =\tilde{\mathbf{v}}_{\mathbf{x}} \\
\frac{d}{dt} \tilde{\mathbf{v}}_{\mathbf{x}} & =[\Delta \tilde{\mathbf{q}}]_{\mathbf{x}} -\Sigma \tilde{\mathbf{v}}_{\mathbf{x}}.
\end{cases}
\end{align*}
Then, Theorem \ref{mr1} is generalized to this dynamics with stochastic perturbation and the asymptotic behavior of the microcanonical Green-Kubo integral is given as
\begin{align*}
\begin{cases}
\kappa^{1,1}_{E}(t) \sim t^{\frac{1}{4}} \  (t \to \infty) \quad (d=1,d^*=2m) \\ 
\kappa^{1,1}_{E}(t) \sim t^{\frac{1}{2}} \  (t \to \infty) \quad (d=1,d^* \neq 2m) \\ 
\kappa^{1,1}_{E}(t) \sim  \log t \  (t \to \infty) \quad (d=2,d^* \ge 2)\\
\displaystyle \limsup_{t \to \infty} \kappa^{1,1}_{E}(t)  < \infty \quad \quad (d \ge 3,d^* \ge 2).
\end{cases}
\end{align*}
In particular, if $d^*$ is odd, then $\kappa^{1,1}_{E}(t) \sim t^{\frac{1}{2}}$ for $d=1$. The generalization from $B\sigma$ to $\Sigma$ is straightforward, so we omit the proof.
\end{rem}

\subsection{Conserved quantities and microcanonical states}\label{measec1}

As described in the Introduction, the dynamics (\ref{eq:dynamics}) conserves the total energy $\sum_{\mathbf{x} \in \Z^d_N}\E_{\mathbf{x}}$. In particular, $G \E_{\mathbf{x}}=0$. On the other hand, the dynamics does not conserved the total velocity $\sum_{\mathbf{x} \in \Z^d_N}\mathbf{v}_{\mathbf{x}}$ if $B \neq 0$ since $G \sum_{x \in \Z_N^d}\mathbf{v}_{\mathbf{x}} \neq 0$. 

We have already seen that if we start from the initial condition $\sum_{\mathbf{x}}\mathbf{q}_{\mathbf{x}}=\sum_{\mathbf{x}}\mathbf{v}_{\mathbf{x}}=0$, then the condition is conserved by the dynamics. We discuss what happens for general initial conditions here. Let $\bar{\mathbf{q}}=\frac{1}{N^d}\sum_{\mathbf{x}}\mathbf{q}_{\mathbf{x}}$ and $\bar{\mathbf{v}}=\frac{1}{N^d}\sum_{\mathbf{x}}\mathbf{v}_{\mathbf{x}}$. Then, it is easy to see that 
\begin{align}\label{eq:dynamicsQV} 
\begin{cases}
\frac{d}{dt} \bar{\mathbf{q}}  & =\bar{\mathbf{v}}\\
\frac{d}{dt} \bar{\mathbf{v}} & =-B\sigma  \bar{\mathbf{v}}
\end{cases}
\end{align}
and then $|\bar{\mathbf{v}}|^2$ is conserved. Moreover, the centered configuration $(\mathbf{Q}_{\mathbf{x}}:=\mathbf{q}_{\mathbf{x}}-\bar{\mathbf{q}}, \mathbf{V}_{\mathbf{x}}:=\mathbf{v}_{\mathbf{x}}-\bar{\mathbf{v}})_{\mathbf{x} \in \Z_N^d}$ also evolves according to the stochastic dynamics generated by $L$. In this way, the dynamics of $(\mathbf{q}_{\mathbf{x}},\mathbf{v}_{\mathbf{x}})$ is decomposed into the dynamics of $(\mathbf{Q}_{\mathbf{x}}, \mathbf{V}_{\mathbf{x}})$ and $( \bar{\mathbf{q}}, \bar{\mathbf{v}})$. In particular, under any stationary measure, the distribution of $(\mathbf{Q}_{\mathbf{x}}, \mathbf{V}_{\mathbf{x}})$ is also stationary.
%The energy $\E_{\mathbf{x}}$ is also decomposed as
%\begin{align*}
%\E_{\mathbf{x}} & =\frac{|\mathbf{v}_{\mathbf{x}}|^2}{2}+\sum_{|\mathbf{y}-\mathbf{x}|=1}\frac{|\mathbf{q}_{\mathbf{y}}-\mathbf{q}_{\mathbf{x}}|^2}{4} = \frac{|\tilde{\mathbf{v}}_{\mathbf{x}}|^2+|\mathbf{V}|^2}{2}+\sum_{|\mathbf{y}-\mathbf{x}|=1}\frac{|\tilde{\mathbf{q}}_{\mathbf{y}}-\tilde{\mathbf{q}}_{\mathbf{x}}|^2}{4} \\
%& = \tilde{\E}_\mathbf{x} +\frac{|\mathbf{V}|^2}{2}
%\end{align*}
%where $\tilde{\E}_\mathbf{x}$ is the energy of the oscillator $\mathbf{x}$ with respect to the configuration $(\tilde{\mathbf{q}}_{\mathbf{x}}, \tilde{\mathbf{v}}_{\mathbf{x}})$. 

In the next subsection, we will see that the total instantaneous energy current denoted by $\sum_{\mathbf{x}} j_{\mathbf{x},\mathbf{x}+\mathbf{e}_1}$ is explicitly given as
\[
\sum_{\mathbf{x}} j_{\mathbf{x},\mathbf{x}+\mathbf{e}_1}=-\frac{1}{2}\sum_{j=1}^{d^*}\sum_{\mathbf{x}}v_{\mathbf{x}}^j(q_{\mathbf{x}+\mathbf{e}_1}^j-q_{\mathbf{x}-\mathbf{e}_1}^j)
\]
and plays an essential role for the asymptotic behavior of the Green-Kubo integral. Since 
\[
\sum_{j=1}^{d^*}\sum_{\mathbf{x}}v_{\mathbf{x}}^j(q_{\mathbf{x}+\mathbf{e}_1}^j-q_{\mathbf{x}-\mathbf{e}_1}^j)=\sum_{j=1}^{d^*}\sum_{\mathbf{x}}V_{\mathbf{x}}^j(Q_{\mathbf{x}+\mathbf{e}_1}^j-Q_{\mathbf{x}-\mathbf{e}_1}^j),
\]
this total instantaneous energy current depends only on the centered configuration $(Q_{\mathbf{x}}, V_{\mathbf{x}})$. Also, since $\sum_{\mathbf{x}}Q_{\mathbf{x}}=\sum_{\mathbf{x}}V_{\mathbf{x}} =\mathbf{0}$ at time $t=0$, it stays in $\Omega_{N,\tilde{E}}$ for $\tilde{E}=\frac{1}{N^d} \sum \tilde{\mathcal{E}}_{\mathbf{x}}$ where $\tilde{\E}_\mathbf{x}$ is the energy of the oscillator $\mathbf{x}$ associated to the configuration $(\mathbf{Q}_{\mathbf{x}}, \mathbf{V}_{\mathbf{x}})$. By definition, 
\begin{align*}
E:=\frac{1}{N^d}\sum_{\mathbf{x}}\E_{\mathbf{x}} & =\frac{1}{N^d}\sum_{\mathbf{x}} \left(\frac{|\mathbf{v}_{\mathbf{x}}|^2}{2}+\sum_{|\mathbf{y}-\mathbf{x}|=1}\frac{|\mathbf{q}_{\mathbf{y}}-\mathbf{q}_{\mathbf{x}}|^2}{4} \right) \\
& = \frac{1}{N^d}\sum_{\mathbf{x}} \left( \frac{|\mathbf{V}_{\mathbf{x}}|^2+2\mathbf{V}_{\mathbf{x}} \cdot \bar{\mathbf{v}} +  |\bar{\mathbf{v}}|^2}{2}+\sum_{|\mathbf{y}-\mathbf{x}|=1}\frac{|\mathbf{Q}_{\mathbf{y}}-\mathbf{Q}_{\mathbf{x}}|^2}{4}\right) \\
& = \frac{1}{N^d}\sum_{\mathbf{x}} \tilde{\E}_\mathbf{x} +\frac{|\bar{\mathbf{v}}|^2}{2}=\tilde{E}+\frac{|\bar{\mathbf{v}}|^2}{2}.
\end{align*}

From the above observation, on a general microcanonical state with the averaged energy $E$ and the square of the norm of averaged velocity $|\bar{\mathbf{v}}|^2$, the distribution of the centered configuration $(\mathbf{Q}_{\mathbf{x}}, \mathbf{V}_{\mathbf{x}})$ is $\mu_{N,\tilde{E}}$ where $\tilde{E}=E-\frac{|\bar{\mathbf{v}}|^2}{2}$, and the microcanonical Green-Kubo integral is given via $\kappa_{N,\tilde{E}}^{a,b}$. 

In this way, we can reduce the study of the Green-Kubo integral on a general microcanonical state to the one on the microcanonical state $\Omega_{N,E}$.

\subsection{Instantaneous energy current correlation}\label{currentsec1}

In this subsection, we study the asymptotic behavior of the correlation of the integrated energy current $\kappa_{N,E}^{a,b}(t)$. We follow the strategy of \cite{BBO}.

The energy conservation law can be read locally as
\[
\E_{\mathbf{x}}(t)-\E_{\mathbf{x}}(0) = \sum_{a=1}^d \big(J_{\mathbf{x}-\mathbf{e}_a,\mathbf{x}}([0, t]) - J_{\mathbf{x},\mathbf{x}+\mathbf{e}_a} ([0, t])\big)
\]
where $J_{\mathbf{x},\mathbf{x}+\mathbf{e}_a} ([0, t])$ is the total energy current between $\mathbf{x}$ and $\mathbf{x}+\mathbf{e}_a$ up to
time $t$. This can be written as
\[
J_{\mathbf{x},\mathbf{x}+\mathbf{e}_a} ([0, t])=\int_0^t j_{\mathbf{x},\mathbf{x}+\mathbf{e}_a}(s)ds + \M_{\mathbf{x},\mathbf{x}+\mathbf{e}_a}(t)
\]
where $\M_{\mathbf{x},\mathbf{x}+\mathbf{e}_a}(t)$ is the martingale given by
\[
\M_{\mathbf{x},\mathbf{x}+\mathbf{e}_a}(t)=-\frac{1}{2}\sum_{j=1}^{d^*} \int_0^t \big\{ (v_{\mathbf{x}+\mathbf{e}_a}^j)^2(s-)-(v_{\mathbf{x}}^j)^2(s-) \big\} dM_{j,\mathbf{x},\mathbf{x}+\mathbf{e}_a}(s)
\]
with $M_{j,\mathbf{x},\mathbf{x}+\mathbf{e}_a}(t)=N_{j,\mathbf{x},\mathbf{x}+\mathbf{e}_a}^{\ga}(t)-\ga t$ and $\{N_{j,\mathbf{x},\mathbf{x}+\mathbf{e}_a}^{\ga}(t)\}_{j=1,\dots,d^*, \ \mathbf{x} \in \Z_N^d,a=1,\dots,d}$ are $d^* d N^d$-independent Poisson processes with intensity $\ga$.

%Let $\mathbb{S}_{\a,i,i+1}F(q,v)=F(q,v^{\a,i,i+1})-F(q,v)$ for $\a=x,y$ and $i \in \T_N$. Then, our dynamics is given as 
%\[
%F(q(t),v(t))-F(q(0),v(0))=\int_0^t \mathbb{L}F(q(s),v(s))ds +  \sum_{i=1}^{N} \sum_{\a=x,y} \int_0^t  \mathbb{S}_{\a,i,i+1}F(q(s),v(s)) dM_{\a,i}(s)
%\]
%where $M_{\a,i}(t)=N_{\a,i}^{\ga}(t)-\ga t$ and $\{N_{\a,i}^{\ga}(t)\}_{\a=x,y, \ i \in \T_N}$are $2N$-independent poisson processes with parameter $\ga$.

The instantaneous energy current $j_{\mathbf{x},\mathbf{x}+\mathbf{e}_a}$ can be written as $j_{\mathbf{x},\mathbf{x}+\mathbf{e}_a}=j_{\mathbf{x},\mathbf{x}+\mathbf{e}_a}^a+j_{\mathbf{x},\mathbf{x}+\mathbf{e}_a}^s$ with
\begin{align*}
j_{\mathbf{x},\mathbf{x}+\mathbf{e}_a}^a  = -\frac{1}{2} \sum_{j=1}^{d^*} (q_{\mathbf{x}+\mathbf{e}_a}^j-q_{\mathbf{x}}^j)(v_{\mathbf{x}+\mathbf{e}_a}^j+v_{\mathbf{x}}^j), \quad 
j_{\mathbf{x},\mathbf{x}+\mathbf{e}_a}^s  =- \frac{\ga}{2} \sum_{j=1}^{d^*} \big\{ (v_{\mathbf{x}+\mathbf{e}_a}^j)^2-(v_{\mathbf{x}}^j)^2 \big\}
\end{align*}
where $j_{\mathbf{x},\mathbf{x}+\mathbf{e}_a}^a$ is the contribution of the deterministic dynamics to the instantaneous energy current and $j_{\mathbf{x},\mathbf{x}+\mathbf{e}_a}^s$ is the stochastic noise contribution to it.
\begin{rem}
Since $G \E_{\mathbf{x}}=0$, the instantaneous energy currents does not depend on $B$.
\end{rem}
Recall that the microcanocial Green-Kubo integral is defined by taking the infinite size limit in
\begin{align*}
\kappa_{N,E}^{a,b}(t)=\frac{1}{2 N^dE^2 t} \mathbb{E}_{N,E} [ \big( \sum_{\mathbf{x} \in \Z^d_N} J_{\mathbf{x},\mathbf{x}+\mathbf{e}_a}([0,t]) \big)\big( \sum_{\mathbf{x} \in \Z^d_N} J_{\mathbf{x},\mathbf{x}+\mathbf{e}_b}([0,t]) \big) ].
\end{align*}

We first give a simple lemma about a spatial symmetry of the system.

\begin{lem}\label{symm}
For any $1 \le a,b \le d$,
\[
\kappa_{N,E}^{a,b}(t) =\delta_{a,b} \kappa_{N,E}^{1,1}(t).
\]
\end{lem}
\begin{proof}
Let $\tau: \{1,2,\dots, d\} \to \{1,2,\dots, d\}$ be any permutation and $(\tau \mathbf{x})^a=\mathbf{x}^{\tau(a)}$. From the symmetry with respect to $\tau$ of the generator $L$ or simply by the explicit form of (\ref{eq:dynamics}) and (\ref{Sop}), the distribution of the Markov process $(\mathbf{q}_{\tau \mathbf{x}} (t) ,\mathbf{v}_{\tau \mathbf{x}}(t))$ with the initial measure $\mu_{N,E}$ is the same as that of $(\mathbf{q}_{\mathbf{x}} (t) ,\mathbf{v}_{\mathbf{x}}(t))$ with the same initial measure. In particular, $\kappa_{N,E}^{a,a}(t)=\kappa_{N,E}^{1,1}(t)$ for any $1 \le a \le d$. Next, let $R^a: \Z^d_N \to \Z^d_N$ be the permutation $(R^a  \mathbf{x})^b = (-1)^{\delta_{a,b}}\mathbf{x}^b$. By the same reason, the distribution of $(\mathbf{q}_{R^a \mathbf{x}} (t) ,\mathbf{v}_{R^a \mathbf{x}}(t))$ is same as that of $(\mathbf{q}_{\mathbf{x}} (t) ,\mathbf{v}_{\mathbf{x}}(t))$ with the initial measure $\mu_{N,E}$. For $a \neq b$, 
\[
\sum_{\mathbf{x} \in \Z^d_N} J_{\mathbf{x},\mathbf{x}+\mathbf{e}_a}([0,t]) = -\sum_{\mathbf{x} \in \Z^d_N} J_{R^a\mathbf{x},R^a(\mathbf{x}+\mathbf{e}_a)}([0,t]), \quad  \sum_{\mathbf{x} \in \Z^d_N} J_{\mathbf{x},\mathbf{x}+\mathbf{e}_b}([0,t]) = \sum_{\mathbf{x} \in \Z^d_N} J_{R^a\mathbf{x},R^a(\mathbf{x}+\mathbf{e}_b)}([0,t])
\]
and so $\kappa_{N,E}^{a,b}(t)= -\kappa_{N,E}^{a,b}(t)=0$. 
\end{proof}

We denote $\mathcal{J}_{\mathbf{e}_1}[0,t]=\sum_{\mathbf{x} \in \Z_N^d}J_{\mathbf{x},\mathbf{x}+\mathbf{e}_1}([0,t])$. Since $\sum_{\mathbf{x}} j^s_{\mathbf{x},\mathbf{x}+\mathbf{e}_1}=0$, we have
\begin{align*}
 \frac{1}{N^d}\mathbb{E}_{N,E} [\mathcal{J}_{\mathbf{e}_1}[0,t]^2 ] 
% = \frac{1}{N^d}\mathbb{E}_{N,\E} [ ( \int_0^t \sum_{\mathbf{x}} j_{\mathbf{x},\mathbf{x}+\mathbf{e}_1}(s)ds  + \M^1(t) )^2 ] \\
= \frac{1}{N^d} \mathbb{E}_{N,E} [ ( \int_0^t \sum_{\mathbf{x}} j^a_{\mathbf{x},\mathbf{x}+\mathbf{e}_1}(s)ds  + \M^1(t) ) ^2 ]
\end{align*}
where $\M^1$ is a martingale given by
\[
\M^1 (t)= -\frac{1}{2}\sum_{\mathbf{x} \in \Z_N^d}\sum_{j=1}^{d^*} \int_0^t \big\{ (v_{\mathbf{x}+\mathbf{e}_1}^j)^2(s-)-(v_{\mathbf{x}}^j)^2(s-) \big\} dM_{j,\mathbf{x},\mathbf{x}+\mathbf{e}_1}(s).
\] 

Since $\sum_{\mathbf{x}} j^a_{\mathbf{x},\mathbf{x}+\mathbf{e}_1}=-\frac{1}{2}\sum_{j=1}^{d^*}\sum_{\mathbf{x}}v_{\mathbf{x}}^j(q_{\mathbf{x}+\mathbf{e}_1}^j-q_{\mathbf{x}-\mathbf{e}_1}^j)$ is bounded under the condition $\sum_{\mathbf{x}} \E_{\mathbf{x}} =E N^d$, $\A^1(t):=\int_0^t \sum_{\mathbf{x}} j^a_{\mathbf{x},\mathbf{x}+\mathbf{e}_1}(s)ds $ is a bounded variation process. 
Then, by Ito's formula for jump processes (c.f. \cite{P}), we have
\begin{align*}
\mathbb{E}_{N,E} [ \A^1(t) \M^1(t) ] & =\mathbb{E}_{N,E} [ \int_0^t\M^1(s)d\A^1(s) ]= \mathbb{E}_{N,E} [ \int_0^t\M^1(s)\sum_{\mathbf{x}} j^a_{\mathbf{x},\mathbf{x}+\mathbf{e}_1}(s)ds ]\\
& = \int_0^t \mathbb{E}_{N,E} [ \M^1(s)\sum_{\mathbf{x}} j^a_{\mathbf{x},\mathbf{x}+\mathbf{e}_1}(s) ]ds.
\end{align*}
Then, by considering the time reversed process as in \cite{BBO}, we obtain $\mathbb{E}_{N,E} [ \M^1(s)\sum_{\mathbf{x}} j^a_{\mathbf{x},\mathbf{x}+\mathbf{e}_1}(s) ]=0$ for any $s \ge 0$, and hence $\mathbb{E}_{N,E} [ \A^1(t) \M^1(t) ]=0$. Here, the generator of the time reversed process is $-A-BG+\gamma S$, but it does not make any difficulty to apply the argument in \cite{BBO}, since the current does not depend on $B$.
%\end{align*}

Also,
\begin{align*}
\frac{1}{N^d} \mathbb{E}_{N,E} [(\M^1(t))^2 ]= \frac{\ga}{4N^d} \sum_{\mathbf{x} \in \Z_N^d}\sum_{j=1}^{d^*} \int_0^t \mathbb{E}_{N,E} [ \big\{ (v_{\mathbf{x}+\mathbf{e}_1}^j)^2(s-)-(v_{\mathbf{x}}^j)^2(s-) \big\}^2 
]  ds \\
= \frac{\ga t}{4} \sum_{j=1}^{d^*} E_{N,E}[\big\{ (v_{\mathbf{e}_1}^j)^2-(v_{\mathbf{0}}^j)^2 \big\}^2 ] =\frac{d^* \ga t}{4}E_{N,E}[\big\{ (v_{\mathbf{e}_1}^1)^2-(v_{\mathbf{0}}^1)^2 \big\}^2 ].
\end{align*}

Thanks to the equivalence of ensembles given in Lemma \ref{equivalence} in Appendix \ref{app:equivalence}, the last quantity is equal to 
\[
\frac{E^2\ga t }{d^*}+ o(N)
\] 
where $\lim_{N \to \infty} o(N) =0$.
So far, we have shown that
\begin{align*}
& \frac{1}{2 E^2 tN^d} \mathbb{E}_{N,E} [(\sum_{\mathbf{x} \in \Z^d_N} J_{\mathbf{x},\mathbf{x}+\mathbf{e}_1}([0,t]))^2] \\
& = \frac{1}{2 E^2 tN^d} \int_0^t \int_0^t \mathbb{E}_{N,E}[\big( \sum_{\mathbf{x}} j^a_{\mathbf{x},\mathbf{x}+\mathbf{e}_1}(s) \big)\big(\sum_{\mathbf{x}} j^a_{\mathbf{x},\mathbf{x}+\mathbf{e}_1}(u) \big) ]dsdu + \frac{\ga}{2d^*}+ o(N). 
\end{align*}
By the spatial translation symmetry and a simple computation, the last term is rewritten as
\begin{align*}
& \frac{1}{2 E^2 t}\int_0^t \int_0^t \mathbb{E}_{N,E}[ \sum_{\mathbf{x}} j^a_{\mathbf{x},\mathbf{x}+\mathbf{e}_1}(s) j^a_{\mathbf{0},\mathbf{e}_1}(u) ]dsdu + \frac{\ga}{2d^*}+ o(N)  \\
% & =\int_0^t (1-\frac{s}{t}) \mathbb{E}_{N,\E}[ \sum_{\mathbf{x}} j^a_{\mathbf{x},\mathbf{x}+\mathbf{e}_1}(s) j^a_{\mathbf{0},\mathbf{e}_1}(0) ]\rangle ds + \frac{\ga  d^* const??}{8}+ o(N) \\
 & = \frac{1}{E^2} \int_0^t (1-\frac{s}{t}) C_N(s) ds + \frac{\ga}{2d^*}+ o(N)
\end{align*}
where $C_{N}(s)=\mathbb{E}_{N,E}[ \sum_{\mathbf{x}} j^a_{\mathbf{x},\mathbf{x}+\mathbf{e}_1}(s) j^a_{\mathbf{0},\mathbf{e}_1}(0) ]$. 
Then, Theorem \ref{mr1} follows immediately from the next proposition.
\begin{prop}\label{propu1}
The sequence of functions $\{ C_{N} :[0, \infty) \to \R \} $ converges compactly to a function $C_{\infty}:[0,\infty) \to \R$ as $N \to \infty$, namely the convergence is uniform on each compact subset of $[0,\infty)$. Moreover, 
\begin{align*}
C_{\infty}(s)=C_{1}(s)+C_{2}(s)+C_{3}(s)+(d^*-2)C_{4}(s)
\end{align*}
where $\displaystyle \limsup_{t \to \infty} \int_0^t(1-\frac{s}{t})C_{1}(s) < \infty$, $C_{2}(s) \sim s^{-d/2-1}$, $C_3(s) \sim s^{-d/4-1/2}$ and $C_4(s) \sim s^{-d/2}$ as $s \to \infty$.
\end{prop}
\begin{rem}
As we will see in Subsection \ref{propsec12}, $C_1(s)$ is an oscillating term and in particular $|C_1(s)| \sim s^{-d/2}$. In $d=1$, this term is bigger than the term $C_3(s)$, which is the main term for the Green-Kubo integral. This big oscillation of the current correlation function is also seen numerically in \cite{TSS}.
\end{rem}

\subsection{Proof of Proposition \ref{propu1}}
The proof of Proposition \ref{propu1} is divided into several steps. We first show the relative compactness of the set of functions $\{C_N\}_N$.
\begin{lem}\label{lem:uniformity1}
The set of functions $\{C_N : [0,\infty) \to \R\}_N$ is uniformly bounded and equicontinuous.
\end{lem}
\begin{proof}
By the translation invariance of the measure $\mu_{N,E}$ and the Cauchy-Schwarz inequality, 
\begin{align*}
& |C_N(s)|= |\frac{1}{N^d}\mathbb{E}_{N,E}[ \sum_{\mathbf{x}} j_{\mathbf{x},\mathbf{x}+\mathbf{e}_1}^a(s) \sum_{\mathbf{y}} j^a_{\mathbf{y},\mathbf{y}+\mathbf{e}_1}(0) ]| \\
& \le \frac{1}{N^d} \sqrt{\mathbb{E}_{N,E} [(\sum_{\mathbf{x}} j_{\mathbf{x},\mathbf{x}+\mathbf{e}_1}^a(s))^2] \mathbb{E}_{N,E} [ (\sum_{\mathbf{y}} j^a_{\mathbf{y},\mathbf{y}+\mathbf{e}_1}(0))^2 ]}.
\end{align*}
Using the time stationarity of the dynamics, we have
\[
\sup_{s \in [0,\infty)}|C_N(s)| \le  \frac{1}{N^d} E_{N,E}[ (\sum_{\mathbf{x}} j_{\mathbf{x},\mathbf{x}+\mathbf{e}_1}^a)^2].
\]

Recall that $\sum_{\mathbf{x}} j_{\mathbf{x},\mathbf{x}+\mathbf{e}_1}^a=-\frac{1}{2}\sum_{j=1}^{d^*} \sum_{\mathbf{x}} (q_{\mathbf{x}+\mathbf{e}_1}^j-q_{\mathbf{x}-\mathbf{e}_1}^j)v_{\mathbf{x}}^j$. By the symmetry of the measure $\mu_{N,E}$ under the change of variables $v^j \to -v^j$ for any fixed $j$, we have
\begin{align*}
& E_{N,E}[ (\sum_{\mathbf{x}} j_{\mathbf{x},\mathbf{x}+\mathbf{e}_1}^a)^2]=\frac{1}{4} E_{N,E}[ (\sum_{j=1}^{d^*} \sum_{\mathbf{x}} (q_{\mathbf{x}+\mathbf{e}_1}^j-q_{\mathbf{x}-\mathbf{e}_1}^j)v_{\mathbf{x}}^j)^2] \\
&  = \frac{1}{4}\sum_{j=1}^{d^*}\sum_{\mathbf{x},\mathbf{y} } E_{N,E}[ (q_{\mathbf{x}+\mathbf{e}_1}^j-q_{\mathbf{x}-\mathbf{e}_1}^j)v_{\mathbf{x}}^j  (q_{\mathbf{y}+\mathbf{e}_1}^j-q_{\mathbf{y}-\mathbf{e}_1}^j)v_{\mathbf{y}}^j].
\end{align*}
Since the measure $\mu_{N,E}$ is translation invariant and symmetric with respect to the components $j=1,2,\dots,d^*$, the last expression is equal to
\begin{align*}
\frac{d^*N^d}{4}\sum_{\mathbf{x}} E_{N,E}[ (q_{\mathbf{x}+\mathbf{e}_1}^1-q_{\mathbf{x}-\mathbf{e}_1}^1)v_{\mathbf{x}}^1  (q_{\mathbf{e}_1}^1-q_{-\mathbf{e}_1}^1)v_{\mathbf{0}}^1]. 
\end{align*}
By the symmetry of the measure $\mu_{N,E}$ with respect to any permutation on the index set $\Z_N^d$ for $v$, 
\[E_{N,E}[ (q_{\mathbf{x}+\mathbf{e}_1}^1-q_{\mathbf{x}-\mathbf{e}_1}^1)v_{\mathbf{x}}^1  (q_{\mathbf{e}_1}^1-q_{-\mathbf{e}_1}^1)v_{\mathbf{0}}^1]= E_{N,E}[ (q_{\mathbf{x}+\mathbf{e}_1}^1-q_{\mathbf{x}-\mathbf{e}_1}^1)v_{\mathbf{y}}^1  (q_{\mathbf{e}_1}^1-q_{-\mathbf{e}_1}^1)v_{\mathbf{0}}^1]
\]
for any $\mathbf{x},\mathbf{y} \neq \mathbf{0}$. Noting $\sum_{\mathbf{x}} v_{\mathbf{x}}^1=0$ under $\mu_{N,E}$, we have
\[E_{N,E}[ (q_{\mathbf{x}+\mathbf{e}_1}^1-q_{\mathbf{x}-\mathbf{e}_1}^1)v_{\mathbf{x}}^1  (q_{\mathbf{e}_1}^1-q_{-\mathbf{e}_1}^1)v_{\mathbf{0}}^1]= -\frac{1}{N^d-1}E_{N,E}[ (q_{\mathbf{x}+\mathbf{e}_1}^1-q_{\mathbf{x}-\mathbf{e}_1}^1) (q_{\mathbf{e}_1}^1-q_{-\mathbf{e}_1}^1)(v_{\mathbf{0}}^1)^2]
\]
for any $\mathbf{x} \neq \mathbf{0}$. Therefore,
\begin{align*}
&\frac{4}{d^*N^d}E_{N,E}[ (\sum_{\mathbf{x}} j_{\mathbf{x},\mathbf{x}+\mathbf{e}_1}^a)^2] \\
&= E_{N,E}[ (q_{\mathbf{e}_1}^1-q_{-\mathbf{e}_1}^1)^2(v_{\mathbf{0}}^1)^2] -\frac{1}{N^d-1} \sum_{ \mathbf{x} \neq \mathbf{0}}E_{N,E}[ (q_{\mathbf{x}+\mathbf{e}_1}^1-q_{\mathbf{x}-\mathbf{e}_1}^1) (q_{\mathbf{e}_1}^1-q_{-\mathbf{e}_1}^1)(v_{\mathbf{0}}^1)^2]  \\
&=\frac{N^d}{N^d-1}  E_{N,E}[ (q_{\mathbf{e}_1}^1-q_{-\mathbf{e}_1}^1)^2(v_{\mathbf{0}}^1)^2]. %&= d^{*}N^d E_N[ ((q_{\mathbf{e}_1}^1-q_{-\mathbf{e}_1}^1)v_{\mathbf{0}}^1)^2].
\end{align*}
For the last equality, we use the fact that $\sum_{ \mathbf{x} \neq \mathbf{0}}q_{\mathbf{x}\pm \mathbf{e}_1}^j=-q_{\pm\mathbf{e}_1}^j$ under $\mu_{N,E}$. Hence, the equivalence of ensemble (Lemma \ref{equivalence} (iv)) shows $ E_{N,E}[ (\sum_{\mathbf{x}} j_{\mathbf{x},\mathbf{x}+\mathbf{e}_1}^a)^2]$ is of order $N^d$ and the set of functions $\{C_N : [0,\infty) \to \R\}_N$ is uniformly bounded.

Next, we show the equicontinuity. Note that
\begin{align*}
& L (\sum_{\mathbf{x}} j_{\mathbf{x},\mathbf{x}+\mathbf{e}_1}^a) =-\frac{1}{2}L ( \sum_{j=1}^{d^*} \sum_{\mathbf{x}} (q_{\mathbf{x}+\mathbf{e}_1}^j-q_{\mathbf{x}-\mathbf{e}_1}^j) v_{\mathbf{x}}^j) \\
& = -\frac{1}{2}\sum_{j=1}^{d^*}  \sum_{\mathbf{x}} \{ (v_{\mathbf{x}+\mathbf{e}_1}^j-v_{\mathbf{x}-\mathbf{e}_1}^j) v_{\mathbf{x}}^j +  (q_{\mathbf{x}+\mathbf{e}_1}^j-q_{\mathbf{x}-\mathbf{e}_1}^j)( \Delta q_{\mathbf{x}}^j +\delta_{j,1} B v_{\mathbf{x}}^2-\delta_{j,2} B v_{\mathbf{x}}^1) +\\
& \quad \ga (q_{\mathbf{x}+\mathbf{e}_1}^j-q_{\mathbf{x}-\mathbf{e}_1}^j) \Delta  v_{\mathbf{x}}^j \} \\
& = -\frac{1}{2}\sum_{j=1}^{d^*}  \sum_{\mathbf{x}}  (q_{\mathbf{x}+\mathbf{e}_1}^j-q_{\mathbf{x}-\mathbf{e}_1}^j)( \delta_{j,1} B v_{\mathbf{x}}^2-\delta_{j,2} B v_{\mathbf{x}}^1 +\ga \Delta  v_{\mathbf{x}}^j ).
\end{align*}
We denote by $W$ the last term, namely
\[
W=-\frac{1}{2}\sum_{j=1}^{d^*}  \sum_{\mathbf{x}}  (q_{\mathbf{x}+\mathbf{e}_1}^j-q_{\mathbf{x}-\mathbf{e}_1}^j)( \delta_{j,1} B v_{\mathbf{x}}^2-\delta_{j,2} B v_{\mathbf{x}}^1 +\ga \Delta  v_{\mathbf{x}}^j ).
\]
Then,
\begin{align*}
\sum_{\mathbf{x}} j_{\mathbf{x},\mathbf{x}+\mathbf{e}_1}^a(t) - \sum_{\mathbf{x}} j_{\mathbf{x},\mathbf{x}+\mathbf{e}_1}^a(0) =\int_0^t W(s) ds + m_t 
\end{align*}
where $m_t$ is a martingale. Therefore, 
\begin{align*}
&| C_N(t)-C_N(s)|  \\
& =|\frac{1}{N^d}\mathbb{E}_{N,E}[ \sum_{\mathbf{x}} j^a_{\mathbf{x},\mathbf{x}+\mathbf{e}_1}(t) \sum_{\mathbf{y}} j^a_{\mathbf{y},\mathbf{y}+\mathbf{e}_1}(0) ] - \frac{1}{N^d} \mathbb{E}_{N,E}[ \sum_{\mathbf{x}} j^a_{\mathbf{x},\mathbf{x}+\mathbf{e}_1}(s) \sum_{\mathbf{y}} j^a_{\mathbf{y},\mathbf{y}+\mathbf{e}_1}(0) ] | \\
& \le \frac{1}{N^d}  \int_s^t \mathbb{E}_{N,E}[ |W(r)  \sum_{\mathbf{y}} j^a_{\mathbf{y},\mathbf{y}+\mathbf{e}_1}(0)|]  dr.
\end{align*}
By the Cauchy-Schwarz inequality and the time stationarity of the dynamics,
\begin{align*}
 \mathbb{E}_{N,E}[ |W(r)  \sum_{\mathbf{y}} j^a_{\mathbf{y},\mathbf{y}+\mathbf{e}_1}(0)|] 
 &\le \sqrt{ \mathbb{E}_{N,E}[ |W(r)|^2] \mathbb{E}_{N,E}[ |\sum_{\mathbf{y}} j^a_{\mathbf{y},\mathbf{y}+\mathbf{e}_1}(0)|^2]} \\
& = \sqrt{ E_{N,E}[ W^2] E_{N,E}[ (\sum_{\mathbf{y}} j^a_{\mathbf{y},\mathbf{y}+\mathbf{e}_1})^2]}.
\end{align*}
By the same argument for the term $E_{N,E}[ (\sum_{\mathbf{x}} j^a_{\mathbf{x},\mathbf{x}+\mathbf{e}_1})^2]$, we can show that $E_{N,E}[ W^2]$ is of order $N^d$. Hence, the term $\frac{1}{N^d} \sqrt{ E_{N,E}[ W^2] E_{N,E}[ (\sum_{\mathbf{y}} j^a_{\mathbf{y},\mathbf{y}+\mathbf{e}_1})^2]}$ is uniformly bounded in $N$ and so the set of functions $\{C_N : [0,\infty) \to \R\}_N$ is equicontinuous.
\end{proof}

Next, we show that the Laplace transform of $C_N$, which we denote by $\tilde{C}_N$ converges pointwise to a function $\tilde{C}_{\infty}$ on $(0,\infty)$. For each $\la >0$, let $\tilde{C}_N(\lambda)=\int_0^{\infty} C_N(s) \exp(-\la s)ds$ and $\tilde{C}_{\infty}: (0,\infty) \to \R$ be a function given by
\begin{align}
& \tilde{C}_{\infty}(\la) =\frac{2E^2}{{d^*}^2} \int_{[0,1]^d} \frac{\sin^2(2\pi \theta^1)}{\omega_{\theta}^2} \frac{(\lambda+ \ga \omega_{\theta}^2)(\la^2+2\la\ga\omega_{\theta}^2+4\omega_{\theta}^2)}{(\lambda+ \ga \omega_{\theta}^2)^2(\la^2+2\la\ga\omega_{\theta}^2+4\omega_{\theta}^2)+B^2\lambda(\lambda+2\ga\omega_{\theta}^2)}d\theta \nonumber \\ 
& + (d^*-2) \frac{E^2}{{d^*}^2} \int_{[0,1]^d} \frac{\sin^2(2\pi \theta^1)}{\omega_{\theta}^2}  \frac{1}{\la+\ga\omega_{\theta}^2}d\theta \label{eq:laplace1}
\end{align}
where $\omega_{\theta}^2=4\sum_{a=1}^d \sin^2 (\pi \theta^a)$.

\begin{prop}\label{prop:resolvent1}
For any $\lambda>0$, 
\[ \lim_{ N \to \infty}\tilde{C}_N(\lambda)=\tilde{C}_{\infty}(\la).
\]
\end{prop}

We give a proof of this proposition in the next subsection.

From a general observation of the convergence of functions, Lemma \ref{lem:uniformity1} and Proposition \ref{prop:resolvent1} imply the existence of the limiting function $C_{\infty}$ on $[0,\infty)$ whose Laplace transform is $\tilde{C}_{\infty}$ (we give a rigorous argument for this in Proposition \ref{prop:ascoli} in Appendix \ref{app:convergence} for completeness). 

Finally, we study the asymptotic behavior of the inverse Laplace transform of $\tilde{C}_{\infty}$ in Subsection \ref{propsec12} which completes the proof of Proposition \ref{propu1}.

\begin{rem}
By taking $B=0$ in (\ref{eq:laplace1}), we obtain 
\[
\int_0^{\infty} C_{\infty}(s) \exp(-\la s)ds=\frac{E^2}{d^*} \int_{[0,1]^d} \frac{\sin^2(2\pi \theta^1)}{\omega_{\theta}^2}  \frac{1}{\la+\ga\omega_{\theta}^2}d\theta
\]
which recovers the result (48) of \cite{BBO} (where the stochastic noise is different, hence the constant does not coincide).
\end{rem}

\subsection{Proof of Proposition \ref{prop:resolvent1}}\label{propsec11}

We follow the strategy of \cite{BBO} again, for which an explicit expression of a resolvent equation is crucial.

For each $\la>0$, let $u_{\la,N}$ be the solution of the resolvent equation  
\[(\la - L) u_{\la,N}=\sum_{\mathbf{x}} j^a_{\mathbf{x},\mathbf{x}+\mathbf{e}_1}.
\]
Since $L$ is the generator of the process, we have $\tilde{C}_N(\lambda)=E_{N,E}[ u_{\la,N} \ j^a_{\mathbf{0},\mathbf{e}_1}]$. 

%Let $h(\mathbf{z})=\frac{1}{2}(\delta_{\mathbf{z}+\mathbf{e}_1}-\delta_{\mathbf{z}-\mathbf{e}_1})$ and a set of functions $(g^1_{\la,N},g^2_{\la,N},g^3_{\la,N},g^4_{\la,N})$ on $\Z_N^d$ be the solution of (\ref{eq:forg1-4}) for $\bar{h}$ and $g_{\la,N}$ be the solution of (\ref{eq:forg}) for .

From Lemmas \ref{lem:resolvent1} and \ref{lem:resolvent2} in Appendix \ref{app:resolvent}, $u_{\la,N}$ is explicitly given by 
\begin{align*}
u_{\la,N} & =\sum_{\mathbf{x},\mathbf{y} \in \Z_N^d} (g^1_{\la,N}(\mathbf{x}-\mathbf{y})q_{\mathbf{x}}^1 q_{\mathbf{y}}^2+g^2_{\la,N}(\mathbf{x}-\mathbf{y})(q_{\mathbf{x}}^1 v_{\mathbf{y}}^1+q_{\mathbf{x}}^2 v_{\mathbf{y}}^2) \\
& +g^3_{\la,N}(\mathbf{x}-\mathbf{y})(q_{\mathbf{x}}^1 v_{\mathbf{y}}^2-q_{\mathbf{x}}^2 v_{\mathbf{y}}^1) +g^4_{\la,N}(\mathbf{x}-\mathbf{y})v_{\mathbf{x}}^1 v_{\mathbf{y}}^2) \\
& + \sum_{j=3}^{d^*}\sum_{\mathbf{x},\mathbf{y}\in \Z_N^d} g_{\la,N}(\mathbf{x}-\mathbf{y})q_{\mathbf{x}}^j v_{\mathbf{y}}^j
\end{align*}
with the solutions of (\ref{eq:forg}) and (\ref{eq:forg1-4}).
 By the symmetry of the measure $\mu_{N,E}$ under the change of variables $q^j \to -q^j$ or $v^j \to -v^j$ for any fixed $j$, we have
\begin{align*}
\tilde{C}_N(\lambda) = & E_N[ u_{\la,N} \ j^a_{\mathbf{0},\mathbf{e}_1}]= \sum_{\mathbf{x},\mathbf{y} \in \Z_N^d} g^2_{\la,N}(\mathbf{x}-\mathbf{y})E_N[(q_{\mathbf{x}}^1 v_{\mathbf{y}}^1+q_{\mathbf{x}}^2 v_{\mathbf{y}}^2)  \ j^a_{\mathbf{0},\mathbf{e}_1}] \\
& + \sum_{j=3}^{d^*}\sum_{\mathbf{x},\mathbf{y}\in \Z_N^d} g_{\la,N}(\mathbf{x}-\mathbf{y}) E_N[  q_{\mathbf{x}}^j v_{\mathbf{y}}^j \ j^a_{\mathbf{0},\mathbf{e}_1}].  \\
\end{align*}

% For any $\mathbf{x},\mathbf{y},\mathbf{z},\mathbf{w} \in \Z_N^d$, $E_{N}[q_{\mathbf{x}}^j v_{\mathbf{y}}^j q_{\mathbf{z}}^m v_{\mathbf{w}}^m]=0$ if $j \neq m$ since the measure $\mu_{N}$ is invariant under the change of variables $q^j \to -q^j$ or $v^j \to -v^j$ for any $j$. In the same way, we have 
%\[
%E_{N}[q_{\mathbf{x}}^1 q_{\mathbf{y}}^2 q_{\mathbf{z}}^j v_{\mathbf{w}}^j]=E_{N}[q_{\mathbf{x}}^1 v_{\mathbf{y}}^2 q_{\mathbf{z}}^j v_{\mathbf{w}}^j]=E_{N}[q_{\mathbf{x}}^2 v_{\mathbf{y}}^1 q_{\mathbf{z}}^j v_{\mathbf{w}}^j]
%=E_{N}[v_{\mathbf{x}}^1 v_{\mathbf{y}}^2 q_{\mathbf{z}}^j v_{\mathbf{w}}^j]=0
%\]
%for any $j=1,2,\dots,d^*$ and $\mathbf{x},\mathbf{y},\mathbf{z},\mathbf{w} \in \Z_N^d$.
Moreover, by the symmetry with respect to the components $j=1,2,\dots, d^*$, we have
\begin{align*}
& E_{N,E}[q_{\mathbf{x}}^j v_{\mathbf{y}}^j  \ j^a_{\mathbf{0},\mathbf{e}_1}] = -\frac{1}{2}E_{N,E}[q_{\mathbf{x}}^1 v_{\mathbf{y}}^1(q_{\mathbf{e}_1}^1-q_{\mathbf{0}}^1)(v_{\mathbf{e}_1}^1+v_{\mathbf{0}}^1)]
\end{align*}
for any $j$. Therefore, 
\begin{align}
\tilde{C}_N(\lambda) & = -\frac{1}{2} \sum_{\mathbf{x},\mathbf{y} \in \Z_N^d} (2g^2_{\la,N}\left(\mathbf{x}-\mathbf{y}) +(d^*-2) g_{\la,N}(\mathbf{x}-\mathbf{y})\right) E_{N,E}[q_{\mathbf{x}}^1 v_{\mathbf{y}}^1(q_{\mathbf{e}_1}^1-q_{\mathbf{0}}^1)(v_{\mathbf{e}_1}^1+v_{\mathbf{0}}^1)] \nonumber \\
& = -\frac{1}{2} \sum_{\mathbf{x},\mathbf{y} \in \Z_N^d} (2g^2_{\la,N}\left(\mathbf{x}-\mathbf{y}) +(d^*-2) g_{\la,N}(\mathbf{x}-\mathbf{y})\right) E_{N,E}[q_{\mathbf{x}}^1 v_{\mathbf{y}}^1(q_{\mathbf{e}_1}^1-q_{\mathbf{e}_{-1}}^1)v_{\mathbf{0}}^1]. \label{eq:resolvent1}
\end{align}
For the second equality, we use the fact that the sum is taken over all $\mathbf{x},\mathbf{y} \in \Z_N^d$.

By the symmetry of the measure $\mu_{N,E}$ with respect to any permutation on the index set $\Z_N^d$ for $v$, \[E_{N,E}[q_{\mathbf{x}}^1 v_{\mathbf{y}}^1(q_{\mathbf{e}_1}^1-q_{-\mathbf{e}_1}^1)v_{\mathbf{0}}^1]= E_{N,E}[q_{\mathbf{x}}^1 v_{\mathbf{y}'}^1(q_{\mathbf{e}_1}^1-q_{-\mathbf{e}_1}^1)v_{\mathbf{0}}^1]
\]
for any $\mathbf{y} \neq \mathbf{0}$ and $\mathbf{y}' \neq \mathbf{0}$. Noting $\sum_{\mathbf{x}} v_{\mathbf{x}}^1=0$, we have
\[
E_{N,E}[q_{\mathbf{x}}^1 v_{\mathbf{y}}^1(q_{\mathbf{e}_1}^1-q_{-\mathbf{e}_1}^1)v_{\mathbf{0}}^1]=-\frac{1}{N^d-1}E_{N,E}[q_{\mathbf{x}}^1 (q_{\mathbf{e}_1}^1-q_{-\mathbf{e}_1}^1)(v_{\mathbf{0}}^1)^2]
\] for any $\mathbf{y} \neq \mathbf{0}$, and so the last term of (\ref{eq:resolvent1}) is rewritten as
\begin{align*}
& -\frac{1}{2} \sum_{\mathbf{x} \in \Z_N^d}\big( 2g^2_{\la,N}(\mathbf{x})+(d^*-2)g_{\la,N}(\mathbf{x})\big)E_{N,E}[q_{\mathbf{x}}^1 (q_{\mathbf{e}_1}^1-q_{-\mathbf{e}_1}^1)(v_{\mathbf{0}}^1)^2] \\
& +\frac{1}{2(N^d-1)} \sum_{\mathbf{x},\mathbf{y} \in \Z_N^d,\mathbf{y} \neq \mathbf{0} }\big( 2g^2_{\la,N}(\mathbf{x}-\mathbf{y})+(d^*-2)g_{\la,N}(\mathbf{x}-\mathbf{y})\big) E_{N,E}[q_{\mathbf{x}}^1 (q_{\mathbf{e}_1}^1-q_{-\mathbf{e}_1}^1)(v_{\mathbf{0}}^1)^2]. \\
\end{align*}
Since $g^2_{\la,N}(\mathbf{x})=-g^2_{\la,N}(-\mathbf{x})$, $g_{\la,N}(\mathbf{x})=-g_{\la,N}(-\mathbf{x})$ and in particular $ \sum_{\mathbf{x}} g^2_{\la,N}(\mathbf{x})= \sum_{\mathbf{x}} g_{\la,N}(\mathbf{x})=0$,
 we finally obtain that 
\begin{align*}
\tilde{C}_N(\lambda) & = -\frac{1}{2}\frac{N^d}{N^d-1} \sum_{\mathbf{x} \in \Z_N^d}\big(2g^2_{\la,N}(\mathbf{x})+(d^*-2)g_{\la,N}(\mathbf{x})\big)E_{N,E}[q_{\mathbf{x}}^1 (q_{\mathbf{e}_1}^1-q_{-\mathbf{e}_1}^1)(v_{\mathbf{0}}^1)^2]\\
&=-\frac{1}{4}\frac{N^d}{N^d-1} \sum_{\mathbf{x} \in \Z_N^d}\big(2g^2_{\la,N}(\mathbf{x})+(d^*-2)g_{\la,N}(\mathbf{x})\big)E_{N,E}[(q_{\mathbf{x}}^1 - q_{-\mathbf{x}}^1)(q_{\mathbf{e}_1}^1-q_{-\mathbf{e}_1}^1)(v_{\mathbf{0}}^1)^2].
\end{align*}

Let us define the Fourier transform $\hat{f}(\mathbf{\xi})$, $\mathbf{\xi} \in \Z^d_N$ of a function $f$ on $\Z^d_N$ by 
\[
\hat{f}(\mathbf{\xi})=\sum_{\mathbf{z} \in \Z^d_N } f(\mathbf{z})e^{-2\pi i \frac{\mathbf{\xi}}{N} \cdot \mathbf{z}}.
\]
The inverse Fourier transform is given by
\[
f(\mathbf{z})=\frac{1}{N^d}\sum_{\mathbf{\xi} \in \Z^d_N } \hat{f}(\mathbf{\xi})e^{2\pi i \frac{\mathbf{\xi}}{N} \cdot \mathbf{z}}.
\]
%and $\tilde{v}(\theta):=\sum_{\mathbf{z} \in \Z^d } v(\mathbf{z})e^{2\pi i \theta \cdot z}$ for $v : \Z^d \to \R$. $\tilde{g^2}_{\la}(\theta), \tilde{\Gamma}(\theta)$ are same for $\theta \in [0,1]^d$. The inverse transform is given by 

From Lemma \ref{equivalence} (iv),
we have
\begin{align*}
& E_{N,E}[(q_{\mathbf{x}}^1 - q_{-\mathbf{x}}^1)(q_{\mathbf{e}_1}^1-q_{-\mathbf{e}_1}^1)(v_{\mathbf{0}}^1)^2] \\
&= \frac{1}{N^{d}}\left(\frac{E^2}{{d^*}^2}+O(N^{-d})\right) \sum_{\mathbf{\xi} \in \Z^d_N, \mathbf{\xi} \neq \mathbf{0}}\frac{\sin(2\pi  \frac{\mathbf{\xi}}{N} \cdot \mathbf{x})\sin(2\pi  \frac{\mathbf{\xi}}{N} \cdot \mathbf{e}_1) }{\sum_{a=1}^d\sin(\frac{\pi \mathbf{\xi}^a}{N})^2}.
\end{align*}
Since
\begin{align*}
g_{\la,N}(\mathbf{x})-g_{\la,N}(-\mathbf{x})= \frac{2i}{N^d}\sum_{\mathbf{\xi} \in \Z^d_N } \sin(2\pi  \frac{\mathbf{\xi}}{N} \cdot \mathbf{x}) \hat{g}_{\la,N} (\mathbf{\xi}),
\end{align*}
we have
\begin{align*}
& \sum_{\mathbf{x} \in \Z_N^d} g_{\la,N}(\mathbf{x})E_{N,E}[(q_{\mathbf{x}}^1 - q_{-\mathbf{x}}^1)(q_{\mathbf{e}_1}^1-q_{-\mathbf{e}_1}^1)(v_{\mathbf{0}}^1)^2]\\
&=\frac{1}{2} \sum_{\mathbf{x} \in \Z_N^d} (g_{\la,N}(\mathbf{x})-g_{\la,N}(-\mathbf{x}))E_{N,E}[(q_{\mathbf{x}}^1 - q_{-\mathbf{x}}^1)(q_{\mathbf{e}_1}^1-q_{-\mathbf{e}_1}^1)(v_{\mathbf{0}}^1)^2]\\
&= \sum_{\mathbf{x} \in \Z_N^d}\frac{i}{N^d}\sum_{\mathbf{\xi} \in \Z^d_N } \sin(2\pi  \frac{\mathbf{\xi}}{N} \cdot \mathbf{x}) \hat{g}_{\la,N}(\mathbf{\xi})\frac{1}{N^{d}} \left(\frac{E^2}{{d^*}^2}+O(N^{-d})\right) \sum_{\mathbf{\zeta} \in \Z^d_N, \mathbf{\zeta} \neq \mathbf{0} }\frac{\sin(2\pi  \frac{\mathbf{\zeta}}{N} \cdot \mathbf{x})\sin(2\pi  \frac{\mathbf{\zeta}}{N} \cdot \mathbf{e}_1) }{\sum_{a=1}^d\sin(\frac{\pi \mathbf{\zeta}^{a}}{N})^2}\\
&= \frac{i}{N^{d}} \left(\frac{E^2}{{d^*}^2}+O(N^{-d})\right) \sum_{\mathbf{\xi} \in \Z^d_N,  \mathbf{\xi} \neq \mathbf{0} } \hat{g}_{\la,N}(\mathbf{\xi}) \frac{\sin(2\pi  \frac{\mathbf{\xi}}{N} \cdot \mathbf{e}_1) }{\sum_{a=1}^d\sin(\frac{\pi \mathbf{\xi}^{a}}{N})^2}.
\end{align*}
At the last equality, we use the fact
\[
\sum_{\mathbf{x} \in \Z_N^d} \sin(2\pi  \frac{\mathbf{\xi}}{N} \cdot \mathbf{x})\sin(2\pi  \frac{\mathbf{\zeta}}{N} \cdot \mathbf{x}) =\frac{N^d}{2}(\delta_{\mathbf{\xi},\mathbf{\zeta}}-\delta_{\mathbf{\xi},-\mathbf{\zeta}}).
\]
On the other hand, by the direct computations $\hat{g}_{\la,N}(\mathbf{\xi})=\hat{g}_{\la}(\frac{ \mathbf{\xi}}{N})$ where $\hat{g}_{\la}(\theta)=\frac{i\sin(2\pi \theta^1)}{\la+\ga \omega_{\theta}^2}$. In this way, we have
\[
 \sum_{\mathbf{x} \in \Z_N^d} g_{\la,N}(\mathbf{x})E_{N,E}[(q_{\mathbf{x}}^1 - q_{-\mathbf{x}}^1)(q_{\mathbf{e}_1}^1-q_{-\mathbf{e}_1}^1)(v_{\mathbf{0}}^1)^2] 
 \to  \frac{-4E^2}{{d^*}^2} \int_{[0,1]^d} \frac{\sin^2(2\pi \theta^1)}{\omega_{\theta}^2}  \frac{1}{\la+\ga\omega_{\theta}^2}d\theta
 \]
as $N \to \infty$ since the function $\frac{\sin^2(2\pi \theta^1)}{\omega_{\theta}^2}  \frac{1}{\la+\ga\omega_{\theta}^2}$ is continuous except at the boundary and uniformly bounded on $[0,1]^d$.

In the same manner, we will have the following convergence
\[
 \sum_{\mathbf{x} \in \Z_N^d} g^2_{\la,N}(\mathbf{x})E_{N,E}[(q_{\mathbf{x}}^1 - q_{-\mathbf{x}}^1)(q_{\mathbf{e}_1}^1-q_{-\mathbf{e}_1}^1)(v_{\mathbf{0}}^1)^2] 
 \to  \frac{4E^2 i}{{d^*}^2} \int_{[0,1]^d} \hat{g}_{\lambda}^2(\theta) \frac{\sin(2\pi \theta^1)}{\omega_{\theta}^2} d\theta
 \]
if $\hat{g}^2_{\la,N}(\mathbf{\xi})= \hat{g}^2_{\la}(\frac{ \mathbf{\xi}}{N})$ for some function $\hat{g}^2_{\la}$ and $\hat{g}_{\la}^2(\theta) \frac{\sin(2\pi \theta^1)}{\omega_{\theta}^2}$ behaves well at the boundary. We prove that this is the case in the rest of this subsection.

From the explicit expression (\ref{eq:forg1-4}), it is easy to see that $\hat{g}^2_{\la,N}(\mathbf{\xi})= \hat{g}^2_{\la}(\frac{ \mathbf{\xi}}{N})$ with $(\hat{g}^1_{\la},\hat{g}^2_{\la},\hat{g}^3_{\la},\hat{g}^4_{\la})$ given as the solution of the following linear equation:
\begin{equation*}
\left(
\begin{array}{cccc}
\la & 0 & 2\omega_{\theta}^2 &  0 \\
0 & \lambda+ \ga \omega_{\theta}^2 &  B &  0  \\
-1 & -B & \lambda+ \ga \omega_{\theta}^2   &   \omega_{\theta}^2 \\
0 & 0 & -2 &  \lambda+ 2 \ga \omega_{\theta}^2  
\end{array}
\right)
\left(
\begin{array}{c}
\hat{g}^1_{\la} \\
\hat{g}^2_{\la}  \\
\hat{g}^3_{\la} \\
\hat{g}^4_{\la} 
\end{array}
\right)
=
\left(
\begin{array}{c}
0 \\
i\sin(2\pi \theta^1)  \\
0 \\
0
\end{array}
\right).
\end{equation*}
Its explicit solution is 
\begin{align*}
(\hat{g}^1_{\la},\hat{g}^2_{\la},\hat{g}^3_{\la},\hat{g}^4_{\la})=\frac{i\sin(2\pi \theta^1)}{Q(\lambda)}\left(-2B\omega_{\theta}^2 (\la+2\ga \omega_{\theta}^2 ), P(\lambda),B\lambda(\lambda+2\ga\omega_{\theta}^2 ),2B\lambda\right)
\end{align*}
with $P(\la)=(\lambda+ \ga \omega_{\theta}^2)(\la^2+2\la\ga\omega_{\theta}^2+4\omega_{\theta}^2)$ and $Q(\la)=(\lambda+ \ga \omega_{\theta}^2)^2(\la^2+2\la\ga\omega_{\theta}^2+4\omega_{\theta}^2)+B^2\lambda(\lambda+2\ga\omega_{\theta}^2)$.
In particular, 
\[
\hat{g}_{\la}^2(\theta) \frac{\sin(2\pi \theta^1)}{\omega_{\theta}^2}=\frac{i\sin^2(2\pi \theta^1)}{\omega_{\theta}^2} \frac{P(\lambda)}{Q(\lambda)}
\]
is continuous except at the boundary and uniformly bounded on $[0,1]^d$. Therefore, we have 
\[
 \sum_{\mathbf{x} \in \Z_N^d} g^2_{\la,N}(\mathbf{x})E_{N,E}[(q_{\mathbf{x}}^1 - q_{-\mathbf{x}}^1)(q_{\mathbf{e}_1}^1-q_{-\mathbf{e}_1}^1)(v_{\mathbf{0}}^1)^2] 
 \to  \frac{-4E^2}{{d^*}^2} \int_{[0,1]^d} \frac{\sin^2(2\pi \theta^1)}{\omega_{\theta}^2} \frac{P(\lambda)}{Q(\lambda)}d\theta
 \]
as $N \to \infty$.

\subsection{Analysis of the inverse Laplace transform}\label{propsec12} In this section, we study the inverse Laplace transform of (\ref{eq:laplace1}) in detail and give the asymptotic behavior of $C(s)$. 

First, we study the term
\[
C_4(t):=\mathcal{L}^{-1} \left[ \int_{[0,1]^d} \frac{\sin^2(2\pi \theta^1)}{\omega_{\theta}^2}  \frac{1}{\la+\ga\omega_{\theta}^2} d\theta \right] (t) = \int_{[0,1]^d} \frac{\sin^2(2\pi \theta^1)}{\omega_{\theta}^2} \exp(-t\ga\omega_{\theta}^2) d\theta 
\]
where $\mathcal{L}^{-1}$ represents the inverse Laplace transform operator.
The asymptotic behavior of this term is already studied in \cite{BBO} as $C_4(t) \sim t^{-d/2}$.

%Let $P(\la)=(\lambda+ \ga \omega_{\theta}^2)(\la^2+2\la\ga\omega_{\theta}^2+4\omega_{\theta}^2)$ and $Q(\la)=(\lambda+ \ga \omega_{\theta}^2)^2(\la^2+2\la\ga\omega_{\theta}^2+4\omega_{\theta}^2)+B^2\lambda(\lambda+2\ga\omega_{\theta}^2)$.
To analyze the remaining term, 
we first study the inverse Laplace transform of $\frac{P(\la)}{Q(\la)}$.
Let $\bar{\lambda}=\lambda+\ga \omega_{\theta}^2$. Then, $P(\la)=\bar{\lambda}(\bar{\lambda}^2-\ga^2\omega_{\theta}^4+4\omega_{\theta}^2)$ and 
$Q(\la)=\bar{\lambda}^4+(B^2-\ga^2\omega_{\theta}^4+4\omega_{\theta}^2)\bar{\lambda}^2-\ga^2B^2\omega_{\theta}^4$.

Define $\a_1(\theta)\ge 0,\a_2(\theta)\ge 0$ as the solution of
\begin{equation}\label{eqfora}
\begin{cases}
\a_1(\theta)^2-\a_2(\theta)^2=B^2-\ga^2\omega_{\theta}^4+4\omega_{\theta}^2\\
\a_1(\theta)^2\a_2(\theta)^2=\ga^2B^2\omega_{\theta}^4.
\end{cases}
\end{equation}
We emphasize that any pair of the form $(\pm \a_1(\theta), \pm \a_2(\theta))$ satisfies the equations (\ref{eqfora}), but we choose nonnegative solutions and denote them by $\a_1(\theta)$ and $\a_2(\theta)$ which are unique.
Then $Q(\la)$ is factorized as
$Q(\la)=(\bar{\lambda}^2+\a_1(\theta)^2)(\bar{\lambda}^2-\a_2(\theta)^2)$ and
\begin{equation*}
\frac{P(\la)}{Q(\la)}=\bar{\lambda}\frac{\bar{\lambda}^2-\ga^2\omega_{\theta}^4+4\omega_{\theta}^2}{(\bar{\lambda}^2+\a_1(\theta)^2)(\bar{\lambda}^2-\a_2(\theta)^2)}=\bar{\lambda}\left( \frac{\b_1(\theta)}{\bar{\lambda}^2+\a_1(\theta)^2} +\frac{2\b_2(\theta)}{\bar{\lambda}^2-\a_2(\theta)^2}\right)
\end{equation*}
where %$\b_1(\theta)+2\b_2(\theta)=1$, $-\b_1(\theta)\a_2(\theta)^2+2\b_2(\theta)\a_1(\theta)^2= -\ga^2\omega_{\theta}^4+4\omega_{\theta}^2$. More precisely, 
$\b_1(\theta)=\frac{\a_2(\theta)^2+B^2}{\a_1(\theta)^2+\a_2(\theta)^2}$ and $\b_2(\theta)=\frac{\a_1(\theta)^2-B^2}{2(\a_1(\theta)^2+\a_2(\theta)^2)}$.
Therefore, 
\begin{equation*}
\frac{P(\la)}{Q(\la)}=\b_1(\theta)\frac{\bar{\lambda}}{\bar{\lambda}^2+\a_1(\theta)^2}+\b_2(\theta)\frac{1}{\bar{\lambda}+\a_2(\theta)}+\b_2(\theta)\frac{1}{\bar{\lambda}-\a_2(\theta)}
\end{equation*}
and we can calculate the inverse Laplace transform explicitly as 
\begin{align*}
\mathcal{L}^{-1} \left[\frac{P(\la)}{Q(\la)} \right] (t) & =\b_1(\theta)\exp(-\ga\omega_{\theta}^2t)\cos(\a_1(\theta)t)+\\
& \b_2(\theta)\exp(-\ga\omega_{\theta}^2t-\a_2(\theta)t)+\b_2(\theta)\exp(-\ga\omega_{\theta}^2t+\a_2(\theta)t).
\end{align*}
Therefore, we have
\begin{align}
\mathcal{L}^{-1}& \left[\int_{[0,1]^d} \frac{\sin^2(2 \pi \theta^1)}{\omega_{\theta}^2}\frac{P(\la)}{Q(\la)} d\theta \right] (t) \nonumber \\
& =\int_{[0,1]^d} \frac{\sin^2(2 \pi \theta^1)}{\omega_{\theta}^2} \big ( \b_1(\theta)\exp(-\ga\omega_{\theta}^2t)\cos(\a_1(\theta)t)+ \nonumber \\
& \b_2(\theta)\exp(-\ga\omega_{\theta}^2t-\a_2(\theta)t)+\b_2(\theta)\exp(-\ga\omega_{\theta}^2t+\a_2(\theta)t) \big) d\theta. \label{eq:explicitlaplace}
\end{align}

By definition,
\begin{align*}
\a_1(\theta)^2=\frac{(B^2-\ga^2\omega_{\theta}^4+4\omega_{\theta}^2)+\sqrt{(B^2-\ga^2\omega_{\theta}^4+4\omega_{\theta}^2)^2+4\ga^2B^2\omega_{\theta}^4}}{2} \\
\a_2(\theta)^2=\frac{-(B^2-\ga^2\omega_{\theta}^4+4\omega_{\theta}^2)+\sqrt{(B^2-\ga^2\omega_{\theta}^4+4\omega_{\theta}^2)^2+4\ga^2B^2\omega_{\theta}^4}}{2}. \\
\end{align*}
Hence, $\a_1,\a_2,\b_1,\b_2$ are continuous functions on $[0,1]^d$. Therefore, standard analysis shows the behavior of the term (\ref{eq:explicitlaplace}) as $t$ goes to infinity is governed by the behavior of the functions $\a_1,\a_2$ and $\b_1,\b_2$ around the minimal value of $\ga\omega_{\theta}^2, \ga\omega_{\theta}^2+\a_2(\theta)$ and $\ga\omega_{\theta}^2-\a_2(\theta)$. By the explicit expression,
\begin{align}\label{eq:alpha2behave1}
\ga^2\omega_{\theta}^4-\a_2(\theta)^2=\frac{B^2+\ga^2\omega_{\theta}^4+4\omega_{\theta}^2}{2} \big(1-\sqrt{1-\frac{16\ga^2\omega_{\theta}^6}{(B^2+\ga^2\omega_{\theta}^4+4\omega_{\theta}^2)^2}} \big)  \ge 0.
\end{align}
Therefore, $\ga\omega_{\theta}^2, \ga\omega_{\theta}^2+\a_2(\theta)$ and $\ga\omega_{\theta}^2-\a_2(\theta) $ are $0$ if and only if $\omega_{\theta}^2=0$. By symmetry, we treat only the case $\theta = (0,0,\dots,0)$. 

We study the asymptotic behavior of the following three terms separately:
\begin{align*}
C_1(t):=\int_{[0,1]^d} \frac{\sin^2(2 \pi \theta^1)}{\omega_{\theta}^2}\b_1(\theta)\exp(-\ga\omega_{\theta}^2t)\cos(\a_1(\theta)t)d\theta, \\
C_2(t):=\int_{[0,1]^d} \frac{\sin^2(2 \pi \theta^1)}{\omega_{\theta}^2} \b_2(\theta)\exp(-\ga\omega_{\theta}^2t-\a_2(\theta)t) d\theta, \\
C_3(t):=\int_{[0,1]^d} \frac{\sin^2(2 \pi \theta^1)}{\omega_{\theta}^2} \b_2(\theta)\exp(-\ga\omega_{\theta}^2t+\a_2(\theta)t)  d\theta.
\end{align*}

Note that 
\begin{align*}
\a_1(\theta)^2 -B^2 & =2\omega_{\theta}^2 -\frac{\ga^2\omega_{\theta}^4}{2}+\frac{1}{2}(\sqrt{(B^2-\ga^2\omega_{\theta}^4+4\omega_{\theta}^2)^2+4\ga^2B^2\omega_{\theta}^4}-B^2)  \\
& = 2\omega_{\theta}^2 +\frac{B^2}{2}( \sqrt{1+8\frac{\omega_{\theta}^2}{B^2}+o(|\theta|^2)}-1) +o(|\theta|^2) =4\omega_{\theta}^2 +o(|\theta|^2) 
\end{align*}
as $|\theta| \to 0$. 
It implies $\b_2(\theta)= \frac{2}{B^2}|\theta|^2 + o(|\theta|^2)$ as $|\theta| \to 0$.

Then, since (\ref{eq:alpha2behave1}) implies $\a_2(\theta)= \ga |\theta|^2 + o(|\theta|^2)$ as $|\theta| \to 0$, we have 
\[
C_2(t) \sim \int_{[0,1]^d} \frac{|\theta^1|^2}{|\theta|^2} |\theta|^2\exp(-2 \ga|\theta|^2t) d\theta
= t^{-d/2-1}\int_{[0,\sqrt{t}]^d} u_1^2 \exp(-2\ga |u|^2) du \sim t^{-d/2-1}.
\]
Also, from (\ref{eq:alpha2behave1}), we have
\begin{align*}
\ga^2\omega_{\theta}^4-\a_2(\theta)^2=\frac{4\ga^2\omega_{\theta}^6}{B^2} + o(|\theta|^6), \quad
\ga\omega_{\theta}^2+\a_2(\theta)=2\ga\omega_{\theta}^2 + o(|\theta|^2),
\end{align*}
and so $\ga\omega_{\theta}^2- \a_2(\theta)= \frac{2\ga \omega_{\theta}^4}{B^2} + o(|\theta|^4)$. Then, 
\[
C_3(t) \sim \int_{[0,1]^d} \frac{|\theta^1|^2}{|\theta|^2} |\theta|^2\exp(-\frac{ 2\ga|\theta|^4}{B^2}t) d\theta
= t^{-d/4-1/2}\int_{[0,t^{1/4}]^d} u_1^2 \exp(-\frac{2\ga |u|^4}{B^2}) du \sim t^{-d/4-1/2}.
\]

Finally, we study the oscillation term $C_1(t)$. We estimate the time integral of this term, namely,
\begin{align*}
\int_0^T & (1-\frac{t}{T})\int_{[0,1]^d} \frac{\sin^2(2 \pi \theta^1)}{\omega_{\theta}^2} \b_1(\theta)\exp(-\ga\omega_{\theta}^2t)\cos(\a_1(\theta)t)d\theta dt \\
& =\int_{[0,1]^d}  \frac{\sin^2(2 \pi \theta^1)}{\omega_{\theta}^2}  \b_1(\theta)\int_0^T (1-\frac{t}{T}) \exp(-\ga\omega_{\theta}^2t)\cos(\a_1(\theta)t)dt d\theta. 
\end{align*}
By a direct computation,
\begin{align*}
& \int_0^T   \exp(-\ga\omega_{\theta}^2t)\cos(\a_1(\theta)t)dt \\
%& = \frac{\exp(-\ga\omega_{\theta}^2 T)(-\ga\omega_{\theta}^2 \cos(\a_1(\theta)T)+\a_1(\theta)\sin(\a_1(\theta)T))}{\ga^2\omega_{\theta}^4+\a_1(\theta)^2} - \frac{-\ga\omega_{\theta}^2}{\ga^2\omega_{\theta}^4+\a_1(\theta)^2} \\
&=\frac{\exp(-\ga\omega_{\theta}^2 T)(-\ga\omega_{\theta}^2 \cos(\a_1(\theta)T)+\a_1(\theta)\sin(\a_1(\theta)T))+\ga\omega_{\theta}^2}{\ga^2\omega_{\theta}^4+\a_1(\theta)^2}
\end{align*}
and
\begin{align*}
& \frac{1}{T} \int_0^T t\exp(-\ga\omega_{\theta}^2t)\cos(\a_1(\theta)t)dt \\
& = \frac{\exp(-\ga\omega_{\theta}^2 T)(-\ga\omega_{\theta}^2 \cos(\a_1(\theta)T)+\a_1(\theta)\sin(\a_1(\theta)T))}{\ga^2\omega_{\theta}^4+\a_1(\theta)^2} \\
& - \frac{\exp(-\ga\omega_{\theta}^2 T)((\ga^2\omega_{\theta}^4-\a_1(\theta)^2) \cos(\a_1(\theta)T)-2\ga\omega_{\theta}^2\a_1(\theta)\sin(\a_1(\theta)T))}{T(\ga^2\omega_{\theta}^4+\a_1(\theta)^2)^2} \\
& +  \frac{\ga^2\omega_{\theta}^4-\a_1(\theta)^2}{T(\ga^2\omega_{\theta}^4+\a_1(\theta)^2)^2}. 
\end{align*}
Then, we have 
\begin{align*}
& \limsup_{T \to \infty} |\int_{[0,1]^d}\frac{\sin^2(2 \pi \theta^1)}{\omega_{\theta}^2}  \b_1(\theta)\int_0^T (1-\frac{t}{T}) \exp(-\ga\omega_{\theta}^2t)\cos(\a_1(\theta)t)dt d\theta| \\
& \le  \int_{[0,1]^d} |\b_1(\theta)| \frac{\ga\omega_{\theta}^2}{\ga^2\omega_{\theta}^4+\a_1(\theta)^2}d\theta < \infty.
\end{align*}

\section{Uniform and alternate charge models under canonical measures}\label{canosec}

This section provides a detailed study of the one-dimensional chain of oscillators with uniform and alternate charges in two-dimensional space under canonical measures. In the same way as the last section, we start by recalling the description of the model in Subsection \ref{modelsec2}, and characterize conserved quantities and introduce canonical measures in Subsection \ref{meassec2}. The strategy of the proof of Theorem \ref{mr2} given in Subsection \ref{currentsec2} is essentially the same as that for Theorem \ref{mr1}, but we need one new and important step, which is given in Subsection \ref{reductionsec}. The main result in this subsection allows us to reduce the problem to solve the resolvent equation in $(\mathbf{r},(\mathbf{v})$ coordinate to that in $(\mathbf{q},\mathbf{v})$ coordinate. In the last two subsections, we give proofs of other ingredients of the proof of Theorem \ref{mr2}, mainly for the alternate charge model.

\subsection{Model}\label{modelsec2}

As already discussed in the Introduction, we only consider the case $d=1$ and $d^*=2$ in this section. Let $(\mathbf{r}_x(t),\mathbf{v}_x(t))_{x \in \Z_N}$ be a a Markov process on $\R^{4N}$ generated by $L_r^{(\#)}=A_r+BG_r^{(\#)}+\gamma S_r$ where the operators $A_r,G_r^{(\#)}$ and $S$ are given in (\ref{Sopc}) and (\ref{AGopc}) for $\#=0,i,ii$. We call these dynamics as the dynamics (0), (i) and (ii) respectively.

\subsection{Conserved quantities and canonical measures}\label{meassec2}

As observed in the Introduction, the total energy $\sum_{\mathbf{x} \in \Z_N}\E_{\mathbf{x}}$ is conserved under any of the dynamics (0), (i) and (ii). The total deformation $\sum_{x \in \Z_N}r_x^j$ for $j=1,2$ are also conserved by all of them. On the other hand, the total velocity $\sum_{x \in \Z_N}v_x^j$ for $j=1,2$ are conserved only by the dynamics (0). Note that there is no conserved quantity which corresponds to the pseudomomentum $\tilde{\mathbf{p}}$ since it is given in terms of $\mathbf{q}$. However, for the dynamics (ii), $\sum_{x \in \Z_N^e} (v_x^1+v_{x+1}^1+Br_x^2)$ and $\sum_{x \in \Z_N^e} (v_x^2+v_{x+1}^2-Br_x^1)$ are also conserved where $\Z_N^e:=\{ x \in \Z_N \ ; \ x \equiv 0 \mod  2\}$. 

By direct computations, one sees immediately that the dynamics (0), (i) and (ii) are all stationary under $\mu_{N,\beta}$.
More generally, for each $\beta>0$ and $\tau=(\tau_1,\tau_2) \in \R^2$,
\begin{align*}
\mu_{N,\b,\tau}(dr dp) & =\frac{1}{Z_{\b,\tau}^{4N}}\exp(-\beta (\sum_{x} \E_x+ \sum_{j=1}^2 \tau_j\sum_{x}r_x^j)) \Pi_{x \in \Z_N} \Pi_{j=1}^2 dr_x^j dv_x^j \\
\end{align*}
are stationary measures for the dynamics (0), (i) and (ii). We mainly study the case $\tau=0$, but give some discussion on the asymptotic behavior of the Green-Kubo integral in the case $\tau \neq 0$ at Remark \ref{generalmeas}.

\subsection{Instantaneous energy current correlation}\label{currentsec2}

Recall that $J_{x,x+1} ([0, t])$ is the total energy current between $x$ and $x+1$ up to
time $t$. Let $j_{x,x+1}$ be the instantaneous energy current given by $j_{x,x+1}=j^a_{x,x+1}+j^s_{x,x+1}$ with
\begin{align*}
j_{x,x+1}^a = -\frac{1}{2} \sum_{j=1}^{2} r_x^j(v_{x+1}^j+v_x^j), \quad
j_{x,x+1}^s  =- \ga \sum_{j=1}^{2} \left( \frac{(v_{x+1}^j)^2}{2}-\frac{(v_x^j)^2}{2} \right). 
\end{align*}
Note that these currents are common among the dynamics (0), (i) and (ii).

The total energy current is written as
\[
J_{x,x+1}([0,t])=\int_0^t j_{x,x+1}(s)ds +\sum_{j=1}^{2} \int_0^t \left( \frac{(v_{x+1}^j(s-))^2}{2}-\frac{(v_x^j(s-))^2}{2} \right)dM_{j,x,x+1}(s)
\]
where $M_{j,x,x+1}(t)=N_{j,x,x+1}^{\ga}(t)-\ga t$ and $\{N_{j,x,x+1}^{\ga}(t)\}_{j=1,2, \ x \in \Z_N}$ are $2 N$-independent poisson processes with intensity $\ga$.

We can apply the same argument in Subsection \ref{currentsec1} to obtain 
\begin{align*}
& \frac{\b^2}{8tN}\mathbb{E}_{N,\b}^{(\#)} [ \big( \sum_{x \in \Z_N}  J_{x,x+1}([0,t]) \big)^2]  =\frac{\b^2}{4} \int_0^t (1-\frac{s}{t}) D_N^{(\#)}(s) ds + \frac{\ga}{4}
\end{align*}
with $D_N^{(\#)}(s)=\mathbb{E}_{N,\b}^{(\#)}[  \sum_{x} j^a_{x,x+1}(s) j^a_{0,1}(0)  ]$ for $\#=0,i,ii$. Here, $\mathbb{E}_{N,\b}^{(\#)}$ is the expectation for the dynamics $(\#)$ starting from the canonical measure $\mu_{N,\b}$. Since the canonical measures is product, we do not need $o(N)$ correction term.

Now, the next proposition is suffice to prove Theorem \ref{mr2}.
\begin{prop}\label{prop:correlation2}
For each $\#=0,i,ii$, the sequence of functions $\{ D_N^{(\#)} :[0, \infty) \to \R \} $ converges compactly to a function $D_{\infty}^{(\#)}:[0,\infty) \to \R$ as $N \to \infty$, namely the convergence is uniform on each compact subset of $[0,\infty)$. Moreover,
\[
D_{\infty}^{(0)}(s) \sim s^{-1/2}
\]
as $s \to \infty$, 
\begin{align*}
D_{\infty}^{(i)}(s) =D_1(s)+D_2(s)+D_3(s)
\end{align*}
where $\displaystyle \limsup_{t \to \infty} \int_0^t(1-\frac{s}{t})D_1(s) < \infty$, $D_2(s) \sim s^{-3/2}$, $D_3(s) \sim s^{-3/4}$ as $s \to \infty$.
Also, if $\ga \le 1$,
\[
D_{\infty}^{(ii)}(s) \sim s^{-1/2}.
\]
\end{prop}

The main strategy of the proof of Proposition \ref{prop:correlation2} is the same as that for Proposition \ref{propu1}. 

For each $\la >0$ and $\#=0,i,ii$, let $\tilde{D}_N^{\#}(\lambda)=\int_0^{\infty} D_N^{(\#)}(s) \exp(-\la s)ds$ be the Laplace transform of $D_N^{(\#)}$ and $\tilde{D}_{\infty}^{(\#)}: (0,\infty) \to \R$ be functions defined by
\begin{align*}
\tilde{D}_{\infty}^{(0)}(\la) & =\frac{2}{\beta^2} \int_{[0,1]} \cos(\pi\theta)^2   \frac{1}{\la+\ga\omega_{\theta}^2}d\theta,   \\
\tilde{D}_{\infty}^{(i)}(\la) & =\frac{2}{\beta^2} \int_{[0,1]} \cos(\pi\theta)^2 \frac{(\lambda+ \ga \omega_{\theta}^2)(\la^2+2\la\ga\omega_{\theta}^2+4\omega_{\theta}^2)}{(\lambda+ \ga \omega_{\theta}^2)^2(\la^2+2\la\ga\omega_{\theta}^2+4\omega_{\theta}^2)+B^2\lambda(\lambda+2\ga\omega_{\theta}^2)}d\theta,  \\
\tilde{D}_{\infty}^{(ii)}(\la) & =\frac{2}{\beta^2}\int_{[0,1]} \cos(\pi\theta)^2 \frac{R(\la)}{S(\la)}d\theta 
\end{align*}
where $\omega_{\theta}^2=4\sin^2 (\pi \theta)$. $R(\la)$ and $S(\la)$ are polynomial functions of order $5$ and $6$ respectively, and their explicit expressions are given at (\ref{eq:FG}) where $\bar{\la}=\la+2\ga$.

We can follow the proof of Lemma \ref{lem:uniformity1} to prove the next lemma, so we omit the proof. 
\begin{lem}\label{lem:uniformity2}
For each $\#=0,i,ii$, the set of functions $\{D_N^{(\#)} : [0,\infty) \to \R\}_N$ is uniformly bounded and equicontinuous.
\end{lem}

We prove the next theorem in Subsection \ref{propsec21}. 
\begin{prop}\label{prop:resolvent2}
For any $\lambda>0$ and $\#=0,i,ii$, 
\[ \lim_{ N \to \infty}\tilde{D}^{(\#)}_N(\lambda)=\tilde{D}^{(\#)}_{\infty}(\la).
\]
\end{prop}
The detailed analysis of the inverse Laplace transform of $\tilde{D}^{(\#)}_{\infty}$ for $\#=0,i,ii$ completes the proof of Proposition \ref{prop:correlation2}. Since the case $\#=0$ and $i$ are studied in the last section, we only analyze the inverse Laplace transform of $\tilde{D}_{\infty}^{(ii)}$ in Subsection \ref{inverselaplace2}. 

\subsection{Reduction from the coordinate $(\mathbf{q},\mathbf{v})$ to $(\mathbf{r},\mathbf{v})$}\label{reductionsec}

In the proof of Proposition \ref{prop:resolvent2}, an explicit expression of the resolvent equation (\ref{eq:canonical}) plays an essential role. However, it is not a simple problem to solve the equation in the coordinate $(\mathbf{r},\mathbf{v})$. In this subsection, we introduce a technique to obtain this solution from the solution of the associated resolvent equation in the coordinate $(\mathbf{q},\mathbf{v})$.

We consider the change of variable $\Phi: (\mathbf{r},\mathbf{v}) \to (\mathbf{q},\mathbf{v})$ defined by
\begin{align*}
q^j_x=-\sum_{y=x}^{N} (r^j_y-\bar{r}^j), \quad \bar{r}^j=\frac{1}{N}\sum_{x=1}^N r_x^j
\end{align*}
for $x=1,2,\dots,N$ and define $q^j_x=q^j_{x+N}$ for any $x \in \Z_N$. Then, we have $q^{j}_{x+1}-q^j_x=r^j_x-\bar{r}^j$.

For each $\#=0,i,ii$, let $G^{(\#)}$ be an operator which is formally given as $G_r^{(\#)}$ but acting on functions $f (\mathbf{q},\mathbf{v})$ rather than $f (\mathbf{r},\mathbf{v})$. Define the operator $L^{(\#)}=A+BG^{(\#)}+\gamma S$ acting on functions $f (\mathbf{q},\mathbf{v}) \in C^1(\R^{4N})$ with the operators $A$ and $G$ defined at (\ref{AGop}). Note that $L^{(i)}=L$ where $L$ is the generator of the dynamics studied in Section \ref{microsec}.

\begin{prop}\label{reduction}
Suppose a pair of functions $F(\mathbf{q},\mathbf{v})$ and $H(\mathbf{q},\mathbf{v})$ satisfies $(\la - L^{(\#)}) F =H$ and $\sum_{x=1}^N \partial_{q_x^j}F  =0$ for $j=1,2$. Then $F_r(\mathbf{r},\mathbf{v})$ and $H_r(\mathbf{r},\mathbf{v})$ satisfies 
$(\la - L^{(\#)}_r) F_r =H_r$ where $F_r( \mathbf{r},\mathbf{v})= F(\Phi (\mathbf{r},\mathbf{v}))$ and $H_r( \mathbf{r},\mathbf{v})= H(\Phi (\mathbf{r},\mathbf{v}))$.
\end{prop}

\begin{proof}
Since $G^{(\#)}F_r (\mathbf{r},\mathbf{v})=(G^{(\#)}F)(\Phi(\mathbf{r},\mathbf{v}))$ and $S_r F_r( \mathbf{r},\mathbf{v})=(SF)(\Phi( \mathbf{r},\mathbf{v}))$, we only need to show that $A_rF_r (\mathbf{r},\mathbf{v})=(AF)(\Phi (\mathbf{r},\mathbf{v}))$. By definition, we have 
\[
q^j_{x+1}-2q^j_x+q^j_{x-1}=-\sum_{z=x+1}^{N} (r_z^j-\bar{r}^j) +2 \sum_{z=x}^{N} (r_z^j-\bar{r}^j)  -\sum_{z=x-1}^{N} (r_z^j-\bar{r}^j)= r^j_x-r^j_{x-1}.
\]
 Also, 
\[
\partial_{r_x^j}F_r(\mathbf{r},\mathbf{v})=\sum_{y=1}^N \frac{\partial q_y^j}{\partial r_x^j }(\partial_{q_y^j}F)(\Phi (\mathbf{r},\mathbf{v}))=\sum_{y=1}^N (-\mathbf{1}_{x \ge y} +\frac{1}{N})(\partial_{q_y^j}F)(\Phi (\mathbf{r},\mathbf{v}))
\]
and hence
\[
(\partial_{r_x^j}-\partial_{r_{x-1}^j})F_r (\mathbf{r},\mathbf{v})=- (\partial_{q_x^j}F)(\Phi (\mathbf{r},\mathbf{v}))
\]
for $x= 2,3,\dots,N$. On the other hand, by assumption,
\[
(\partial_{r_{1}^j}-\partial_{r_N^j})F_r (\mathbf{r},\mathbf{v})=\sum_{y=2}^N (\partial_{q_y^j}F)(\Phi (\mathbf{r},\mathbf{v})) = -(\partial_{q_1^j}F )(\Phi (\mathbf{r},\mathbf{v})).
\]
Therefore, $A_rF_r (\mathbf{r},\mathbf{v})=( A F)(\Phi (\mathbf{r},\mathbf{v}))$.
\end{proof}

\subsection{Proof of Proposition \ref{prop:resolvent2}}\label{propsec21}

For $\la>0$, let $v^{(\#)}_{\la,N}$ be the solution of the resolvent equation  
\begin{equation}\label{eq:canonical}
(\la - L_r^{(\#)}) v^{(\#)}_{\la,N}=\sum_{x} j_{x,x+1}^a
\end{equation}
for $\#=0,i,ii$. Then, $\tilde{D}^{(\#)}_N(\lambda)=E_{N,\b}[ v^{(\#)}_{\la,N} \ j^a_{0,1}]$. 

Unlike the microcanonical case, it is not easy to solve the equation (\ref{eq:canonical}) directly. So, we consider the associated resolvent equation for the coordinate $(\mathbf{q},\mathbf{v}) \in \R^{4N}$ and use its solution to obtain the solution of (\ref{eq:canonical}). Let $u^{(\#)}_{\la,N}$ be the solution of the resolvent equation 
\begin{equation}\label{eq:microcanonical}
(\la - L^{(\#)}) u^{(\#)}_{\la,N}=-\frac{1}{2}\sum_{x}\sum_{j=1}^2 (q_{x+1}^j-q_x^j)(v_{x+1}^j+v_x^j).
\end{equation}
%with $L^{(\#)}=A+BG^{(\#)}+\ga S$ where $A,S$ are the operators defined in (\ref{Sop}), (\ref{AGop}), and $G^{(\#)}=G_r^{(\#)}$ but they act on functions $f(\mathbf{q},\mathbf{v})$ instead of the functions $f(\mathbf{r},\mathbf{v})$.
For $\#=0,i,ii$, the explicit form of $u^{(\#)}_{\la,N}$ is given in Lemmas \ref{lem:resolvent1},\ref{lem:resolvent2} and \ref{lem:resolvent3} respectively.
From Proposition \ref{reduction} and Lemma \ref{lem:derivative0}, 
\[
(\la - L_r^{(\#)}) v^{(\#)}_{\la,N,*} = -\frac{1}{2}\sum_{x}\sum_{j=1}^2 (r_x^j-\bar{r}^j)(v_{x+1}^j+v_x^j) = \sum_{x} j_{x,x+1}^a +N \sum_{j=1}^2 \bar{r}^j\bar{v}^j
\]
where $v^{(\#)}_{\la,N,*}(r,v) = u^{(\#)}_{\la,N} (\Phi (r,v))$, $\bar{r}^j=\frac{1}{N}\sum_{x \in \Z_N} r_x^j$ and $\bar{v}^j=\frac{1}{N}\sum_{x \in \Z_N} v_x^j$.

Let $ v^{(\#)}_{\la,N,**}$ be the solution of the resolvent equation  
\[
(\la - L_r^{(\#)}) v^{(\#)}_{\la,N,**}=N \sum_{j=1}^2 \bar{r}^j\bar{v}^j.
\]
Then, we have $v^{(\#)}_{\la,N}= v^{(\#)}_{\la,N,*} - v^{(\#)}_{\la,N,**}$. 

By direct computations, we have
\begin{equation*}
v^{(\#)}_{\la,N,**} = 
\begin{cases}
 \frac{N}{\la} \sum_{j=1}^2 \bar{r}^j\bar{v}^j \quad \text{for} \quad \#=0, \\
 -N ( \bar{v}^1 \frac{\la \bar{r}^1-B\bar{r}^2}{\la^2+B^2}+ \bar{v}^2 \frac{B \bar{r}^1 + \la \bar{r}^2}{\la^2+B^2}) \quad \text{for} \quad \#=i, \\
 \end{cases}
\end{equation*}
and 
\begin{align*}
v^{(\#)}_{\la,N,**} = 
 -\frac{N}{\la(\la^2+4\ga \la +4+B^2)} \big( (\la^2+4\ga \la +4) \bar{r}^1 \bar{v}^1-B\la \bar{r}^2 \check{v}^1  \\
 +  B\la\bar{r}^1 \check{v}^2 +(\la^2+4\ga \la +4)\bar{r}^2\bar{v}^2 +2B (\bar{r}^1\check{r}^2- \bar{r}^2\check{r}^1) \big)
\end{align*}
for $\#=ii$ where $\check{r}^j=\frac{1}{N}\sum_{x \in \Z_N}(-1)^xr_x^j$ and $\check{v}^j=\frac{1}{N}\sum_{x \in \Z_N}(-1)^xv_x^j$. For any $j,k=1,2$, we have $E_{N,\b}[\bar{r}^j\bar{v}^k j^a_{0,1}]=-\delta_{j,k}E_{N,\b}[\bar{r}^jr^j_0]E_{N,\b}[\bar{v}^jv^j_0]=-\delta_{j,k}\frac{1}{N^2\beta^2}$. Hence, 
\[
\lim_{N \to \infty} E_{N,\b}[v^{(\#)}_{\la,N,**} \ j^a_{0,1}]=0
\]
for $\#=0,i$. Similarly, we can also check that 
\[
\lim_{N \to \infty} E_{N,\b}[v^{(ii)}_{\la,N,**} \ j^a_{0,1}]=0.
\]
Therefore, $E_{N,\b}[ v^{(\#)}_{\la,N,*} \ j^a_{0,1}]$ is the only term which contributes to the limit of $\tilde{D}^{(\#)}_N(\lambda)$.

\begin{rem}\label{generalmeas}
Under general canonical measures $\mu_{N,\b,\tau}$, 
\begin{align*}
\lim_{N \to \infty}E_{N,\b,\tau}[v^{(0)}_{\la,N,**} \ j^a_{0,1}]  &=-\frac{1}{\la}\frac{\tau_1^2+\tau_2^2}{2\beta}, \\
\lim_{N \to \infty} E_{N,\b,\tau}[v^{(i)}_{\la,N,**} \ j^a_{0,1}] &=-\frac{\la}{\la^2+B^2}\frac{\tau_1^2+\tau_2^2}{2\beta}, \\ 
\lim_{N \to \infty} E_{N,\b,\tau}[v^{(0)}_{\la,N,**} \ j^a_{0,1}] &=-\frac{\la^2+4\ga \la +4}{\la(\la^2+4\ga \la +4+B^2)}\frac{\tau_1^2+\tau_2^2}{2\beta}
\end{align*}
and so the term $E_{N,\b,\tau}[v^{(\#)}_{\la,N,**}j^a_{0,1}]$ does not vanish in the limit. On the other hand, the behavior of the term $E_{N,\b,\tau}[v^{(\#)}_{\la,N,*}j^a_{0,1}]$ does not depend on $\tau$, which we can see from the argument below. Therefore, by the inverse Laplace transform, we can see that the current correlation $D_{\infty}^{(\#)}(t)$ does not decay under the dynamics (0) and (ii). On the other hand, under the dynamics (i) the term $E_{N,\b,\tau}[v^{(i)}_{\la,N,**} \ j^a_{0,1}]$ is an oscillation term and so the asymptotic behavior of the Green-Kubo integral does not depend on $\tau$.
\end{rem}

We study the term $E_{N,\b}[ v^{(\#)}_{\la,N,*} \ j^a_{0,1}]$ by the same strategy in the last section. 
First, let us consider the uniform case. From Lemma \ref{lem:resolvent2} in Appendix \ref{app:resolvent}, we have 
\begin{align*}
u^{(i)}_{\la,N} & =\sum_{x,y \in \Z_N} (g^1_{\la,N}(x-y)q_{x}^1 q_{y}^2+g^2_{\la,N}(x-y)(q_{x}^1 v_{y}^1+q_{x}^2 v_{y}^2) \\
& +g^3_{\la,N}(x-y)(q_{x}^1 v_{y}^2-q_{x}^2 v_{y}^1) +g^4_{\la,N}(x-y)v_{x}^1 v_{y}^2)  
\end{align*}
where $(g^1_{\la,N},g^2_{\la,N},g^3_{\la,N},g^4_{\la,N})$ is the solution of (\ref{eq:forg1-4}). As already discussed,
$v^{(i)}_{\la,N,*}$ is obtained from $u^{(i)}_{\la,N}$ by replacing $q_x^j$ by $-\sum_{y=x}^{N} (r^j_y-\bar{r}^j)$. 
Since the measure $\mu_{N,\beta}$ is product, it is easy to see that 
\begin{align*}
E_{N,\b}[ v^{(i)}_{\la,N,*} \ j^a_{0,1}] & = \sum_{x,y \in \Z_N} g^2_{\la,N}(x-y)E_{N,\b}[((-\sum_{z=x}^{N} (r^1_z-\bar{r}^1)) v_{y}^1+(-\sum_{z=x}^{N} (r^2_z-\bar{r}^2)) v_{y}^2)  \ j^a_{0,1}] \nonumber \\
& = \sum_{x,y \in \Z_N} g^2_{\la,N}(x-y)E_{N,\b}[\sum_{z=x}^{N} (r^1_z-\bar{r}^1) r_0]E_{N,\b} [ v_{y}^1 (v_{1}^1+v_{0}^1)]  \nonumber  \\
&= \frac{1}{\beta} \sum_{x \in \Z_N} ( g^2_{\la,N}(x) +g^2_{\la,N}(x-1))  E_{N,\b}[\sum_{z=x}^{N} (r^1_z-\bar{r}^1) r_0]  \nonumber  \\
&= \frac{1}{\beta} \sum_{x \in \Z_N}g^2_{\la,N}(x)E_{N,\b}[\big( \sum_{z=x}^{N} (r^1_z-\bar{r}^1) + \sum_{z=x+1}^{N} (r^1_z-\bar{r}^1) \big)r_0] .
\end{align*}

Let $L(x):=E_{N,\b}[\big(\sum_{z=x}^{N} (r^1_z-\bar{r}^1) + \sum_{z=x+1}^{N} (r^1_z-\bar{r}^1) \big)r_0]$ for $x \in \Z_N$. By simple computations,
\begin{align*}
L(x+1)-2L(x)+L(x-1) & =E_{N,\b}[(-r_{x+1}^1+r_{x-1}^1) r_0]= \frac{1}{\beta} (-\delta_{x,-1}+\delta_{x,1}).
\end{align*}
Since $\sum_{x \in \Z_N}g^2_{\la,N}(x)=0$, 
\begin{align*}
E_{N,\b}[ v^{(i)}_{\la,N,*} \ j^a_{0,1}] &= \frac{1}{\beta} \sum_{x \in \Z_N} g^2_{\la,N}(x)L(x) = \frac{1}{\beta} \sum_{x \in \Z_N} g^2_{\la,N}(x)(L(x)-\bar{L})
\end{align*}
where $\bar{L}=\frac{1}{N}\sum_{x}L(x)$ and by Parseval's identity,
\begin{align*}
E_{N,\b}[ v^{(i)}_{\la,N,*} \ j^a_{0,1}] = -\frac{1}{\beta^2}\frac{1}{N} \sum_{\xi \in \Z_N,  \xi \neq \mathbf{0} } \hat{g}^2_{\la,N}(\xi) \frac{i \sin (2 \frac{\pi \xi}{N}) }{2\sin^2(\frac{\pi \xi}{N})}.
\end{align*}
In this way, for $\hat{g}_{\lambda}^2(\frac{\xi}{N})=\hat{g}^2_{\la,N}(\xi)$, we have 
\[
\tilde{D}_N^{(i)}(\lambda) 
 \to  \frac{-i}{2\beta^2}\int_{[0,1]} \hat{g}_{\lambda}^2(\theta)\frac{\sin(2\pi \theta)}{\sin^2(\pi \theta)}d\theta= \frac{-2i}{\beta^2} \int_{[0,1]} \hat{g}_{\lambda}^2(\theta)\frac{\sin(2\pi \theta)}{\omega_{\theta}^2}d\theta.
 \]
The inverse Laplace transform of this limiting function has been already studied in the last section, so we conclude Proposition \ref{prop:resolvent2} for the uniform case, namely the case $\#=i$. The case $\#=0$ can be shown in the same way, so we omit details. 

Finally, we consider the alternate case. From Lemma \ref{lem:resolvent3} in Appendix \ref{app:resolvent}, we have
\begin{align*}
u^{(ii)}_{\la,N} & =\sum_{x \equiv y \mod 2}\big(h^1_{\la,N}(x-y) (-1)^y q_x^1 q_y^2+h^2_{\la,N} (x-y) (-1)^y v_x^1 v_y^2) \big) \\
& + \sum_{x,y} \big( h^3_{\la,N} (x-y)(q_x^1 v_y^1+ q_x^2 v_y^2)+ h^4_{\la,N}(x-y)  (-1)^y  (q_x^1 v_y^2- q_x^2 v_y^1)\big) 
\end{align*}
where $(h^1_{\la,N},h^2_{\la,N},h^3_{\la,N},h^4_{\la,N})$ is the solution of (\ref{eq:forg1-6 alt}). Let $\hat{h}^3_{\la,N}$ be the discrete Fourier transform of $h^3_{\la,N}$. We can apply the same argument as above to show that 
\[
\tilde{D}_N^{(ii)}(\lambda) 
 \to  \frac{-i}{2\b^2} \int_{[0,1]} \hat{h}_{\lambda}^3(\theta)\frac{\sin(2\pi \theta)}{\sin^2(\pi \theta)}d\theta= \frac{-2i}{\b^2} \int_{[0,1]} \hat{h}_{\lambda}^3(\theta)\frac{\sin(2\pi \theta)}{\omega_{\theta}^2}d\theta
 \]
if $\hat{h}^3_{\la,N}(\xi)=\hat{h}^3_{\la}(\frac{\xi}{N})$ for some function $\hat{h}^3_{\la}$ and $\hat{h}^3_{\la}(\theta) \frac{\sin(2\pi \theta)}{\omega_{\theta}^2}$ behaves well at the boundary. We prove this is the case in the rest of this subsection.

We define the discrete Fourier transform $\hat{h}^i_{\la,N,e}$ and $\hat{h}^i_{\la,N,o}$ for $i=1,2,3,4$ by 
\[
\hat{h}^i_{\la,N,e} (\xi)=\sum_{x : \text{even}} h^i_{\la,N}(x) e^{-2\pi i \frac{\xi}{N}x}, \quad \hat{h}^i_{\la,N,o} (\xi)=\sum_{x: \text{odd}} h^i_{\la,N}(x) e^{-2\pi i \frac{\xi}{N}x}
\]
where the first sum is taken over all $x \equiv 0 \mod 2$ and the second one is take over all $x \equiv 1 \mod 2$. Recall that we assume $N$ to be even for the alternate case. From the explicit expression (\ref{eq:forg1-6 alt}), we have
\begin{align*}
\hat{h}^1_{\la,N,e}(\xi)=\frac{4}{\la}(-\hat{h}^4_{\la,N,e}(\xi)-\cos(\frac{2\pi \xi}{N})\hat{h}^4_{\la,N,o}(\xi)),  \quad \hat{h}^2_{\la,N,e}(\xi)=\frac{2}{\la+4\ga}\hat{h}^4_{\la,N,e}(\xi).
\end{align*}
With these expressions, it is easy to see that $\hat{h}^i_{\la,N,e}(\xi)=\hat{h}_{\la,e}(\frac{\xi}{N})$ and $\hat{h}^i_{\la,N,o}(\xi)=\hat{h}_{\la,o}(\frac{\xi}{N})$ for $i=3,4$ where $( \hat{h}^3_{\la,o},\hat{h}^3_{\la,e}, \hat{h}^4_{\la,o},\hat{h}^4_{\la,e})$ is the solution of the following linear equation:
\begin{align*}
& \left(
\begin{array}{cccc}
\bar{\la} & -2\ga\cos(2\pi\theta) & B &  0 \\
-2\ga\cos(2\pi\theta) & \bar{\la}  &  0 &  B  \\
-B & 0 & \bar{\la}    &   2\cos(2\pi\theta)(\ga-\frac{2}{\bar{\la} +2\ga}) \\
0 & -B & 2\cos(2\pi\theta)(\ga+\frac{2}{\bar{\la} -2\ga})  &  \bar{\la} (1+\frac{8}{\bar{\la} ^2-4\ga^2})  
\end{array}
\right)
\left(
\begin{array}{c}
\hat{h}^3_{\la,o} \\
\hat{h}^3_{\la,e}  \\
\hat{h}^4_{\la,o} \\
\hat{h}^4_{\la,e} 
\end{array}
\right) \\
& = 
\left(
\begin{array}{c}
i\sin(2\pi \theta) \\
0  \\
0 \\
0
\end{array}
\right)
\end{align*}
where $\bar{\la} =\la+2\ga$. Then, by a direct computation, we have
\[
\hat{h}_{\lambda}^3(\theta)\frac{-i\sin(2\pi \theta)}{\omega_{\theta}^2} =\cos^2(\pi \theta) \frac{R(\la)}{S(\la)}
\]
where 
\begin{align}
R(\la) & =\bar{\la}((B^2+\bar{\la}^2)(8-4\ga^2+\bar{\la}^2)-8B^2) \nonumber \\
& + 2(B^2(2\bar{\la}+\ga(4-4\ga^2+\bar{\la}^2))+\ga\bar{\la}^2(8-4\ga^2+\bar{\la}^2))\cos(2\pi\theta) \nonumber \\
&+(4+4\ga^4-\ga^2(8+\bar{\la}^2))(4\bar{\la}\cos^2(2\pi\theta)+8\ga \cos^3(2\pi\theta)), \label{eq:FG} \\
S(\la)& =(B^2+\bar{\la}^2)((B^2+\bar{\la}^2)(8-4\ga^2+\bar{\la}^2)-8B^2) \nonumber \\
& +8(-B^2\ga^2(4-4\ga^2+\bar{\la}^2)+\bar{\la}^2(2+4\ga^4-\ga^2(8+\bar{\la}^2)))\cos^2(2\pi\theta) \nonumber \\
& -16\ga^2(4+4\ga^4-\ga^2(8+\bar{\la}^2))\cos^4(2\pi\theta). \nonumber
\end{align}

To simplify the notation, we introduce $Y=X-4\ga^2$. Then, we have
\begin{align*}
%S(\la) & =\bar{\la}S_1(\la)+S_2(\la), \\
S(\la)& =(B^2+Y+4\ga^2)((B^2+Y+4\ga^2)(8+Y)-8B^2) \\
& +8(-B^2\ga^2(4+Y)+(Y+4\ga^2)(2-8\ga^2-\ga^2Y))(1-\sin^2(2\pi\theta)) \\
& -16\ga^2(4-8\ga^2-\ga^2Y)(1-\sin^2(2\pi\theta))^2 \\
&= Y^3+(8+2B^2+8\ga^2\sin^2(2\pi\theta))Y^2 \\
& +(B^4+8(1+\ga^2\sin^2(2\pi\theta))B^2+16\cos^2(2\pi\theta)+64\ga^2\sin^2(2\pi\theta)+16\ga^4\sin^4(2\pi\theta))Y \\
&+ 32B^2\ga^2\sin^2(2\pi\theta)+64\ga^2\sin^2(2\pi\theta) (1-\sin^2(2\pi\theta))+128\ga^4\sin^4(2\pi\theta).
\end{align*}
We call the last expression $\tilde{S}(Y)$. Since $Y = \la^2 +4\la \ga >0$, 
\[\inf_{\theta \in [0,1)}S(\la)=\inf_{\theta \in [0,1)}\tilde{S}(Y) \ge Y^3  >0.
\]
Therefore, $\cos^2(\pi \theta) \frac{R(\la)}{S(\la)}$ is continuous and bounded on $[0,1]$ and so we complete the proof of Proposition \ref{prop:resolvent2}.

\subsection{Analysis of the inverse Laplace transform}\label{inverselaplace2} 

In this subsection, we study the inverse Laplace transform of $\int_0^1 \cos^2(\pi \theta) \frac{S(\la)}{R(\la)}d\theta$ and its asymptotic behavior. First note that, 
%\[
%\int_{[0,1)} \cos^2(\pi \theta)  \frac{\cos(2\pi\theta)}{S(\la)}d\theta= \int_{[0,1)}\cos^2(\pi \theta)   \frac{\cos^3(2\pi\theta)}{S(\la)}d\theta=0
%\]
%by the symmetry. Then, we have
\[
\int_0^1 \cos^2(\pi \theta)  \frac{R(\la)}{S(\la)}d\theta =2\int_{0}^{\frac{1}{2}}\cos^2(\pi \theta)  \frac{R(\la)}{S(\la)}d\theta
\]
since $R(\la)$ and $S(\la)$ are symmetric in $\theta$. Also, since $\sin^2(2\pi\theta)=\sin^2(2\pi(\frac{1}{2}-\theta))$, $\cos(2\pi\theta)=-\cos(2\pi(\frac{1}{2}-\theta))$ and $\cos(\pi\theta)=\sin(\pi(\frac{1}{2}-\theta))$,
\[
\int_{0}^{\frac{1}{2}}\cos^2(\pi \theta)  \frac{R(\la)}{S(\la)}d\theta=\int_{0}^{\frac{1}{4}}\frac{\bar{\la}R_1(\la)}{S(\la)}d\theta+\int_{0}^{\frac{1}{4}}\cos(2\pi\theta)\frac{R_2(\la)}{S(\la)}d\theta
\]
where 
\begin{align*}
R_1(\la) &=(B^2+\bar{\la}^2)(8-4\ga^2+\bar{\la}^2)-8B^2+4 (4+4\ga^4-\ga^2(8+\bar{\la}^2))\cos^2(2\pi\theta), \\
R_2(\la) & =2(B^2(2\bar{\la}+\ga(4-4\ga^2+\bar{\la}^2))+\ga\bar{\la}^2(8-4\ga^2+\bar{\la}^2))\cos(2\pi\theta) \\
 & +8\ga(4+4\ga^4-\ga^2(8+\bar{\la}^2))\cos^3(2\pi\theta).
\end{align*}
Let $\bar{R}(\la)=\bar{\la}R_1(\la)+R_2(\la)\cos(2\pi\theta)$. Then, from the above observations, 
\[
\int_0^1 \cos^2(\pi \theta)  \frac{R(\la)}{S(\la)}d\theta =2\int_0^{\frac{1}{4}}\frac{\bar{R}(\la)}{S(\la)}d\theta.
\]
%We define $\tilde{R}(Y)=(B^2+Y+4\ga^2)(8+Y)-8B^2+ 4( 4-8\ga^2-\ga^2Y)\cos^2(2\pi\theta)$.

%Let us start to simplify notations. $S(\la)=\bar{\la}S_1(\la)+S_2(\la)$, where
%\begin{align*}
%S_1(\la)& =(B^2+\bar{\la}^2)(8-4\ga^2+\bar{\la}^2)-8B^2 + 4B^2\cos(2\pi\theta) \\
%&+4 (4+4\ga^4-\ga^2(8+\bar{\la}^2))\cos^2(2\pi\theta)
%\end{align*}
%and 
%\begin{align*}
%S_2(\la)& =2(B^2\ga(4-4\ga^2+\bar{\la}^2)-\ga\bar{\la}^2(8-4\ga^2+\bar{\la}^2))\cos(2\pi\theta) \\
%&+8\ga  (4+4\ga^4-\ga^2(8+\bar{\la}^2))\cos^3(2\pi\theta), \\
%\end{align*}

To obtain the inverse Laplace transform of $\frac{\bar{R}(\la)}{S(\la)}$, we compute the partial fraction decomposition of $\frac{\bar{R}(\la)}{S(\la)}$. For this, we first study the zero points of $S(\la)$, namely the solution of $S(\la)=0$. Let $B=2\tilde{B}$ and $T(Y)=\frac{\tilde{S}(4Y)}{64}$. Then, 
by a simple computation, we have 
\begin{align*}
T(Y) & =Y^3+2(1+\tilde{B}^2+\ga_{\theta}^2)Y^2+(\tilde{B}^4+2(1+\ga_{\theta}^2)\tilde{B}^2+\cos^2(2\pi\theta)+4\ga_{\theta}^2+\ga_{\theta}^4)Y \\
& +2\tilde{B}^2\ga_{\theta}^2+\ga_{\theta}^2\cos^2(2\pi\theta)+2\ga_{\theta}^4
\end{align*}
where $\ga_{\theta}=\ga\sin(2\pi\theta)$. In particular, for $\theta=0$, $T(Y)=Y(Y+\tilde{B}^2+1)^2$. So, $T(Y)=0$ has two solutions $0, -\tilde{B}^2-1$.

\begin{lem}
Assume $\ga  \le 1$. Then, for each fixed $\theta \in (0,\frac{1}{4}]$, the third order polynomial $T(Y)=0$ has three distinct solutions
$\tilde{\a}_1(\theta),\tilde{\a}_2(\theta),\tilde{\a}_3(\theta) \in \R$ such that $0>\tilde{\a}_1(\theta)>-\ga^2> \tilde{\a}_2(\theta) > -\tilde{B}^2-1 > \tilde{\a}_3(\theta)$.
\end{lem}
\begin{proof}
For any $\theta \in (0,\frac{1}{4}]$ and $Y \ge 0 $, obviously $T(Y) >0$. Also,
\begin{align*}
T(-\ga^2) & =-\ga^2\left( \tilde{B}^2+1-\ga^2+\sin^2(2\pi \theta)(\ga^2-1)\right)^2 <0, \\
T(-\tilde{B}^2-1) & =(\tilde{B}^2+1-\ga^2)\sin^2(2\pi\theta)-\ga^2(\tilde{B}^2\ga^2+1-\ga^2)\sin^4(2\pi\theta)  \\
 & \ge (\tilde{B}^2+1-\ga^2)\sin^2(2\pi\theta)-(\tilde{B}^2+1-\ga^2)\sin^4(2\pi\theta) >0.
\end{align*}
\end{proof}
We define $\tilde{\a}_1(0)=0,\tilde{\a}_2(0)=\tilde{\a}_3(0)=-\tilde{B}^2-1$. Then, by the continuity of the coefficients of polynomial $T(Y)$, $\tilde{\a}_i(\theta)$ are continuous on $[0,\frac{1}{4}]$. From this, we have 
\[
S(\la)=64\prod_{i=1}^3(\frac{Y}{4}-\tilde{\a}_i(\theta))=\prod_{i=1}^3(Y-4\tilde{\a}_i(\theta))=\prod_{i=1}^3(\bar{\la}^2-4\ga^2-4\tilde{\a}_i(\theta)).
\]

Next, we work on the numerator $\bar{R}(\la)$. By a direct computation, we have
\begin{align*}
\bar{R}(\la)& =\bar{\la}\{(B^2+Y+4\ga^2)(8+Y)-8B^2+ 4( 4+B^2-8\ga^2-\ga^2Y)\cos^2(2\pi\theta)\} \\
 &+(2B^2\ga(4+Y)+2\ga(Y+4\ga^2)(Y+8))\cos^2(2\pi\theta) +8\ga(4-8\ga^2-\ga^2Y)\cos^4(2\pi\theta)
\end{align*}
where we use the notation $Y=\bar{\la}^2-4\ga^2$. Let 
\[
U_1(Y)=(B^2+Y+4\ga^2)(8+Y)-8B^2+ 4( 4+B^2-8\ga^2-\ga^2Y)\cos^2(2\pi\theta),
\] 
\[
U_2(Y)=(2B^2\ga(4+Y)+2\ga(Y+4\ga^2)(Y+8))\cos^2(2\pi\theta) +8\ga(4-8\ga^2-\ga^2Y)\cos^4(2\pi\theta)
\]
so as $\bar{R}(\la)=\bar{\la}U_1(Y)+U_2(Y)$.

Now, we give the partial fraction decomposition of $\frac{\bar{R}(\la)}{S(\la)}$. Let $\a_1(\theta)^2=4\ga^2+4\tilde{\a}_1(\theta) >0$ and $\a_i(\theta)^2=-4\ga^2-4\tilde{\a}_i(\theta) >0$ for $i=2,3$ where $\a_i(\theta) > 0$ for $i=1,2,3$ and $\theta \in [0,\frac{1}{4}]$. Though we also use the notation $\a_1(\theta),\a_2(\theta)$ in Subsection \ref{propsec12} which are clearly different from the ones we have just defined, since there is no room for misunderstanding, we use the same notation here.
Also, define
\[
\beta_i(\theta)=U_1(4\tilde{\a}_i(\theta))\prod_{j \neq i} \frac{1}{4(\tilde{\a}_i(\theta)-\tilde{\a}_j(\theta))}, \quad \beta_{i+3}(\theta)=U_2(4\tilde{\a}_i(\theta))\prod_{j \neq i} \frac{1}{4(\tilde{\a}_i(\theta)-\tilde{\a}_j(\theta))}
\]
 for $\theta \in (0,\frac{1}{4}]$ and $i=1,2,3$. 
%\begin{lem}
%$\beta_i(\theta)$ is continuous on $[0,\frac{1}{4}]$ and $\beta_1(0)=\lim_{\theta \to 0} \beta_1(\theta)=\frac{16}{(B^2+4)^2}$.
%\end{lem}
%\begin{proof}
%We only need to prove that $\lim_{\theta \to 0}\beta_2(\theta)$
%\end{proof}

Then, for each $\theta \in (0,\frac{1}{4}]$,
\begin{align*}
\frac{\bar{\la}U_1(Y)}{S(\la)} & =\frac{\bar{\la}U_1(Y)}{\prod_{i=1}^3(Y-4\tilde{\a}_i(\theta))} =\bar{\la}\sum_{i=1}^3\frac{\beta_i(\theta)}{Y-4\tilde{\a}_i(\theta)} \\
&= \frac{\beta_1(\theta)}{2}\frac{1}{\bar{\la}-\a_1(\theta)}+\frac{\beta_1(\theta)}{2} \frac{1}{\bar{\la}+ \a_1(\theta)} +\beta_2(\theta) \frac{\bar{\la}}{\bar{\la}^2+\a_2(\theta)^2}+\beta_3(\theta) \frac{\bar{\la}}{\bar{\la}^2+\a_3(\theta)^2}
\end{align*}
and
\begin{align*}
\frac{U_2(Y)}{S(\la)} & =\frac{U_2(Y)}{\prod_{i=1}^3(Y-4\tilde{\a}_i(\theta))} =\sum_{i=1}^3\frac{\beta_{i+3}(\theta)}{Y-4\tilde{\a}_i(\theta)} \\
&= \frac{\beta_4(\theta)}{2\a_1(\theta)}\frac{1}{\bar{\la}-\a_1(\theta)}-\frac{\beta_4(\theta)}{2\a_1(\theta)} \frac{1}{\bar{\la}+ \a_1(\theta)} +\beta_5(\theta) \frac{1}{\bar{\la}^2+\a_2(\theta)^2}+\beta_6(\theta) \frac{1}{\bar{\la}^2+\a_3(\theta)^2}.
\end{align*}
Therefore, we have 
\begin{align*}
& \mathcal{L}^{-1}\left[ \int_0^{\frac{1}{4}} \frac{\bar{\la}U_1(Y)}{S(\la)}d\theta \right] (t) \\
& = \int_0^{\frac{1}{4}}  \mathcal{L}^{-1}\left[ \frac{\beta_1(\theta)}{2} \frac{1}{\bar{\la}-\a_1(\theta)}+\frac{\beta_1(\theta)}{2} \frac{1}{\bar{\la}+ \a_1(\theta)} +\beta_2(\theta) \frac{\bar{\la}}{\bar{\la}^2+\a_2(\theta)^2}+\beta_3(\theta) \frac{\bar{\la}}{\bar{\la}^2+\a_3(\theta)^2} \right] (t) d\theta \\
& = \exp(-2\ga t) \int_0^{\frac{1}{4}} \left( \frac{\b_1(\theta) }{2} \{ \exp(\a_1(\theta)t) + \exp(-\a_1(\theta)t) \}  + \b_2(\theta) \cos ( \a_2(\theta)t) +  \b_3(\theta) \cos (\a_3(\theta)t) \right) d\theta \\
\end{align*}
and
\begin{align*}
& \mathcal{L}^{-1}\left[ \int_0^{\frac{1}{4}} \frac{U_2(Y)}{S(\la)}d\theta \right] (t) \\
& = \int_0^{\frac{1}{4}}  \mathcal{L}^{-1} \left[ \frac{\beta_4(\theta)}{2\a_1(\theta)} \frac{1}{\bar{\la}-\a_1(\theta)}-\frac{\beta_4(\theta)}{2\a_1(\theta)} \frac{1}{\bar{\la}+ \a_1(\theta)} +\frac{\beta_5(\theta)}{\a_2(\theta)} \frac{\a_2(\theta)}{\bar{\la}^2+\a_2(\theta)^2}+\frac{\b_6(\theta)}{\a_3(\theta)} \frac{\a_3(\theta)}{\bar{\la}^2+\a_3(\theta)^2} \right] (t) d\theta \\
& = \exp(-2\ga t) \int_0^{\frac{1}{4}}\left( \frac{\b_4(\theta) }{2\a_1(\theta)} \{ \exp(\a_1(\theta)t) - \exp(-\a_1(\theta)t) \}  + \frac{\b_5(\theta)}{\a_2(\theta)} \sin ( \a_2(\theta)t) +  \frac{\b_6(\theta)}{\a_3(\theta)} \sin (\a_3(\theta)t) \right) d\theta. \\
\end{align*}

Since $\beta_1(0):=\lim_{\theta \to 0} \beta_1(\theta)=\frac{4(B^2+4)}{(B^2+4)^2} < \infty$, $\beta_1(\theta)$ is continous on $[0,\frac{1}{4}]$ and 
\[
|\exp(-2\ga t) \int_0^{\frac{1}{4}} \frac{\b_1(\theta) }{2} \exp(-\a_1(\theta)t) d\theta | \le C\exp(-2\ga t)
\]
for some positive constant $C$.

In the same way, 
\[
|\exp(-2\ga t) \int_0^{\frac{1}{4}} \frac{\b_4(\theta) }{2\a_1(\theta)} \exp(-\a_1(\theta)t) d\theta | \le C\exp(-2\ga t)
\]
for some positive constant $C$.

Next, we consider the term
\[
\int_0^{\frac{1}{4}} \big( \b_2(\theta) \cos ( \a_2(\theta)t) +  \b_3(\theta) \cos (\a_3(\theta)t) \big) d\theta.
\]

\begin{lem}
There exists a constant $C>0$ such that
\[\sup_{\theta \in [0,\frac{1}{4}]}|\b_2(\theta) \cos ( \a_2(\theta)t) +  \b_3(\theta) \cos (\a_3(\theta)t) | \le C t.
\]
\end{lem}
\begin{proof}
We only need to show that for some $C>0$,
\[
\limsup_{ \theta \to 0} |\b_2(\theta) \cos ( \a_2(\theta)t) +  \b_3(\theta) \cos (\a_3(\theta)t) | \le C t.
\]
By definition,
\begin{align*}
& \b_2(\theta) \cos ( \a_2(\theta)t) +  \b_3(\theta) \cos (\a_3(\theta)t)  \\
& =\frac{1}{\a_3(\theta)^2-\a_2(\theta)^2} \left(\frac{U_1(-\a_2(\theta)^2-4\ga^2)\cos ( \a_2(\theta)t)}{\a_1(\theta)^2+\a_2(\theta)^2} -\frac{U_1(-\a_3(\theta)^2-4\ga^2)\cos ( \a_3(\theta)t)}{\a_1(\theta)^2+\a_3(\theta)^2} \right) \\
&= \frac{1}{\a_3(\theta)^2-\a_2(\theta)^2} \left(\frac{V_1(-\a_2(\theta)^2-4\ga^2)\cos ( \a_2(\theta)t)}{\a_1(\theta)^2+\a_2(\theta)^2} -\frac{V_1(-\a_3(\theta)^2-4\ga^2)\cos ( \a_3(\theta)t)}{\a_1(\theta)^2+\a_3(\theta)^2} \right)  \\
& + \frac{\cos^2(2\pi\theta)}{\a_3(\theta)^2-\a_2(\theta)^2} \left(\frac{V_2(-\a_2(\theta)^2-4\ga^2)\cos ( \a_2(\theta)t)}{\a_1(\theta)^2+\a_2(\theta)^2} -\frac{V_2(-\a_3(\theta)^2-4\ga^2)\cos ( \a_3(\theta)t)}{\a_1(\theta)^2+\a_3(\theta)^2} \right)
\end{align*}
where $V_1(Y)=(B^2+Y+4\ga^2)(8+Y)-8B^2$ and $V_2(Y)=4(4+B^2-8\ga^2-\ga^2Y)$. Let 
\[
f_{i,t}(x)=\frac{V_i(-x^2-4\ga^2)\cos(xt)}{4\ga^2+x^2}
\] for $i=1,2$. Then, since $\lim_{\theta \to 0}\a_1(\theta)=2\ga$
\begin{align*}
& \lim_{ \theta \to 0} \big(\b_2(\theta) \cos ( \a_2(\theta)t) +  \b_3(\theta) \cos (\a_3(\theta)t) \big) \\
& =\lim_{ \theta \to 0} \big( \frac{-1}{\a_2(\theta)+\a_3(\theta)} \frac{f_{1,t}(\a_2(\theta))-f_{1,t}(\a_3(\theta))}{\a_2(\theta)-\a_3(\theta)}  + \frac{-\cos^2(2\pi\theta)}{\a_2(\theta)+\a_3(\theta)} \frac{f_{2,t}(\a_2(\theta))-f_{2,t}(\a_3(\theta))}{\a_2(\theta)-\a_3(\theta)}\big).
\end{align*}
Applying $\lim_{\theta \to 0}\a_2(\theta) =\lim_{\theta \to 0}\a_3(\theta)=2\sqrt{B^2+4-4\ga^2} >0$, we have
\[
\lim_{ \theta \to 0} \big(\b_2(\theta) \cos ( \a_2(\theta)t) +  \b_3(\theta) \cos (\a_3(\theta)t) \big)=\frac{-1}{4\sqrt{B^2+4-4\ga^2}}\sum_{i=1}^2f'_{i,t}(2\sqrt{B^2+4-4\ga^2}).
\]
From the explicit expression of $f_{i,t}$ for $i=1,2$, $|f'_{i,t}(2\sqrt{B^2+4-4\ga^2})| \le Ct$ for some $C>0$, and hence we complete the proof.
\end{proof}

From the above lemma, we have
\[
 \left| \exp(-2\ga t) \int_0^{\frac{1}{4}} \left(\b_2(\theta) \cos ( \a_2(\theta)t) +  \b_3(\theta) \cos (\a_3(\theta)t) \right) d\theta \right| \le C t \exp(-2\ga t)
\]
for some positive constant $C$. 

In the same way, we have
\[
 \left| \exp(-2\ga t) \int_0^{\frac{1}{4}} \left( \frac{\b_5(\theta)}{\a_2(\theta)} \sin ( \a_2(\theta)t) +  \frac{\b_6(\theta)}{\a_3(\theta)} \sin (\a_3(\theta)t) \right) d\theta \right| \le C t \exp(-2\ga t)
\]
for some positive constant $C$. 

Finally, we study the term 
\[
 \int_0^{\frac{1}{4}}\left( \b_1(\theta) +\frac{\b_4(\theta)}{\a_1(\theta)}\right) \exp((-2\ga+\a_1(\theta))t) d\theta.
\]

Since $\a_1(\theta) < 2\ga$ for $\theta \neq 0$, we only need to consider its behavior around $\theta=0$. For this, we study the behavior of $\tilde{\a}_1(\theta) \le0$ which satisfies $\tilde{\a}_1(0)=0$ and $T(\tilde{\a}_1(\theta))=0$. By expanding $T(c\sin^2(2\pi\theta))$ around $\theta=0$, we have
\begin{equation}\label{kappaasym}
\lim_{\theta \to 0}\frac{\tilde{\a}_1(\theta)}{\sin^2(2\pi\theta)}=-\frac{\ga^2(2\tilde{B}^2+1)}{(\tilde{B}^2+1)^2}=-\frac{8\ga^2(B^2+2)}{(B^2+4)^2}.
\end{equation}
Namely, $-4\tilde{\a}_1(\theta)=4\ga^2-\a_1(\theta)^2 = O(\theta^2)$ as $\theta \to 0$. Hence $2\ga-\a_1(\theta) \sim O(\theta^2)$ as $\theta \to 0$.
Also, by a direct computation,
\begin{equation}\label{kappaasym2}
\lim_{\theta \to 0}\left( \b_1(\theta) +\frac{\b_4(\theta)}{\a_1(\theta)}\right) =\frac{8}{B^2+4} \neq 0.
\end{equation}
Therefore,
\[
 \int_0^{\frac{1}{4}} \left( \b_1(\theta) +\frac{\b_4(\theta)}{\a_1(\theta)}\right) \exp \big((-2\ga+\a_1(\theta))t\big) d\theta \sim t^{-1/2}.
\]

\begin{rem}\label{ga>1case}
If $\ga >1$, $T(Y)=0$ may have a multiple root for some $\theta \neq 0$. Even for such a case, we expect the asymptotic behavior of the thermal conductivity is same as the case for $\ga \le 1$ by the following reason. First, for general $\ga>0$, $\tilde{\a}_i(\theta), \ i=1,2,3$, the solutions of $T(Y)=0$, are all continuous and if $\tilde{\a}_i(\theta) \in \R$, then $\tilde{\a}_i(\theta) <0$ for $\theta \in (0,\frac{1}{4}]$. Moreover, $\tilde{\a}_1(\theta) \in \R$ for small enough $\theta$ and the asymptotic behavior of $\tilde{\a}_1(\theta)$ as $\theta \to 0$, which is given at (\ref{kappaasym}), also holds. Namely, among $6$ poles of $\frac{\bar{R}(\la)}{S(\la)}$, there is just one pole which can contribute to the divergence of the thermal conductivity. Since the exponent is determined by the asymptotic behavior (\ref{kappaasym}) and (\ref{kappaasym2}) which both hold for general $\gamma >0$, we should have $D^{(ii)}(t) \sim t^{-\frac{1}{2}}$ in general. However, a rigorous argument will be too complicated, so we do not pursue it here. 
\end{rem}

\section*{Acknowledgement}

The authors express their sincere thanks to Herbert Spohn and Stefano Olla for insightful discussions, and to Shuji Tamaki for careful reading of the manuscript and helpful comments. The authors also thank anonymous referees for their constructive comments that significantly improve the manuscript. KS was supported by JSPS Grants-in-Aid for Scientific Research No. JP26400404 and No. JP16H02211. MS was supported by JSPS Grant-in-Aid for Young Scientists (B) No. JP25800068.

\appendix

\section{Convergence of functions and Laplace transform}\label{app:convergence}

\begin{prop}\label{prop:ascoli}
Suppose a set of functions $f_N:[0,\infty) \to \R$ indexed by $\N$ is uniformly bounded and equicontinuous. Moreover, assume that for any $\lambda>0$, 
\begin{align*}
\lim_{N \to \infty}\int_0^{\infty} e^{-\lambda t} f_N(t)dt 
\end{align*}
exists. Then, there exists a function $f_{\infty}:[0,\infty) \to \R$ such that $f_N$ converges compactly to $f_{\infty}$, namely the convergence is uniform on each compact subset of $[0,\infty)$, and also
\begin{align*}
\lim_{N \to \infty}\int_0^{\infty} e^{-\lambda t} f_N(t)dt =\int_0^{\infty} e^{-\lambda t} f_{\infty}(t)dt.
\end{align*}
\end{prop}

\begin{proof}
By Arzel\'a-Ascoli theorem, if $f_N$ is restricted to each compact interval in $[0,\infty)$, there exists a subsequence ${N_k}$ such that the function $f_{N_k}$ converges in the uniform topology on this interval. Therefore, by the diagonal argument, we can construct a subsequence $\tilde{N_k}$ such that $f_{\tilde{N_k}}$ converges compactly on $[0,\infty)$. Denote its limit by $f_{\infty}$. Then, by the Lebesgue's convergence theorem,
\begin{align*}
\lim_{k \to \infty}\int_0^{\infty} e^{-\lambda t} f_{\tilde{N_k}}(t)dt =\int_0^{\infty} e^{-\lambda t} f_{\infty}(t)dt
\end{align*}
since $\{f_N\}$ is uniformly bounded. Then, by the injectivity of the Laplace transform, the original sequence $f_N$ also should converge to $f_{\infty}(t)$.
\end{proof}

\section{Resolvent equation}\label{app:resolvent}

In this section, we give an explicit solution of the resolvent equation 
\[(\la - L^{(\#)}) u_{\la,N}= \sum_{\mathbf{x}} j^a_{\mathbf{x},\mathbf{x}+\mathbf{e}_1}=\sum_{j=1}^{d^*}\sum_{\mathbf{x},\mathbf{y} \in \Z_N^d}H(\mathbf{x}-\mathbf{y})\mathbf{q}_{\mathbf{x}}^j\mathbf{v}_{\mathbf{y}}^j
\]
for each $\la >0$ where $H(z)=\frac{1}{2}(\delta_{\mathbf{z}+\mathbf{e}_1}-\delta_{\mathbf{z}-\mathbf{e}_1})$ and $L^{(\#)}=A+BG^{(\#)}+\ga S$ for $\#=0,i,ii$. Note that $L^{(i)}$ and $G^{(i)}$ are denoted by $L$ and $G$ in Subsection \ref{microsec} respectively.

First, we concern the components for which the magnetic field does not effect. Let $g_{\la,N}:\Z_N^d \to \R$ be the solution of the equation
\begin{equation}\label{eq:forg}
\lambda g_{\la,N}(z) - \ga \Delta g_{\la,N}(z)=H(z)
\end{equation}
which satisfies $g_{\la,N}(z)=-g_{\la,N}(-z)$. The existence and the uniqueness of such a solution follows straightforwardly by the Fourier transform.

\begin{lem}\label{lem:resolvent1}
For any $j=1,2,\dots,d^*$, let $u_{\la,N,j}=\sum_{\mathbf{x},\mathbf{y} \in \Z_N^d}g_{\la,N}(\mathbf{x}-\mathbf{y})q_{\mathbf{x}}^j v_{\mathbf{y}}^j$. Then,
\[
(\la - (A+\ga S))u_{\la,N,j}=\sum_{\mathbf{x},\mathbf{y} \in \Z_N^d}H(\mathbf{x}-\mathbf{y})q_{\mathbf{x}}^j v_{\mathbf{y}}^j.
\] 
\end{lem}
\begin{proof}
By direct computations,
\begin{align*}
A u_{\la,N,j} = \sum_{\mathbf{x},\mathbf{y} \in \Z_N^d}g_{\la,N}(\mathbf{x}-\mathbf{y})v_{\mathbf{x}}^j v_{\mathbf{y}}^j  +  \sum_{\mathbf{x},\mathbf{y} \in \Z_N^d}g_{\la,N}(\mathbf{x}-\mathbf{y})q_{\mathbf{x}}^j [\Delta q^j]_{\mathbf{y}} \\
= \sum_{\mathbf{x},\mathbf{y} \in \Z_N^d}g_{\la,N}(\mathbf{x}-\mathbf{y})v_{\mathbf{x}}^j v_{\mathbf{y}}^j + \sum_{\mathbf{x},\mathbf{y} \in \Z_N^d} [\Delta g_{\la,N}](\mathbf{x}-\mathbf{y})q_{\mathbf{x}}^j q^j_{\mathbf{y}} =0
\end{align*}
since $g_{\la,N}(z)=-g_{\la,N}(-z)$. Also, a simple computation shows that
\begin{align*}
\ga S u_{\la,N,j} = \ga \sum_{\mathbf{x},\mathbf{y} \in \Z_N^d}g_{\la,N}(\mathbf{x}-\mathbf{y})q_{\mathbf{x}}^j[\Delta v^j]_{\mathbf{y}} 
= \ga \sum_{\mathbf{x},\mathbf{y} \in \Z_N^d} [\Delta g_{\la,N}](\mathbf{x}-\mathbf{y})q_{\mathbf{x}}^j v_{\mathbf{y}}^j
\end{align*}
which concludes the proof. 
\end{proof}

From this lemma, we have
\[(\la - L^{(0)}) \sum_{j=1}^{d^*}u_{\la,N,j}=\sum_{j=1}^{d^*}\sum_{\mathbf{x},\mathbf{y} \in \Z_N^d}H(\mathbf{x}-\mathbf{y})\mathbf{q}_{\mathbf{x}}^j\mathbf{v}_{\mathbf{y}}^j
\]
and
\[(\la - L^{(\#)}) \sum_{j=3}^{d^*}u_{\la,N,j}=\sum_{j=3}^{d^*}\sum_{\mathbf{x},\mathbf{y} \in \Z_N^d}H(\mathbf{x}-\mathbf{y})\mathbf{q}_{\mathbf{x}}^j\mathbf{v}_{\mathbf{y}}^j
\]
for $\#=i,ii$ with $u_{\la,N,j}$ given in Lemma \ref{lem:resolvent1}. 

Next, we work on the components $j=1,2$ for the case with uniform charges.

Let $g^i_{\la,N}:\Z_N^d \to \R \  (i=1,2,3,4)$ be the solution of the simultaneous equations
\begin{align}\label{eq:forg1-4}
\begin{cases}
\lambda g^1_{\la,N}(z) - 2\Delta g^3_{\la,N}(z)=0,  \\
(\lambda -\ga \Delta) g^2_{\la,N}(z) + B g^3_{\la,N}(z)=H(z),  \\
- g^1_{\la,N}(z) - Bg^2_{\la,N}(z)+(\la-  \ga\Delta) g^3_{\la,N}(z) - \Delta g^4_{\la,N}(z)=0, \\
-2 g^3_{\la,N}(z) + (\lambda- 2\gamma\Delta) g^4_{\la,N}(z)=0 \\
\end{cases}
\end{align}
which satisfies $g^i_{\la,N}(z)=-g^i_{\la,N}(-z)$ for $i=1,2,3,4$. The existence and the uniqueness follows by the Fourier transform again.

%In Section \ref{microsec}, we denoted $L^{(i)}$ by $L$, so we use the same notation in the next lemma.
%The next proposition is the key to compute the current-current correlation for the dynamics with the magnetic field.
\begin{lem}\label{lem:resolvent2}
Define $u_{\la,N}$ by
\begin{align*}
u_{\la,N} & =\sum_{\mathbf{x},\mathbf{y} \in \Z_N^d} \big ( g^1_{\la,N} (\mathbf{x}-\mathbf{y})q_{\mathbf{x}}^1 q_{\mathbf{y}}^2+g^2_{\la,N} (\mathbf{x}-\mathbf{y})(q_{\mathbf{x}}^1 v_{\mathbf{y}}^1+q_{\mathbf{x}}^2 v_{\mathbf{y}}^2) \\
& +g^3_{\la,N}(\mathbf{x}-\mathbf{y})(q_{\mathbf{x}}^1 v_{\mathbf{y}}^2-q_{\mathbf{x}}^2 v_{\mathbf{y}}^1)+g^4_{\la,N}(\mathbf{x}-\mathbf{y})v_{\mathbf{x}}^1 v_{\mathbf{y}}^2 \big).
\end{align*}
Then,
\[
(\la-L)u_{\la,N}=(\la - L^{(i)})u_{\la,N}=\sum_{\mathbf{x},\mathbf{y} \in \Z_N^d} \sum_{j=1,2}H(\mathbf{x}-\mathbf{y})q_{\mathbf{x}}^j v_{\mathbf{y}}^j.
\]
\end{lem}
\begin{proof}
It follows by the direct computations:
\begin{align*}
 L^{(i)} & \big(\sum_{\mathbf{x},\mathbf{y} \in \Z_N^d}g^1_{\la,N}(\mathbf{x}-\mathbf{y})q_{\mathbf{x}}^1 q_{\mathbf{y}}^2 \big)= \sum_{\mathbf{x},\mathbf{y} \in \Z_N^d}g^1_{\la,N}(\mathbf{x}-\mathbf{y})(v_{\mathbf{x}}^1 q_{\mathbf{y}}^2 +q_{\mathbf{x}}^1 v_{\mathbf{y}}^2) \\ 
& =\sum_{\mathbf{x},\mathbf{y} \in \Z_N^d}g^1_{\la,N}(\mathbf{x}-\mathbf{y})(q_{\mathbf{x}}^1 v_{\mathbf{y}}^2-q_{\mathbf{x}}^2 v_{\mathbf{y}}^1), \\
 L^{(i)} & \big( \sum_{\mathbf{x},\mathbf{y} \in \Z_N^d}g^2_{\la,N}(\mathbf{x}-\mathbf{y})(q_{\mathbf{x}}^1 v_{\mathbf{y}}^1+q_{\mathbf{x}}^2 v_{\mathbf{y}}^2) \big)  \\
&=  \sum_{\mathbf{x},\mathbf{y} \in \Z_N^d}g^2_{\la,N}(\mathbf{x}-\mathbf{y})\{ (v_{\mathbf{x}}^1 v_{\mathbf{y}}^1+v_{\mathbf{x}}^2 v_{\mathbf{y}}^2) + (q_{\mathbf{x}}^1 [\Delta q^1]_{\mathbf{y}}+q_{\mathbf{x}}^2 [\Delta q^2]_{\mathbf{y}}) \\
& + B(q_{\mathbf{x}}^1 v_{\mathbf{y}}^2-q_{\mathbf{x}}^2 v_{\mathbf{y}}^1) + \gamma (q_{\mathbf{x}}^1 [\Delta v^1]_{\mathbf{y}}+q_{\mathbf{x}}^2 [\Delta v^2]_{\mathbf{y}}) \\
& = \sum_{\mathbf{x},\mathbf{y} \in \Z_N^d} \{ g^2_{\la,N}(\mathbf{x}-\mathbf{y})  B(q_{\mathbf{x}}^1 v_{\mathbf{y}}^2-q_{\mathbf{x}}^2 v_{\mathbf{y}}^1) +  \gamma  [\Delta g^2_{\la,N}](\mathbf{x}-\mathbf{y})(q_{\mathbf{x}}^1 v^1_{\mathbf{y}}+q_{\mathbf{x}}^2 v^2_{\mathbf{y}}) \}, \\
 L^{(i)} & \big(\sum_{\mathbf{x},\mathbf{y} \in \Z_N^d}g^3_{\la,N}(\mathbf{x}-\mathbf{y})(q_{\mathbf{x}}^1 v_{\mathbf{y}}^2-q_{\mathbf{x}}^2 v_{\mathbf{y}}^1) \big)= \sum_{\mathbf{x},\mathbf{y} \in \Z_N^d}g^3_{\la,N}(\mathbf{x}-\mathbf{y})\{ (v_{\mathbf{x}}^1 v_{\mathbf{y}}^2-v_{\mathbf{x}}^2 v_{\mathbf{y}}^1) \\
& +  (q_{\mathbf{x}}^1 [\Delta q^2]_{\mathbf{y}}-q_{\mathbf{x}}^2 [\Delta q^1]_{\mathbf{y}}) + B(-q_{\mathbf{x}}^1 v_{\mathbf{y}}^1-q_{\mathbf{x}}^2 v_{\mathbf{y}}^2) + \gamma (q_{\mathbf{x}}^1 [\Delta v^2]_{\mathbf{y}}- q_{\mathbf{x}}^2 [\Delta v^1]_{\mathbf{y}}) \} \\
& = \sum_{\mathbf{x},\mathbf{y} \in \Z_N^d} \{ 2 g^3_{\la,N}(\mathbf{x}-\mathbf{y}) v_{\mathbf{x}}^1 v_{\mathbf{y}}^2 +2 [\Delta g^3_{\la,N}](\mathbf{x}-\mathbf{y}) q_{\mathbf{x}}^1q_{\mathbf{y}}^2 \\
& -B ( q_{\mathbf{x}}^1 v_{\mathbf{y}}^1+q_{\mathbf{x}}^2 v_{\mathbf{y}}^2) +  \gamma [\Delta g^3_{\la,N}](\mathbf{x}-\mathbf{y}) (q_{\mathbf{x}}^1 v_{\mathbf{y}}^2-q_{\mathbf{x}}^2 v_{\mathbf{y}}^1) \}, \\
 L^{(i)} & \big(\sum_{\mathbf{x},\mathbf{y} \in \Z_N^d}g^4_{\la,N}(\mathbf{x}-\mathbf{y})v_{\mathbf{x}}^1 v_{\mathbf{y}}^2 \big)= \sum_{\mathbf{x},\mathbf{y} \in \Z_N^d}g^4_{\la,N}(\mathbf{x}-\mathbf{y}) ( [\Delta q^1]_{\mathbf{x}} v_{\mathbf{y}}^2 + v_{\mathbf{x}}^1[\Delta q^2]_{\mathbf{y}}  \\
&+ B ( v_{\mathbf{x}}^2 v_{\mathbf{y}}^2 - v_{\mathbf{x}}^1 v_{\mathbf{y}}^1) + \gamma ( [\Delta v^1]_{\mathbf{x}} v^2_{\mathbf{y}} + v_{\mathbf{x}}^1 [\Delta v^2]_{\mathbf{y}}) )\\
& = \sum_{\mathbf{x},\mathbf{y} \in \Z_N^d} \{ [\Delta g^4_{\la,N}](\mathbf{x}-\mathbf{y}) (q_{\mathbf{x}}^1 v_{\mathbf{y}}^2-q_{\mathbf{x}}^2 v_{\mathbf{y}}^1) + 2 \gamma [\Delta g^4_{\la,N}](\mathbf{x}-\mathbf{y}) v_{\mathbf{x}}^1v_{\mathbf{y}}^2\}.
\end{align*}
%since $g_{\la,N}(z)=-g_{\la,N}(-z)$. Also, a simple computation shows that
%\begin{align*}
%\ga S u_{\la,N,j} = \ga \sum_{\mathbf{x},\mathbf{y} \in \Z_N^d}g_{\la,N}(\mathbf{x}-\mathbf{y})q_{\mathbf{x}}^j[\Delta v^j](y)  \\
%= \ga \sum_{\mathbf{x},\mathbf{y} \in \Z_N^d} [\Delta g_{\la,N}](\mathbf{x}-\mathbf{y})q_{\mathbf{x}}^j v_{\mathbf{x}}^j(y).
%\end{align*} 
\end{proof}

Finally, we study the case with alternate charges. For this case, we take $d=1$ and $d^*=2$. We introduce 
\[(\bar{\Delta}F)(z)=\sum_{|y-z|=1} (F(y)+F(z))=F(z+1)+2F(z)+F(z-1)
\] 
for $F: \Z_N \to \R$.
Let $h^i_{\la,N}:\Z_N \to \R \  (i=1,2,3,4,5,6)$ be the solution of the simultaneous equations
\begin{align}\label{eq:forg1-6 alt}
\begin{cases}
\la h^1_{\la,N} (z) + 2 \bar{\Delta}h^4_{\la,N} (z)=0  \quad \text{for} \quad z \equiv 0 \mod 2, \\
(\lambda +4\ga) h^2_{\la,N}  (z) - 2h^4_{\la,N}  (z) =0  \quad \text{for} \quad z \equiv 0  \mod 2,  \\
(\lambda -\ga \Delta) h^3_{\la,N}  (z)  + B h^4_{\la,N}  (z) = H(z) \quad \text{for all} \quad z,\\
(\lambda + \ga \bar{\Delta}) h^4_{\la,N}  (z)  - h^1_{\la,N} (z) +2 h^2_{\la,N} (z) -B h^3_{\la,N}  (z) = 0 \quad \text{for} \quad z \equiv 0  \mod 2,  \\
(\lambda + \ga \bar{\Delta}) h^4_{\la,N}  (z)  - h^2_{\la,N} (z-1) - h^2_{\la,N} (z+1) - B h^3_{\la,N}  (z) = 0 \quad \text{for} \quad z \equiv 1  \mod 2, 
\end{cases}
\end{align}
and satisfies $h^i_{\la,N}(z)=-h^i_{\la,N}(-z)$ for $i=1,2,3,4,5,6$. The existence and the uniqueness follows by the Fourier transform again.

\begin{lem}\label{lem:resolvent3}
Define $u_{\la,N}$ by
\begin{align*}
u_{\la,N} & =\sum_{x \equiv y \mod 2}\big(h^1_{\la,N}(x-y) (-1)^y q_x^1 q_y^2+h^2_{\la,N} (x-y) (-1)^y v_x^1 v_y^2) \big) \\
& + \sum_{x,y} \big( h^3_{\la,N} (x-y)(q_x^1 v_y^1+ q_x^2 v_y^2)+ h^4_{\la,N}(x-y)  (-1)^y  (q_x^1 v_y^2- q_x^2 v_y^1)\big) .
\end{align*}
Then,
\[
(\la - L^{(ii)})u_{\la,N}=\sum_{x,y \in \Z_N} \sum_{j=1,2}H(x-y)q_x^j v_y^j.
\]
\end{lem}
\begin{proof}
It follows by the following direct computations. The sum over $x \equiv y$ means the sum over all pairs $(x,y)$ satisfying $x-y \equiv 0 \mod 2$ and the sum over $x \equiv y+1$ means the sum over all pairs $(x,y)$ satisfying $x-y \equiv 1 \mod 2$. 
\begin{align*}
 L^{(ii)} & \big(\sum_{x \equiv y} h^1_{\la,N}(x-y)(-1)^y q_x^1 q_y^2 \big)= (\sum_{x \equiv y} h^1_{\la,N}(x-y)(-1)^y(v_x^1 q_y^2 +q_x^1 v_y^2) \\ 
& =\sum_{x \equiv y}h^1_{\la,N}(x-y)(-1)^y (q_x^1 v_y^2-q_x^2 v_y^1), \\
 L^{(ii)} & \big(\sum_{x \equiv y}h^2_{\la,N}(x-y) (-1)^y v_x^1 v_y^2 \big)= \sum_{x \equiv y}h^2_{\la,N}(x-y) \big( (-1)^y ( [\Delta q^1]_x v_y^2 + v_x^1[\Delta q^2]_y ) \\
&+ B ((-1)^y(-1)^x v_x^2 v_y^2 - (-1)^y (-1)^y v_x^1 v_y^1) + (-1)^y \gamma ( [\Delta v^1]_x v^2_y + v_x^1 [\Delta v^2]_y) \big) \\
& = \sum_{x \equiv y} \big( -2 h^2_{\la,N}(x-y) (-1)^y (q_x^1 v_y^2-q_x^2 v_y^1) \big)\\
& + \sum_{x \equiv y+1 } \big(  (h^2_{\la,N}(x-y+1)+h^2_{\la,N}(x-y-1) ) (-1)^y (q_x^1 v_y^2-q_x^2 v_y^1) \big) \\
& - 4 \gamma \sum_{x \equiv y}   h^2_{\la,N}(x-y) (-1)^y  v_x^1v_y^2,\\
 L^{(ii)} & \big( \sum_{x,y \in \Z_N}h^3_{\la,N}(x-y)(q_x^1 v_y^1+q_x^2 v_y^2) \big)  =  \sum_{x,y \in \Z_N}h^3_{\la,N}(x-y)\big( (v_x^1 v_y^1+v_x^2 v_y^2) + (q_x^1 [\Delta q^1]_y+q_x^2 [\Delta q^2]_y) \\
& + (-1)^y B(q_x^1 v_y^2-q_x^2 v_y^1) + \gamma (q_x^1 [\Delta v^1]_y+q_x^2 [\Delta v^2]_y) \big) \\
& = \sum_{x,y \in \Z_N} \big( h^3_{\la,N}(x-y) (-1)^y B(q_x^1 v_y^2-q_x^2 v_y^1) +  \gamma  [\Delta h^3_{\la,N}](x-y)(q_x^1 v^1_y+q_x^2 v^2_y) \big), \\
 L^{(ii)} & \big(\sum_{x,y \in \Z_N^d}h^4_{\la,N}(x-y) (-1)^y (q_x^1 v_y^2-q_x^2 v_y^1) \big)= \sum_{x,y \in \Z_N^d}h^4_{\la,N}(x-y)  \big( (-1)^y  (v_x^1 v_y^2-v_x^2 v_y^1) \\
& +  (-1)^y  (q_x^1 [\Delta q^2]_y-q_x^2 [\Delta q^1]_y) + B(-q_x^1 v_y^1-q_x^2 v_y^2) + (-1)^y \gamma (q_x^1 [\Delta v^2]_y- q_x^2 [\Delta v^1]_y) \big)\\
& = \sum_{x,y \in \Z_N} \big(  h^4_{\la,N}(x-y) \{(-1)^x+(-1)^y\} v_x^1 v_y^2 - [\bar{\Delta} h^4_{\la,N}](x-y) \{(-1)^x+(-1)^y\} q_x^1q_y^2 \\
& -B ( q_x^1 v_y^1+q_x^2 v_y^2) +  \gamma [\bar{\Delta} h^4_{\la,N}](x-y) (-1)^y (q_x^1 v_y^2-q_x^2 v_y^1) \big) \\
&=  \sum_{x \equiv y } \big( 2 h^4_{\la,N}(x-y)(-1)^y v_x^1 v_y^2 - 2 [\bar{\Delta} h^4_{\la,N}](x-y) (-1)^y q_x^1q_y^2 \big)\\
& +\sum_{x, y } \big( -  B h^4_{\la,N}(x-y)  ( q_x^1 v_y^1+q_x^2 v_y^2) -  \gamma [\bar{\Delta} h^4_{\la,N}](x-y) (-1)^y (q_x^1 v_y^2-q_x^2 v_y^1) \big). 
\end{align*}
%since $g_{\la,N}(z)=-g_{\la,N}(-z)$. Also, a simple computation shows that
%\begin{align*}
%\ga S u_{\la,N,j} = \ga \sum_{x,y \in \Z_N^d}g_{\la,N}(x-y)q_x^j[\Delta v^j](y)  \\
%= \ga \sum_{x,y \in \Z_N^d} [\Delta g_{\la,N}](x-y)q_x^j v_x^j(y).
%\end{align*} 
\end{proof}

\begin{lem}\label{lem:derivative0}
Let $d=1$ and $d^*=2$. Then, for $j=1,2$, we have
\[
\sum_{x=1}^N \partial_{q_x^j}F  =0
\]
if $F=u_{\la,N,1}$ or $u_{\la,N,2}$ given in Lemma \ref{lem:resolvent1}, or $F=u_{\la,N}$ given in Lemma \ref{lem:resolvent2} or \ref{lem:resolvent3}.
\end{lem}
\begin{proof}
If $F=u_{\la,N,1}$ or $u_{\la,N,2}$ given in Lemma \ref{lem:resolvent1}, then
\[
\sum_{x=1}^N \partial_{q_x^j}F  = \sum_{x,y \in \Z_N}g_{\la,N}(x-y) v_{y}^j=0
\]
since $\sum_{x}g_{\la,N}(x) =0$ for $j=1,2$.

If $F=u_{\la,N}$ given in Lemma \ref{lem:resolvent2},
\[
\sum_{x=1}^N \partial_{q_x^1}F  =\sum_{x,y \in \Z_N} \big ( g^1_{\la,N} (x-y) q_{y}^2+g^2_{\la,N} (x-y) v_{y}^1 +g^3_{\la,N}(x-y)v_{y}^2\big)=0
\]
since $\sum_{x}g^1_{\la,N}(x) =\sum_{x}g^2_{\la,N}(x)=\sum_{x}g^3_{\la,N}(x)=0$. For $j=2$, we can apply the same argument.

If $F=u_{\la,N}$ given in Lemma \ref{lem:resolvent3},
\[
\sum_{x=1}^N \partial_{q_x^1}F  =\sum_{x \equiv y \mod 2}h^1_{\la,N}(x-y) (-1)^y  q_y^2  
 + \sum_{x,y} \big( h^3_{\la,N} (x-y) v_y^1+ h^4_{\la,N}(x-y)  (-1)^y  v_y^2\big) =0
\]
since $\sum_{x:even}h^1_{\la,N}(x)=0$, $\sum_{x}h^3_{\la,N} (x)=\sum_{x}h^4_{\la,N} (x)=0$. For $j=2$, we can apply the same argument.
\end{proof}

\section{Equivalence of ensembles}\label{app:equivalence}

We list the consequence of the equivalence of ensembles used in Section 2. 
\begin{lem}\label{equivalence}
For $j=1,2,\dots, d^*$, the followings hold:\\
i) $E_{N,E}[(v_{\mathbf{0}}^j)^2] =\frac{E}{d^*}$ for any $N$, \\
ii) $E_{N,E}[(v_{\mathbf{0}}^j)^4] \to \frac{3E^2}{(d^*)^2}$ as $N \to \infty$, \\
iii) For any $\mathbf{x} \in \Z^d \setminus  \{\mathbf{0}\}$, $E_{N,E}[(v_{\mathbf{0}}^j)^2(v_{\mathbf{x}}^j)^2] \to \frac{E^2}{(d^*)^2}$  as $N \to \infty$, \\
iv) For any $\mathbf{x} \in \Z_N^d$, 
\[
E_{N,E}[(q_{\mathbf{x}}^j-q_{-\mathbf{x}}^j)(q_{\mathbf{e}_1}^j-q_{-\mathbf{e}_1}^j)(v_{\mathbf{0}}^j)^2]=\frac{1}{N^{d}} \left(\frac{E^2}{{d^*}^2}+O(N^{-d})\right) \sum_{\mathbf{\xi} \in \Z^d_N, \mathbf{\xi}\neq \mathbf{0}}\frac{\sin(2\pi  \frac{\mathbf{\xi}}{N} \cdot \mathbf{x})\sin(2\pi  \frac{\mathbf{\xi}}{N} \cdot \mathbf{e}_1) }{\sum_{k=1}^d\sin(\frac{\pi \mathbf{\xi}^k}{N})^2}.
\] 
Moreover, the term $O(N^{-d})$ does not depend on $\mathbf{x}$.  
%\end{itemize}
\end{lem}
This is essentially done in Lemma 7 of \cite{BBO}, but for completeness, we give a proof.
\begin{proof}
Define 
\[
\tilde{\mathbf{q}}(\mathbf{\xi})=(1-\delta(\mathbf{\xi}))\omega^N(\mathbf{\xi})\hat{\mathbf{q}}(\mathbf{\xi}), \quad \tilde{\mathbf{v}}(\mathbf{\xi})=N^{-d/2}(1-\delta(\mathbf{\xi}))\hat{\mathbf{v}}(\mathbf{\xi})
\]
where $\omega^N(\mathbf{\xi})=2N^{-d/2}\sqrt{\sum_{a=1}^d\sin^2(\frac{\pi\xi^a}{N})}$ and $\hat{\mathbf{q}},\hat{\mathbf{v}}$ are the Fourier transform of $\mathbf{q}$ and $\mathbf{v}$. The factor $1-\delta$ in the definition is due to the condition $\sum_{\mathbf{x}}\mathbf{q}_{\mathbf{x}}=\sum_{\mathbf{x}}\mathbf{v}_{\mathbf{x}}=\mathbf{0}$ assumed in the microcanonical state. The total energy is written as
\begin{align*}
\sum_{\mathbf{x}}\mathcal{E}_{\mathbf{x}}=\frac{1}{2}\sum_{\mathbf{\xi} \neq \mathbf{0}}\{ |\tilde{\mathbf{q}}(\mathbf{\xi})|^2+|\tilde{\mathbf{v}}(\mathbf{\xi})|^2\} 
=\frac{1}{2}\sum_{\mathbf{\xi} \neq \mathbf{0}}\{ \Re^2(\tilde{\mathbf{q}}(\mathbf{\xi}))+\Im^2(\tilde{\mathbf{q}}(\mathbf{\xi}))+ \Re^2(\tilde{\mathbf{v}}(\mathbf{\xi}))+\Im^2(\tilde{\mathbf{v}}(\mathbf{\xi}))\}.
\end{align*}
On $\Z^d_N \setminus \{\mathbf{0}\}$, we define an equivalence relation $\mathbf{\xi} \sim \mathbf{\xi}'$ if and only if $\mathbf{\xi} =-\mathbf{\xi}'$. Denote the class of representatives for $\sim$ by $\mathcal{U}_N^d$. If $\mathbf{\xi} \sim \mathbf{\xi}'$, then we have
\begin{align*}
\Re(\tilde{\mathbf{q}}(\mathbf{\xi}))&=\Re(\tilde{\mathbf{q}}(\mathbf{\xi}')), \quad \Re(\tilde{\mathbf{v}}(\mathbf{\xi}))=\Re(\tilde{\mathbf{v}}(\mathbf{\xi}')) \\
\Im(\tilde{\mathbf{q}}(\mathbf{\xi}))&=-\Im(\tilde{\mathbf{q}}(\mathbf{\xi}')), \quad \Im(\tilde{\mathbf{v}}(\mathbf{\xi}))=-\Im(\tilde{\mathbf{v}}(\mathbf{\xi}')).
\end{align*}
Therefore, if $N$ is odd
\[
\sum_{\mathbf{x}}\mathcal{E}_{\mathbf{x}}=\sum_{\mathbf{\xi} \in \mathcal{U}_N^d}\{ \Re^2(\tilde{\mathbf{q}}(\mathbf{\xi}))+\Im^2(\tilde{\mathbf{q}}(\mathbf{\xi}))+ \Re^2(\tilde{\mathbf{v}}(\mathbf{\xi}))+\Im^2(\tilde{\mathbf{v}}(\mathbf{\xi}))\}
\]
and if $N$ is even, 
\[
\sum_{\mathbf{x}}\mathcal{E}_{\mathbf{x}}=\sum_{\mathbf{\xi} \in \mathcal{U}_N^d, \mathbf{\xi} \neq \mathbf{\xi}^N }\{ \Re^2(\tilde{\mathbf{q}}(\mathbf{\xi}))+\Im^2(\tilde{\mathbf{q}}(\mathbf{\xi}))+ \Re^2(\tilde{\mathbf{v}}(\mathbf{\xi}))+\Im^2(\tilde{\mathbf{v}}(\mathbf{\xi}))\} + \frac{1}{2}\{ \Re^2(\tilde{\mathbf{q}}(\mathbf{\xi}^N))+ \Re^2(\tilde{\mathbf{v}}(\mathbf{\xi}^N))\}
\]
where $\mathbf{\xi}^N=(\frac{N}{2},\frac{N}{2},\dots,\frac{N}{2})$.

Note that the cardinality of $\mathcal{U}_N^d$ is $\frac{N^d-1}{2}$ if $N$ is odd and $\frac{N^d}{2}$ if $N$ is even. If $N$ is odd, under the measure $\mu_{N,E}$, the random variables 
\[
(\Re(\tilde{\mathbf{q}}(\mathbf{\xi})), \Im(\tilde{\mathbf{q}}(\mathbf{\xi})),\Re(\tilde{\mathbf{v}}(\mathbf{\xi})), \Im(\tilde{\mathbf{v}}(\mathbf{\xi}))_{\mathbf{\xi} \in \mathcal{U}_N^d}
\]
are distributed according to the uniform measure on $2d^*(N^d-1)$-dimensional sphere of radius $\sqrt{N^dE}$. If $N$ is even, under the measure $\mu_{N,E}$, the random variables 
\[
(\Re(\tilde{\mathbf{q}}(\mathbf{\xi})), \Im(\tilde{\mathbf{q}}(\mathbf{\xi})),\Re(\tilde{\mathbf{v}}(\mathbf{\xi})), \Im(\tilde{\mathbf{v}}(\mathbf{\xi}))_{\mathbf{\xi} \in \mathcal{U}_N^d, \mathbf{\xi} \neq \mathbf{\xi}^N}, \big(\frac{\Re(\tilde{\mathbf{q}}(\mathbf{\xi}^N))}{\sqrt{2}}, \frac{\Re(\tilde{\mathbf{v}}(\mathbf{\xi}^N))}{\sqrt{2}}\big)
\]
are distributed according to the uniform measure on $2d^*(N^d-1)$-dimensional sphere of radius $\sqrt{N^dE}$.

Then, we can conclude (i) since if $N$ is odd,
\begin{align*}
E_{N,E}[(v_{\mathbf{0}}^j)^2]& =E_{N,E}[( \frac{1}{N^d}\sum_{\mathbf{\xi} \in \Z_N^d}\hat{v}^j(\mathbf{\xi}))^2] = \frac{1}{N^d} E_{N,E}[(\sum_{\mathbf{\xi} \in \Z_N^d, \mathbf{\xi} \neq \mathbf{0}}\tilde{v}^j(\mathbf{\xi}))^2] \\
& = \frac{4}{N^d}\sum_{\mathbf{\xi} \in \mathcal{U}_N^d}  E_{N,E}[\Re(\tilde{v}^j(\mathbf{\xi}))^2] = \frac{4}{N^d}\frac{N^d-1}{2}\frac{N^dE}{2d^*(N^d-1)}=\frac{E}{d^*}.
\end{align*}
We can also do a similar computation and conclude the same result if $N$ is even. 

For (ii), if $N$ is odd, we have
\begin{align*}
& E_{N,E}[(v_{\mathbf{0}}^j)^4] =E_{N,E}[( \frac{1}{N^d}\sum_{\mathbf{\xi} \in \Z_N^d}\hat{v}^j(\mathbf{\xi}))^4] = \frac{1}{N^{2d}} E_{N,E}[(\sum_{\mathbf{\xi} \in \Z_N^d, \mathbf{\xi} \neq \mathbf{0} }\tilde{v}^j(\mathbf{\xi}))^4] \\
& = \frac{16}{N^{2d}}\sum_{\mathbf{\xi} \in \mathcal{U}_N^d}  E_{N,E}[\Re(\tilde{v}^j(\mathbf{\xi}))^4]+\frac{48}{N^{2d}}\sum_{\mathbf{\xi} \neq \mathbf{\xi}' \in \mathcal{U}_N^d}  E_{N,E}[\Re(\tilde{v}^j(\mathbf{\xi}))^2 \Re(\tilde{v}^j(\mathbf{\xi}'))^2]\\
& = \frac{16}{N^{2d}}\frac{N^d-1}{2} E_{N,E}[\Re(\tilde{v}^j(\mathbf{\mathbf{e}_1}))^4]+\frac{48}{N^{2d}}\frac{(N^d-1)(N^d-3)}{4} E_{N,E}[\Re(\tilde{v}^j(\mathbf{e}_1))^2 \Re(\tilde{v}^j(2\mathbf{e}_1))^2].
\end{align*}
Then, we can apply the classical equivalence of ensembles under the uniform measure on the sphere to conclude (ii).
Similar arguments also work for the case that $N$ is even, and also for (iii).

Finally, we prove (iv). We only deal with the case that $N$ is odd, since the other case can be shown in the same way. By definition,
\begin{align*}
q_{\mathbf{x}}^j - q_{-\mathbf{x}}^j = \frac{2i}{N^d}\sum_{\mathbf{\xi} \in \Z^d_N } \sin(2\pi  \frac{\mathbf{\xi}}{N} \cdot \mathbf{x}) \hat{q}^j(\mathbf{\xi}), \quad v_{\mathbf{0}}^j=\frac{1}{N^d}\sum_{\mathbf{\xi} \in \Z^d_N } \hat{v}^j(\mathbf{\xi})
\end{align*}
and so 
\begin{align*}
& E_{N,E}[(q_{\mathbf{x}}^j - q_{-\mathbf{x}}^j)(q_{\mathbf{e}_1}^j-q_{-\mathbf{e}_1}^j)(v_{\mathbf{0}}^j)^2] \\
& = \frac{-4}{N^{4d}}\sum_{\mathbf{\xi},\mathbf{\xi}',\mathbf{\zeta},\mathbf{\zeta}' \in \Z^d_N }\sin(2\pi  \frac{\mathbf{\xi}}{N} \cdot \mathbf{x})\sin(2\pi  \frac{\mathbf{\xi}'}{N} \cdot \mathbf{e}_1) E_{N,E}[\hat{q}^j(\mathbf{\xi})\hat{q}^j(\mathbf{\xi}')\hat{v}^j(\mathbf{\zeta})\hat{v}^j(\mathbf{\zeta}')] \\
& = \frac{-4}{N^{4d}}\sum_{\mathbf{\xi},\mathbf{\xi}',\mathbf{\zeta},\mathbf{\zeta}' \in \Z^d_N, \mathbf{\xi},\mathbf{\xi}',\mathbf{\zeta},\mathbf{\zeta}' \neq \mathbf{0}}\sin(2\pi  \frac{\mathbf{\xi}}{N} \cdot \mathbf{x})\sin(2\pi  \frac{\mathbf{\xi}'}{N} \cdot \mathbf{e}_1) E_{N,E}[\hat{q}^j(\mathbf{\xi})\hat{q}^j(\mathbf{\xi}')\hat{v}^j(\mathbf{\zeta})\hat{v}^j(\mathbf{\zeta}')].
\end{align*}
The last term is equal to 
\begin{align*}
\frac{-4}{N^{3d}}\sum_{\mathbf{\xi},\mathbf{\xi}',\mathbf{\zeta},\mathbf{\zeta}' \in \Z^d_N, \mathbf{\xi},\mathbf{\xi}',\mathbf{\zeta},\mathbf{\zeta}' \neq \mathbf{0} }\frac{\sin(2\pi  \frac{\mathbf{\xi}}{N} \cdot \mathbf{x})\sin(2\pi  \frac{\mathbf{\xi}'}{N} \cdot \mathbf{e}_1) }{\omega^N(\mathbf{\xi})\omega^N(\mathbf{\xi}')}E_{N,E}[\tilde{q}^j(\mathbf{\xi})\tilde{q}^j(\mathbf{\xi}')\tilde{v}^j(\mathbf{\zeta})\tilde{v}^j(\mathbf{\zeta}')].
\end{align*}

From the explicit distribution of $(\tilde{\mathbf{q}}, \tilde{\mathbf{v}})$ studied above, if $\mathbf{\xi} \nsim \mathbf{\xi}'$, then 
\[E_{N,E}[\tilde{q}^j(\mathbf{\xi})\tilde{q}^j(\mathbf{\xi}')\tilde{v}^j(\mathbf{\zeta})\tilde{v}^j(\mathbf{\zeta}')]=-E_{N,E}[\tilde{q}^j(\mathbf{\xi})\tilde{q}^j(\mathbf{\xi}')\tilde{v}^j(\mathbf{\zeta})\tilde{v}^j(\mathbf{\zeta}')]=0.
\] 
The same is true if $\mathbf{\zeta} \nsim \mathbf{\zeta}'$. 
Also, if $\mathbf{\xi} = \mathbf{\xi}'$, then
\begin{align*}
& E_{N,E}[\tilde{q}^j(\mathbf{\xi})\tilde{q}^j(\mathbf{\xi}')\tilde{v}^j(\mathbf{\zeta})\tilde{v}^j(\mathbf{\zeta}')] \\
& =E_{N,E}[\big(\Re(\tilde{q}^j(\mathbf{\xi}))+i \Im(\tilde{q}^j(\mathbf{\xi}))\big)^2 \tilde{v}^j(\mathbf{\zeta})\tilde{v}^j(\mathbf{\zeta}')] \\
& = E_{N,E}[(\Re(\tilde{q}^j(\mathbf{\xi}))^2- \Im(\tilde{q}^j(\mathbf{\xi}))^2) \tilde{v}^j(\mathbf{\zeta})\tilde{v}^j(\mathbf{\zeta}')] +2i E_{N,E}[\Re(\tilde{q}^j(\mathbf{\xi})) \Im(\tilde{q}^j(\mathbf{\xi})) \tilde{v}^j(\mathbf{\zeta})\tilde{v}^j(\mathbf{\zeta}')]=0.
\end{align*}
The same is true if $\mathbf{\zeta} = \mathbf{\zeta}'$. Therefore, the term $E_{N,E}[\tilde{q}^j(\mathbf{\xi})\tilde{q}^j(\mathbf{\xi}')\tilde{v}^j(\mathbf{\zeta})\tilde{v}^j(\mathbf{\zeta}')] \neq 0$  if and only if $\mathbf{\xi}=-\mathbf{\xi}'$ and $\mathbf{\zeta}=-\mathbf{\zeta}'$. Moreover, \begin{align*}
E_{N,E}[\tilde{q}^j(\mathbf{\xi})\tilde{q}^j(\mathbf{-\xi})\tilde{v}^j(\mathbf{\zeta})\tilde{v}^j(\mathbf{-\zeta})] & =E_{N,E}[\big(\Re(\tilde{q}^j(\mathbf{\xi}))^2+ \Im(\tilde{q}^j(\mathbf{\xi}))^2 \big) \big(\Re(\tilde{v}^j(\mathbf{\zeta}))^2+ \Im(\tilde{v}^j(\mathbf{\zeta}))^2 \big)]  \\
&=4E_{N,E}[\Re(\tilde{q}^j(\mathbf{e}_1))^2\Re(\tilde{v}^j(\mathbf{e}_1))^2]
\end{align*}
 does not depend on $\xi, \zeta$. From the explicit distribution
 \[
4E_{N,E}[\Re(\tilde{q}^j(\mathbf{e}_1))^2\Re(\tilde{v}^j(\mathbf{e}_1))^2]=\frac{E^2}{{d^*}^2}+O(N^{-d}).\]
 Therefore,
\begin{align*}
& E_{N,E}[(q_{\mathbf{x}}^1 - q_{-\mathbf{x}}^1)(q_{\mathbf{e}_1}^1-q_{-\mathbf{e}_1}^1)(v_{\mathbf{0}}^1)^2] \\
& = \frac{4}{N^{3d}} \left(\frac{E^2}{{d^*}^2}+O(N^{-d})\right) \sum_{\mathbf{\xi},\mathbf{\zeta} \in \Z^d_N, \mathbf{\xi},\mathbf{\zeta} \neq \mathbf{0} }\frac{\sin(2\pi  \frac{\mathbf{\xi}}{N} \cdot \mathbf{x})\sin(2\pi  \frac{\mathbf{\xi}}{N} \cdot \mathbf{e}_1) }{\omega^N(\mathbf{\xi})\omega^N(\mathbf{\xi})} \\
&= \frac{1}{N^{d}} \left(\frac{E^2}{{d^*}^2}+O(N^{-d})\right) \sum_{\mathbf{\xi} \in \Z^d_N, \mathbf{\xi}\neq \mathbf{0}}\frac{\sin(2\pi  \frac{\mathbf{\xi}}{N} \cdot \mathbf{x})\sin(2\pi  \frac{\mathbf{\xi}}{N} \cdot \mathbf{e}_1) }{\sum_{a=1}^d\sin(\frac{\pi \mathbf{\xi}^a}{N})^2}.
\end{align*}
\end{proof}

\end{document}